\documentclass[12pt,reqno,a4paper]{amsart}
\usepackage[margin=1in]{geometry}

\usepackage[all]{xy}

\usepackage{ amssymb, amsmath, enumerate, amsfonts, amsthm, mathrsfs, url, bm, mathtools}

\numberwithin{equation}{section}

\usepackage{xcolor}  	

\usepackage[backref]{hyperref}
\hypersetup{
    colorlinks=true,
    linkcolor={blue!60!black},
    citecolor={blue!60!black},
    urlcolor={red!60!black}
}

\usepackage{color}

\RequirePackage{doi}

\usepackage[square,sort,comma,numbers]{natbib}
\makeatletter
\@namedef{subjclassname@2020}{%
  \textup{2020} Mathematics Subject Classification}
\makeatother

 \newcommand{\set}[1]{\left\{#1\right\}}
\newcommand{\bigset}[1]{\bigl\{ #1 \bigr\}}

\newcommand{\bigabs}[1]{\bigl| #1 \bigr|}
\newcommand{\Bigabs}[1]{\Bigl| #1 \Bigr|}
\newcommand{\biggabs}[1]{\biggl| #1 \biggr|}

\newcommand{\bigbrac}[1]{\bigl( #1 \bigr)}
\newcommand{\Bigbrac}[1]{\Bigl( #1 \Bigr)}

\newcommand{\norm}[1]{\left\| #1\right\|}

\newcommand{\rd}{\,\mathrm{d}}

\newcommand{\twosum}[2]{ \sum_{\substack{#1\\ #2}}}

\newcommand{\Z}{\mathbb{Z}}

\newcommand{\R}{\mathbb{R}}

\newcommand{\T}{\mathbb{T}}

\newcommand{\eps}{\varepsilon}

\makeatletter
\let\@@pmod\pmod
\DeclareRobustCommand{\pmod}{\@ifstar\@pmods\@@pmod}
\def\@pmods#1{\mkern4mu({\operator@font mod}\mkern 6mu#1)}
\makeatother

\renewcommand{\mod}[1]{\,(\mathrm{mod}\,#1)}

\renewcommand{\leq}{\leqslant}
\renewcommand{\geq}{\geqslant}
\renewcommand{\epsilon}{\varepsilon}
\newcommand{\major}{\mathfrak M}
\newcommand{\minor}{\mathfrak m}

\usepackage[normalem]{ulem}

\newtheorem{theorem}{Theorem}[section]

\newtheorem{corollary}[theorem]{Corollary}

\newtheorem{proposition}[theorem]{Proposition}
\newtheorem{lemma}[theorem]{Lemma}

\theoremstyle{definition}

\numberwithin{theorem}{section}

\begin{document}

\title[Discorrelation in short intervals]{Discorrelation of the M\"obius function with linear phases in short intervals}

\author{Javier Pliego}

\address{Departamento de Matem\'aticas, Universidad Aut\'onoma de Madrid, 28049 Madrid, Spain}

\email{javier.pliego@uam.es}

\author{Mengdi Wang}
\address{\'{E}cole Polytechnique F\'{e}d\'{e}rale de Lausanne (EPFL), Lausanne, Switzerland}
\email{mengdi.wang@epfl.ch}

\maketitle

\begin{abstract}
\backrefsetup{disable}

We introduce a new approach to studying correlations between arithmetic functions and linear phases with respect to highly irrational frequencies over short intervals. As an application, we prove that
\[
\Big|\sum_{x<n \leq x+H}\mu(n) e(n\alpha)\Big| \ll H(\log x)^{-1/3+\varepsilon }
\]
whenever $x^{0.58+\varepsilon }\leq H\leq x$.  In the same range of $H$ we also show that
\[
\Big|\sum_{x<n \leq x+H} \Lambda (n) e(n\alpha)\Big| \ll H(\log x)^{-A}
\]
for arbitrarily large $A>0$ unless there exists an integer $1\leq q\ll (\log x)^{O_A(1)}$ such that $\norm{q\alpha} \ll x  (\log x)^{O_A(1)}/H^2$.	This breaks the $3/5$ barrier in previous work of Matom\"aki--Ter\"av\"ainen \cite[Theorem 1.5]{MT} and of Zhan \cite[Theorem 2]{Zhan}.

Moreover, by our method, any improvement in large value estimates for character-twisted Dirichlet polynomials on the critical range would lead to a corresponding improvement in the  lower bound for $H$.

\backrefsetup{enable}
\end{abstract}


\section{Introduction}\label{introduction}

In this paper, we study correlations of arithmetic functions with the linear phases $e(n\alpha)$ over short intervals, uniformly for $\alpha\in[0,1)$, with the aim of extending existing estimates to shorter intervals. We focus on the von Mangoldt function $\Lambda$ and the M\"obius function $\mu$, although our methods should also apply to divisor functions and possibly other arithmetic functions of interest. Our first result is as follows.

\begin{theorem}\label{thm:mobius}
Suppose that $x^{0.58+\eps }\leq H\leq x$. Then, uniformly for all $\alpha \in \R $, one has
\[
\Bigabs{\sum_{x<n \leq x+H}\mu(n) e(n\alpha)} \ll H(\log x)^{-1/3+\eps }.
\] 	
\end{theorem}

Previously, Zhan~\cite[Theorem 5]{Zhan} showed that this exponential sum exhibits arbitrarily large logarithmic decay when $H\geq x^{5/8+\eps}$. Matom\"aki and Ter\"av\"ainen~\cite[Theorem 1.5]{MT} subsequently extended the range to $H\geq x^{3/5+\eps}$, at the cost of weakening the error term to $O(H(\log x)^{-1/3+\eps})$, as in our Theorem~\ref{thm:mobius}. Their argument exploits the multiplicativity of the M\"obius function but still follows Zhan's approach for highly irrational frequencies. Very recently, in joint work with Matom\"aki, Pandey, and Ter\"av\"ainen~\cite[Corollary 1.3]{MPPTW}, we obtained arbitrarily large power of $\log x$ savings throughout the range $H\geq x^{3/5+\eps}$, using Guth and Maynard's  new large value estimates in \cite{GM}.

The main new idea of this paper concerns highly irrational frequencies, and also applies in the context of the von Mangoldt function.  

\begin{theorem}\label{thm:mangoldt}
Suppose that $x^{0.58+\eps }\leq H\leq x$ and that $A\geq 1$ is a large constant. Then one has
\[
\Bigabs{\sum_{x<n \leq x+H} \Lambda (n) e(n\alpha)} \leq H(\log x)^{-A}
\] and 
\begin{equation}\label{ec1.1}
\Bigabs{\sum_{x<n \leq x+H} \mu (n) e(n\alpha)} \leq H(\log x)^{-A},
\end{equation}
unless there exists an integer $1\leq q\ll (\log x)^{O_A(1)}$ such that
\[
\norm{q\alpha} \ll \frac{x(\log x)^{O_A(1)}}{H^2}.
\]	
\end{theorem}

This improves on \cite[Theorem 2]{Zhan} from 1991, wherein similar consequences were derived for intervals of length $H\geq x^{3/5}(\log x)^{O_{A}(1)}$. An inspection of the statements of both Theorem~\ref{thm:mangoldt} and ~\cite[Corollary 1.3]{MPPTW} reveals that the main barrier to reducing $H$ in the study of $\sum_{x<n\leq x+H}\Lambda(n)e(n\alpha)$ occurs when $\norm{q\alpha}\approx x/H^2$ with $q\leq(\log x)^{O_{A}(1)}$. This  observation motivates the following conditional result, which improves \cite[Theorem 4]{Zhan}.

\begin{theorem}\label{thm1.3}
	Let $T\geq 1$ and $q'\leq (\log T)^{O(1)}$. Assume that
	\begin{equation}\label{eq-moment}
	\sum_{\chi \mod {q'}} \int_T^{T+T^{25/42}} \Bigabs{L\Big(\frac12+it,\chi\Big)}^{4} \rd t\ll T^{25/42+\eps }.
	\end{equation}
	Suppose that $x^{0.58+\eps }\leq H\leq x$. Then, for any arbitrarily large constant $A>1$ one has
	\[
\Bigabs{\sum_{x<n \leq x+H} \Lambda (n) e(n\alpha)} \leq H(\log x)^{-A},
\]
unless there exists an integer $1\leq q\ll (\log x)^{O_A(1)}$ such that
\[
\norm{q\alpha} \ll (\log x)^{O_A(1)}/H.
\]	
\end{theorem}
Indeed, \cite[Theorem 4]{Zhan}  obtained an analogous estimate for intervals of length $H\geq x^{3/5}(\log x)^{O_{A}(1)}$, with $T^{9/16}$ in place of $T^{25/42}$ in the assumption (\ref{eq-moment}).
This theorem allows one to deduce the following consequence concerning the validity of the asymptotic formula for the number of representations of odd numbers as  sums of three primes lying in short intervals. To such an end we introduce, for odd $N\in\mathbb{N}$, the singular series
$$\mathfrak{S}(N)=\prod_{p|N}\Big(1-\frac{1}{(p-1)^{2}}\Big)\prod_{p\nmid N}\Big(1-\frac{1}{(p-1)^{3}}\Big).$$ We also denote by $\mathcal R(N,H)$ to the number of solutions of
$$N=p_{1}+p_{2}+p_{3},\ \ \ \ \ \ \ N/3-H\leq p_{i}\leq N/3+H,\ \ \ \ 1\leq i\leq 3,  $$ where $p_{i}$ are prime numbers.
\begin{corollary}\label{cor}
Let $N$ be a large odd integer and suppose that (\ref{eq-moment}) holds. Then, whenever $ N^{0.58+\varepsilon}\leq H\leq N^{1-\varepsilon}$ one has
$$\mathcal R(N,H)=3\mathfrak{S}(N)H^{2}(\log N)^{-3}+O(H^{2}(\log N)^{-4}).   $$
\end{corollary}

Assuming Theorem~\ref{thm1.3}, Corollary~\ref{cor} follows directly via an application of the circle method. We therefore omit the proof and refer the interested reader to the proof of Theorem~1.1 in \cite{Zhan} for details.

\subsection{Background and sketch outline}

In this subsection, we explain why the $3/5$ barrier arises in previous approaches, with the aim of clarifying how we arrived at our approach to breaking it. We then conclude by outlining the general strategy for proving the theorems.

By Taylor expansion, if $|\beta|\leq x/(H^2\log x)$,  one has
\begin{align*}
e(-x\beta)e(n\beta)&=e\bigbrac{x\beta(\log n-\log x)}+O\Bigbrac{\frac{|n-x|^2|\beta|}{x}}\\
&=x^{-2\pi i x\beta}n^{2\pi i x\beta}+O((\log x)^{-1})
\end{align*}
whenever $|n-x|\leq H$. Thus, writing $\alpha=a/q+\beta$, when $|q\beta|\leq x/(H^2\log x)$ and $q\leq(\log x)^{O(1)}$, we morally have
\[
e(n\beta)\approx c\,n^{2\pi i x\beta},
\]
where $c=e(x\beta)x^{-2\pi i x\beta}$ has modulus one. Consequently,
\begin{align}\label{eq:intro-beta-small}
\Big|\sum_{x<n\leq x+H}\mu(n)e(n\alpha)\bigg|
\leq
\Big|\sum_{x<n\leq x+H}\mu(n)e(an/q)n^{2\pi i x\beta}\bigg|
+O\bigbrac{H(\log x)^{-1}}.
\end{align}
Since $q$ is small and $n^{2\pi i x\beta}$ is multiplicative, methods for studying the M\"obius function in short  arithmetic progressions with small moduli should also apply in this setting.

Therefore, in what follows, we assume that $\alpha$ is highly irrational, in the sense that $\norm{q\alpha}\gg x(\log x)^{O(1)}/H^2$ for every $1\leq q\leq(\log x)^{O(1)}$.

The classical approach to handling highly irrational frequencies goes back to Vinogradov and relies on what are now known as the \emph{Type I} and \emph{Type II} estimates for exponential sums. To illustrate this approach, we apply Heath-Brown's identity \cite{HB} to the M\"obius function, followed by a suitable dyadic decomposition. This reduces the problem to estimating sums of the form

\begin{equation}\label{h2}
\twosum{x<m_1\cdots m_l\leq x+H}{M_i<m_i\leq 2M_i}a_1(m_1)\cdots a_l(m_l) e( m_1\cdots m_l \alpha),
\end{equation}
where $l\geq 1$, $M_1\cdots M_l\asymp x$, $a_i\in \set{1,\mu}$, and $M_i\leq x^{1/l}$ for each $a_i=\mu$. Depending on the sizes of the $M_i$, one seeks either a Type I estimate (if there exists some $a_i=1$ such that $M_i\geq x/H$) or a Type II estimate (obtained by grouping the variables into two suitable blocks).  The work in \cite{MSTT} also introduced \emph{Type I$_2$} estimates by exploring two-dimensional arithmetic progressions.

To describe the resulting combinatorial alternatives which arise in their argument, write $M_i=x^{\gamma_i}$ and $H=x^\theta$. For $\theta>3/5$, the methods developed in \cite[Section 4]{MSTT} (see also \cite[Proposition 6.6]{MPPTW}) can be used to control the following alternatives:
\begin{enumerate}
	\item There exists some $\gamma_i>1-\theta$;
	\item There exist two distinct $\gamma_i\neq \gamma_j$ such that $\gamma_i+\gamma_j>c(1-\theta)$ with $c=\frac32$;
	\item There exists a disjoint partition $I\sqcup J=\set{1,\dots,l}$ such that 
	 \[
	 1-\theta \leq  \sum_{i\in I}\gamma_i,\sum_{j\in J}\gamma_j \leq \theta.
	\] 
\end{enumerate}
We emphasize that the approach of Matom\"aki, Shao, Tao, and Ter\"av\"ainen~\cite{MSTT} applies more generally to nilsequences of finite complexity. For linear phases, similar alternatives, with $c=4/3$ replacing $3/2$ in (2), had already been introduced and analysed by Zhan~\cite{Zhan}. In the spirt of \eqref{eq:intro-beta-small}, one can approximate  $\alpha\in\T$ by a rational number $a/q$ with denominator $q\leq H^2/(x(\log x)^{O(1)})$ and tail $\lvert \beta\rvert\leq x(\log x)^{O(1)}/H^2$, and then use Dirichlet characters to represent the periodic factor $e(an/q)$. After reducing the target function to a product $a_1(m_1)\cdots a_l(m_l)$ as in \eqref{h2}, a standard application of Perron's formula reduces the problem to estimating
\begin{align}\label{eq:intro-dirichlet-poly}
\sum_{\chi \mod q}\sup_{|T| \leq |\beta|x} \int_{T}^{T+\frac xH} \Bigabs{\twosum{x<m_1\cdots m_l\leq 2x}{M_i<m_i\leq 2M_i}\frac{a_1(m_1)\chi(m_1)\cdots a_l(m_l)\chi(m_l)}{(m_1\cdots m_l)^{\frac12+it}}} \rd t.
\end{align}
Zhan obtained the Type I$_2$ estimates (case 2) using the fourth-moment bound for Dirichlet $L$-functions in short intervals 
\begin{align}\label{eq:intro-fourth-moment}
\int_T^{T+T_0} \Bigabs{L\Big(\frac12+it,\chi\Big)}^4 \rd t \ll (qT)^{\eps}(qT_0+qT^{\frac23 }) \qquad \text{ when } T^{1/2}\leq T_0\leq T^{2/3}.
\end{align}

Since the exponents $0\leq \gamma_j\leq 1$ with $\sum_i\gamma_i=1$ are arbitrary, one could in particular take the values $(\gamma_1,\cdots,\gamma_l)=(\frac15,\frac15,\frac15, \frac15,\frac15, 0,\cdots,0)$. However, this configuration satisfies none of these alternatives when $\theta<3/5$ and therefore presents a barrier to applying either approach to intervals shorter than $x^{3/5}$.

Our idea is to introduce a new alternative, which we call the \emph{Type III case}, by partitioning $\set{1,\dots,l}$ into three disjoint sets and applying large value estimates for Dirichlet polynomials twisted by characters, allowing us to handle the configuration $(\gamma_1,\cdots,\gamma_l)=(\frac15,\frac15,\frac15, \frac15,\frac15, 0,\cdots,0)$. However, several obstacles arise. First, the term $qT^{2/3+\eps}$ in~\eqref{eq:intro-fourth-moment} is too large when $T_0$ is close to $T^{1/2+\eps}$ in the Type I$_2$ analysis (case 2) if one ought to sharpen the exponent $\theta=3/5$. This can be circumvented by sharpening the dependence on $q$ combining the spectral theoretic approach employed in \cite[Theorem 4]{JM} with an admissible exponent towards the Ramanujan--Petersson conjecture.

Moreover, the range $q\leq H^2/(x(\log x)^{O(1)})$ in~\eqref{eq:intro-dirichlet-poly} turns out to be too large for the treatment of Type~III sums, since applying existing large-value estimates for Dirichlet characters in this range of $q$ also prevents us from handling intervals of length shorter than $x^{3/5}$. This motivates us to develop a  flexible stationary-phase argument that accommodates a wider range of $|\beta|$. We achieve this by observing that when $|\beta|\leq x/H^2$ the stationary points lie in an interval of length $x/H$; while when $x/H^2\leq |\beta| \leq 1$ the stationary points occur over a longer interval of length $O(|\beta|H)$.

We should also remark that, in our analysis of each of the four alternatives, different Dirichlet approximations are required in order to obtain a smaller admissible value of $H$. This, in turn, requires an initial application of Heath-Brown's identity to decompose the exponential sums into sums of the shape ~\eqref{h2}, and then an application of Dirichlet approximation with a parameter $Q$ depending on the sizes of the $M_i$ in~\eqref{h2}. This introduces additional technical complications into the proof. As far as we are aware, this approach has not previously been pursued in the literature for problems of this flavour.


 It is worth noting an important distinction between our approach and Vinogradov's method. The latter reduces the problem to equidistribution estimates for exponential sums, thereby discarding some of the arithmetic structure of the functions $\mu$ and $\Lambda$. By contrast, our approach to~\eqref{eq:intro-dirichlet-poly}, uniformly in $\alpha$, is closer to methods for studying primes in short intervals without using zero-density estimates. Large value estimates (twisted by characters) therefore play a crucial role, and improvements to these estimates should allow us to treat shorter intervals. 
 
 To conclude this section, we also point out that, very recently, Maynard, Pandey, and Radziwi{\l}{\l}~\cite{MPR} have also used products of three Dirichlet polynomials to sharpen the upper bound for the exponential sum $$\sum_{n\leq x}\Lambda(n)e(n\alpha).$$ They observed that zero-density arguments yield weaker bounds. This prompted us to examine what the corresponding zero-density estimates would give in our setting, where the same phenomenon occurs. Indeed, this direction had already been taken by \cite{PanPan}, wherein they obtained the desired result for intervals of length $H\geq x^{2/3}(\log x)^{O_{A}(1)}$. Following this approach and inserting the zero density estimates in short intervals which would be derived from the large value estimates in \cite[Theorem 1.1]{CGY} would lead to non-trivial cancelation for intervals of length at least $H\geq x^{0.65+\varepsilon}$, such a conclusion being far worse than what is obtained in Theorem \ref{thm:mobius}.

\subsection*{Limiting cases}
In the proofs of the above theorems, there are some limiting cases for the lengths of the Dirichlet polynomials arising from Heath-Brown's identity which force one to impose the restriction $H\geq x^{0.58+\varepsilon}$, thereby preventing any substantial improvement without new arithmetic information. The particular case in which there are five Dirichlet polynomials of length $x^{1/5+o(1)}$ entails bounding the large value set of the polynomials ``corresponding to the $7/10$-line", and this does not seem to be tractable whenever $H\leq x^{0.58}$ with the current technology available.

\subsection*{Organization of the paper} We organise the paper as follows: Section \ref{sec2} includes a collection of preliminary lemmata involving mean value estimates, large value estimates and pointwise estimates for Dirichlet polynomials. In Section \ref{sec3} we briefly present the major arc analysis. We then apply Ramare's identity to extract a small prime factor from $\mu$, and subsequently use Heath-Brown's identity in Section~\ref{sec4} to reduce the estimate to that of $\mu_N$ in \eqref{mobius-n}, where $\mu_N$ depends on the sizes of $(N_1,\dots,N_{2k})$. In Section \ref{sec5} we use stationary phase to reduce the problem to that of estimating mean values of Dirichlet polynomials. We analyse Type I and II sums, Type I$_2$ sums and Type III sums in Sections \ref{sec6}, \ref{sec7} and \ref{sec8} respectively. We then conclude the proof of Theorem \ref{thm:mobius} in Section \ref{sec9}, and include those of Theorems \ref{thm:mangoldt} and \ref{thm1.3} in Sections \ref{sec10} and \ref{sec11} respectively.
 
\subsection*{Notation}

For $\alpha\in\mathbb{R}$ then $\|\alpha\|$ denotes the distance to the nearest integer. We write $f=\Omega(g)$ if $f\gg g$ and $f=\Theta(g)$ whenever $f\asymp g$. Throughout the paper, $\eps>0$ denotes a sufficiently small constant that may vary from occurrence to occurrence. Taking the $\eps$ in the lower bound for $H$ in Theorem~\ref{thm:mobius} as a reference, other occurrences may denote fixed positive multiples of it, such as $100\eps$ or $\eps/100$.

\subsection*{Acknowledgement}
The authors are grateful to Kaisa Matom\"aki and James Maynard for their helpful suggestions and discussions. JP was partially supported by the project PID2024-156180NB-I00 funded by MICINN-Spain. MW was supported by the Swiss National Science Foundation grant TMSGI2-211214. 

\subsection*{AI declaration}
The authors used ChatGPT Pro to improve the language in parts of the manuscript. 

\section{Preliminaries}\label{sec2}

In this section, we collect several technical lemmas that will be used throughout the paper. Most of them concern Dirichlet polynomials, including pointwise bounds, mean values estimates, and large value estimates.



\subsection{Mean values and large values of Dirichlet polynomials}

\begin{lemma}[Mean value estimate]\label{lem:mean-value}
 Suppose that $T,T'\geq 0$ with $T,T'\in \R$ and $q\in\mathbb{N}$. If $f$ is a $k$-divisor bounded function, then
\[
\sum_{\chi \mod q} \int_{T}^{T+T'} \Bigabs{\sum_{N<n \leq 2N}\frac{f(n)\chi(n)}{n^{1/2+it}}}^2 \rd t\ll_k(\log q(T+T'+2)N)^{O_k(1)} (qT'+N).
\]
\end{lemma}
\begin{proof}
See \cite[Theorem 9.12]{IK}.	
\end{proof}

\begin{lemma}[Second moment in short intervals]\label{lem:second-moment}
Suppose that $q\geq 1$, $T'\geq T^{1/3} $ and $N\leq (T+T')^4$ with $T\geq 0$.  Then,
\[
 \sum_{\substack{\chi \mod q}}\int_{T}^{T+T'}\Bigabs{\sum_{N<n \leq 2N}\frac{\chi(n)}{n^{1/2+it}}}^2 \rd t\ll qT'(\log q(T+T'+2))^3+\frac{N}{1+T}. \]
Likewise,
\[
 \sum_{\substack{\chi \mod q}}\int_{T}^{T+T'}\Bigabs{\sum_{N<n \leq 2N}\frac{\chi(n)\,\log n}{n^{1/2+it}}}^2 \rd t\ll qT'(\log q(T+T'+2))^{5} +\frac{N(\log q(T+T'+2))^{2}}{1+T}.
\] 
\end{lemma}
\begin{proof}

This follows from a truncated version of \cite[Lemma 2]{Zhan}, which concerns the second moment of Dirichlet $L$-functions in short intervals. We briefly discuss the argument below, which essentially follows \cite[Corollary 2.14]{MRTI}.	

We first focus on the first estimate.  Set $\mathcal D{\bf1}_I(s,\chi) =\sum_{n\in I}\chi(n)n^{-s}$ for a  finite interval $I\subseteq\Z$. The Fundamental theorem of calculus  gives
\begin{align}\label{eq:2.1}
\mathcal D {\bf1}_{(N,2N]}\Big(\frac12+it, \chi\Big) =\frac{1}{(2N)^{1/2}} \mathcal D {\bf1}_{(N,2N]}(it,\chi) + \int_N^{2N} \mathcal D {\bf1}_{(N,x]} (it,\chi) \frac{\rd x}{2x^{3/2}}. 
\end{align}
By the triangle inequality, it therefore suffices to prove that
\begin{align}\label{eq:second-estimate}
\sum_{\chi \mod q} \int_{T}^{T+T'}\bigabs{ \mathcal D{\bf1}_{[1,N]} (it,\chi)}^2 \rd t \ll qT'N(\log q(T'+T+2))^{3} + \frac{N^2}{(1+T)}.
\end{align} 

Employing the truncated Perron's formula (cf. \cite[Lemma 2.4]{MRTI}) for the choice $U=(q(T'+T))^5+1$ one has
\[
\bigabs{ \mathcal D{\bf1}_{[1,N]} (it,\chi)} \ll \biggabs{ \int_{-U}^{U} L\Bigbrac{1+\frac{1}{\log N} +i(t+u),\chi} \frac{N^{1+\frac{1}{\log N} + iu}}{1+\frac{1}{\log N} +iu} \rd u} + \frac{N(\log 2U)}{U}+1.
\]
We shall shift the contour to the line with real part $1/2$. To such an end, it is useful recalling the convexity bound  \cite[Theorem 3]{Ramachandra-conve}, which asserts that
 \[
 L(\sigma +it,\chi) \ll (q(1+|t|))^{\eps+\max\set{0, \frac{1-\sigma}{2}}}
 \]
  uniformly for $1/2\leq \sigma \leq 2$ and $|t|\geq 1$, where   $\eps>0$ is arbitrary but fixed. Observing that the only possible pole arises from the principal character $\chi_0$, its contribution being $O(\frac{N}{1+|t|})$, we obtain that
\begin{align}\label{eq:2.2}
\bigabs{ \mathcal D{\bf1}_{[1,N]} (it,\chi)} \ll  &\biggabs{ \int_{-U}^{U} L\Bigbrac{\frac12+i(t+u),\chi} \frac{N^{\frac12 + iu}}{\frac12 +iu} \rd u} +1 +\frac{N(qU)^{\varepsilon}}{U}\nonumber
\\
&+\frac{(qU)^{1/4}N^{1/2}}{U}(qU)^{\varepsilon}+{\bf1}_{\chi=\chi_0} \frac{N}{1+|t|}.
\end{align}

It is worth noting that the assumption $N\ll U^{4/5}$ entails that the terms $N(qU)^{\varepsilon}U^{-1}$ and $(qU)^{1/4}N^{1/2}(qU)^{\varepsilon}U^{-1}$ are $O(1)$. Applying H\"older's inequality with respect to the variable $u$ and then splitting the range of integration for the $u$-integral into $|u|\leq T'+T$ and the complementary set, and making a dyadic dissection, we obtain
\begin{align}\label{hhh}
&\sum_{\chi \mod q} \int_{T}^{T+T'}\bigabs{ \mathcal D{\bf1}_{[1,N]} (it,\chi)}^2 \rd t\nonumber\\
\ll& N(\log2 U) \sum_{\substack{\chi \mod q\\ }} \int_{T}^{T+T'}\int_{-U}^{U} \frac{\bigabs{ L\bigbrac{\frac12+i(t+u),\chi}}^2}{1+|u|} \rd u \rd t+q T'+\frac{N^2}{1+T}\nonumber\\
\ll& N (\log2U)^2 \sup_{|u|\leq (T'+T)}  \sum_{\substack{\chi \mod q\\ }} \int_{T}^{T+T'}  \Bigabs{L\Big(\frac12+i(t+u),\chi\Big)}^2 \rd t+q T'+\frac{N^2}{1+T}\\
&+ 	N (\log2U)^2 \sup_{(T'+T)\leq V\leq U} \int_{T}^{T+T'}   \sum_{\substack{\chi \mod q\\ }} \int_{V}^{2V} \frac{\bigabs{L(\frac12+i(t+u),\chi)}^2}{1+|u|} \rd u \rd t.\nonumber
\end{align}
By the second-moment estimate for Dirichlet $L$-functions over long intervals (cf. \cite[Theorem 10.1]{montgomery}), the last term is $O(qNT' (\log qU)^3)$, as desired. For the first integral when $T'\leq T$ we apply \cite[Lemma 2]{Zhan}, while otherwise we use \cite[Theorem 10.1]{montgomery}. In either case, we again infer that it is $O(qNT' (\log q(T+T'+2))^3)$, whence the estimate \eqref{eq:second-estimate} follows accordingly.

As for the second estimate, employing summation by parts we deduce that
\[
\Bigabs{\sum_{N<n \leq 2N}\frac{\chi(n)\,\log n}{n^{s}}} \ll (\log N) \Bigabs{ \sum_{N<n \leq 2N}\frac{\chi(n)}{n^s}} +\frac1N \sum_{N<N'\leq 2N} \Bigabs{ \sum_{N<n \leq N'}\frac{\chi(n)}{n^s}}.
\]
Therefore, we may remove the logarithmic weight at a cost of $O(\log2N)$, and hence $O((\log 2N)^2)$ in the second moment. Since $\log2N\ll\log2q(T+T'+2)$, this proves the second estimate and completes the proof.
\end{proof}

\begin{lemma}[Fourth moment in short intervals]\label{lem:fourth-moment}
	Let $q\in\mathbb{N}$ with $\log q\ll \log T$ and $T,T'\geq 0$, and $\eps>0$ be sufficiently small. Suppose that $ T'\gg T^{\eps }$ and $N \leq (T+T')^4$. Then one has
\[
\sum_{\chi \mod{q}}\int_{T}^{T+T'}\Bigabs{\sum_{N<n \leq 2N}\frac{\chi(n)}{n^{1/2+it}}}^{4}\rd t\ll T^{\varepsilon}(qT'+q^{71/96}T^{2/3}) + \frac{N^2}{1+T^{3}}.
\]
If in addition $T'=T$, the same bound holds with $(\log T)^{O(1)}$ replacing the $T^{\varepsilon}$ factor. The same estimates hold with $\chi(n)$ replaced by
$\chi(n)\log n$ and the last summand having an extra $(\log T)^{O(1)}$ factor.
\end{lemma}

\begin{proof}

As in Lemma~\ref{lem:second-moment}, the second term in the above line arises from the contribution of the principal character. We may evaluate this separately. Indeed, from~\eqref{eq:2.2}, its contribution to the mean value of $|\mathcal D {\bf1}_{(N,2N]}(it, \chi)|^4$ can be bounded by
\[
\sum_{\chi \mod q} \int_{T}^{T+T'}\Bigbrac{{\bf1}_{\chi=\chi_0} \frac{N}{1+t}}^4 \rd t\ll \frac{N^4 }{1+T^{3}}.
\]
This is acceptable, since~\eqref{eq:2.1} gives a weight of size $\Omega(N^2)$ to the fourth moment of $|\mathcal D {\bf1}_{(N,2N]}(\frac12+it, \chi)|$. Therefore, in what follows, we may ignore the contribution of the principal character.

An argument similar to that in the proof of Lemma \ref{lem:second-moment} together with \cite[Theorem 3]{Jutila} yields the desired estimate with the term $q^{71/96}T^{2/3}$ replaced by $qT^{2/3}$ whenever $T^{1/2+\eps }\leq T' \leq T^{2/3}$. To obtain the sharper dependence on $q$ and extend the validity range of $T'$, we use a more precise fourth-moment bound for Dirichlet $L$-functions in short intervals.

We first notice that \cite[Theorem 4]{JM} yields
\[
\sum_{\chi \mod q} \int_{T}^{T+T'} \Bigabs{ L\Big(\frac{1}{2}+it,\chi\Big) }^4 \rd t \ll T^\eps (qT' + q^{\frac{2(1+\eta)}{3}} T^{2/3}),
\]
whenever $T^\eps\ll T'\ll T^{2/3}$ and $\log q\ll \log T$, where $\eta\geq 0$ is an admissible exponent  towards the Ramanujan--Petersson conjecture. The Kim--Sarnak bound in \cite[Appendix 2]{Kim} allows one to take $\eta=7/64$, giving $2(1+\eta)/3 \leq 71/96$. Hence, when $T^\eps \ll T'\ll T^{2/3}$, the truncated argument in the proof of Lemma \ref{lem:second-moment} combined with the fourth moment (cf.~\cite[Lemma 10.9]{Har}) in place of \cite[Theorem 10.1]{montgomery}, and  the above fourth moment bound in place of \cite[Lemma 2]{Zhan} then gives the claimed bound.

 Now  assume that $T^{2/3}\leq T'\leq T$. Partition $[T,T+T']$ into intervals of length $\Theta(T^{2/3})$. Applying the estimate
 \[
\sum_{\chi \mod{q}}\int_{U}^{U+\Theta(T^{2/3})} \Bigabs{ L\Big(\frac{1}{2}+it,\chi\Big) }^4\rd t\ll qT^{2/3+\varepsilon}
\]
uniformly over the left endpoints $T\leq U\leq T+T'$ of these intervals, summing over the intervals and combining the resulting estimate with the truncated argument in the proof of Lemma \ref{lem:second-moment} yields the desired bound.

Finally, the case $T'\geq T$ is handled by the argument used in the proof of Lemma~\ref{lem:second-moment} but employing instead the fourth-moment estimate for Dirichlet $L$-functions over long intervals (cf.~\cite[Lemma 10.9]{Har}) to bound the analogue of the first integral in (\ref{hhh}). This completes the proof.
\end{proof}

\begin{theorem}[Large value estimates with characters]\label{thm:large-values}
Let $(a_{n})$ be a sequence of complex numbers with $\lvert a_{n}\rvert\leq 1$. Suppose that $W$ is a finite set of pairs $(t,\chi)$, where $t\in [T, T+U]$ and $\chi$ is a character modulo $q$, such that whenever $(t,\chi)\neq (t',\chi')$ then either $\chi\neq \chi'$ or $\lvert t-t'\rvert\geq 1$. Assume that \begin{equation*}
\Big\lvert \sum_{N<n \leq 2N} a_{n}\chi(n) n^{it}\Big\rvert\geq V\end{equation*} for all $(t,\chi)\in W$. Then,
\begin{equation}\label{rr4}
\lvert W\rvert\ll (\log 2qU)^{2} \big(N^{2}V^{-2}+qU\min(NV^{-2},N^{4}V^{-6})\big).\end{equation}
Moreover, suppose that $q\leq N^{6/5}$. If $ (qU)^{2/3}\leq N\leq (qU)^{5/6}$ one has
\begin{equation}\label{pp9}
\lvert W\rvert\ll (\log 2qU)^{2}N^{2}V^{-2}+ (qU)^{\varepsilon}qU N^{12/5}V^{-4},\end{equation} and if $N>(qU)^{5/6}$ then 
\begin{equation}\label{pp91}
\lvert W\rvert\ll (\log 2qU)^{2}N^{2}V^{-2}+(qU)^{\varepsilon}(qU)^{1/2} N^{3}V^{-4}.\end{equation}
\end{theorem}

\begin{proof}
We consider the set
$$
W_{U}=\bigset{(t-T,\chi):\ \ (t,\chi)\in  W },
$$ 
and define the sequence of complex numbers $b_{n}=a_{n}n^{iT}$. By the above assumptions, $W_{U}$ is a finite set of pairs $(t,\chi)$, where $t\in [0,U]$ and $\chi$ is a character modulo $q$, with the property that whenever $(t,\chi)\neq (t',\chi')$ then either $\chi\neq \chi'$ or $\lvert t-t'\rvert\geq 1$. Moreover, $\lvert b_{n}\rvert=\lvert a_{n}\rvert\leq 1$ and whenever $(t,\chi)\in W_{U}$ then 
$$\Big\lvert \sum_{N<n \leq 2N} b_{n}\chi(n) n^{it}\Big\rvert=\Big\lvert \sum_{N<n \leq 2N} a_{n}\chi(n) n^{i(t+T)}\Big\rvert\geq V.$$
 The estimate (\ref{rr4}) now follows from \cite[Lemma 10.2]{Har}.

In order to obtain (\ref{pp9}) and \eqref{pp91} we observe that a direct consequence of \cite[Theorem 1.1]{CGY} entails whenever $ N\geq (qU)^{2/3}$ that
\begin{equation}\label{2.8}\lvert W\rvert\ll (qU)^{\varepsilon}(N^{2}V^{-2}+(qU)^{1/2}N^{3}V^{-4}+q^{4/3}UN^{2}V^{-4}+qU N^{12/5}V^{-4}).\end{equation}
We first prove a stronger, loss-separated version of this bound, namely
\begin{align}\label{w-loss-sep}
	|W| \ll (\log 2qU)^{2} N^2V^{-2} +(qU)^\eps \big((qU)^{1/2}N^{3}V^{-4}+q^{4/3}UN^{2}V^{-4}+qU N^{12/5}V^{-4}\big).
\end{align}
We split the analysis into two cases according to whether $N^2V^{-2} \leq (qU)^{1/2}N^{3}V^{-4}$ or not. If the latter holds, the second term dominates the first one in (\ref{2.8}). Therefore, one has
\[
\lvert W\rvert\ll (qU)^{\varepsilon}((qU)^{1/2}N^{3}V^{-4}+q^{4/3}UN^{2}V^{-4}+qU N^{12/5}V^{-4}).
\]
This gives the desired bound \eqref{w-loss-sep}. If instead $N^2V^{-2} > (qU)^{1/2}N^{3}V^{-4}$, we have from \eqref{rr4} that
\[
\lvert W\rvert\ll (\log 2qU)^2 (N^2V^{-2} + qU N^{4} V^{-6}).
\]
We then note that under the assumption $N^2V^{-2} > (qU)^{1/2}N^{3}V^{-4}$ it transpires that
\[
qU N^{4} V^{-6} =\frac{\big( (qU)^{1/2}N^{3}V^{-4} \big)^2}{N^2V^{-2}} \leq (qU)^{1/2}N^{3}V^{-4}.
\]
It then follows that
\[
\lvert W\rvert\ll (\log 2qU)^2 (N^2V^{-2} + (qU)^{1/2}N^{3}V^{-4})\ll (\log 2qU)^2 N^2V^{-2} + (qU)^\eps (qU)^{1/2}N^{3}V^{-4}.
\]
This also yields the estimate \eqref{w-loss-sep}.

We now return to the proof of \eqref{pp9} and \eqref{pp91}.  If $q\leq N^{6/5}$  it is apparent that the third term in \eqref{w-loss-sep} is always smaller than the fourth one, namely
$$q^{4/3}UN^{2}V^{-4}\leq qU N^{12/5}V^{-4},$$
 and therefore
 \[
|W| \ll (\log 2qU)^{2} N^2V^{-2} +(qU)^\eps \big((qU)^{1/2}N^{3}V^{-4}+qU N^{12/5}V^{-4}\big). 
 \]
 
Assume first that $N\leq (qU)^{5/6}$. Then one further has 
$$ (qU)^{1/2}N^{3}V^{-4}\leq qU N^{12/5}V^{-4}.$$ Combining the above two estimates thereby delivers the conclusion (\ref{pp9}). If instead $q\leq N^{6/5}$ but $N>(qU)^{5/6}$ then $$ q^{4/3}UN^{2}V^{-4}\leq qU N^{12/5}V^{-4}\leq (qU)^{1/2}N^{3}V^{-4},$$ whence (\ref{pp91}) holds as well.
\end{proof}

\subsection{Single Dirichlet polynomial}

We take for convenience the parameters 
\begin{equation}\label{PPtilde}
 P=\exp((\log x)^{2/3+\varepsilon_{0}/2}),\ \ \ \widetilde P=x^{1/(\log\log x)^{2}},
 \end{equation} where $\varepsilon_{0}>0$ is sufficiently small. For any $P\leq P_{1}\leq \widetilde P$ and a Dirichlet character $\chi \mod q$, we introduce the Dirichlet polynomial
$$P(s,\chi)=\sum_{P_{1}< p\leq \min\{2P_{1},\widetilde P\}}\chi(p)p^{-s}.$$
\begin{lemma}\label{lem:bd-diri-prime}
 Let $t\in \mathbb R$ and $|t|\leq x$. Let $\chi$ be a Dirichlet character with conductor $q\leq (\log x)^A$. Then, there exists some constant $c>0$ such that
\[
\Big| P(1+it,\chi)\Big| \ll\frac{\log x}{1+|t|} + \exp\big(-c(\log x)^{\eps_{0}/3}\big).
\]
\end{lemma}
\begin{proof}
See \cite[Lemma 3.1]{PSW}.
\end{proof}

\begin{lemma}\label{lem:upper-bd-diri}
Fix $\gamma\in (0,1]$. Then, for any $A>0$, any $1\leq r\leq x$ and any character $\chi$ of modulus $q\leq (\log x)^{A}$, one has
$$\sup_{\lvert t\rvert\leq x}\sup_{I\subset [x^{ \gamma},2 x^{ \gamma}]}\Big\lvert \sum_{l\in I}\frac{\mu(rl)\chi(l)}{l^{1/2+it}}\Big\rvert\ll \frac{x^{\gamma/2}}{(\log x)^{A}}.$$Likewise, for every $0\leq T_{0}\leq x$, there exists some $\delta=\delta(\gamma)>0$ such that
$$\sup_{T_{0}\leq \lvert t\rvert\leq x}\sup_{I\subset [x^{ \gamma},2 x^{ \gamma}]}\Big\lvert \sum_{l\in I}\frac{\chi(l)}{l^{1/2+it}}\Big\rvert\ll x^{\gamma/2-\delta}+\frac{x^{\gamma/2}\log x}{(1+ T_{0})^{1/2}}.$$
\end{lemma}
\begin{proof}
See \cite[Lemma 3.9]{MSTT}.
\end{proof}


\section{Major-arc analysis}\label{sec3}

We shall start our discussion by making a routinary major and minor arc dissection. To such an end, we define the set of \emph{major arcs} $\major$ to be the union of the intervals
$$\major(q,a)=\Big\{\alpha\in [0,1):\ \lvert \alpha-a/q\rvert\leq \frac{(\log x)^{5A}}{H}\Big\},$$ where $q\in\mathbb{N}$ and $a\in\mathbb{Z}$ with $(a,q)=1$ and such that $1\leq q\leq (\log x)^{3A}$ and $0\leq a\leq q$. We also introduce the \emph{minor arcs}  $\minor =[0,1)\setminus \major$.

\begin{proposition}[Major-arc analysis]\label{prop:major-arc}
Suppose that $H\geq x^{17/30+\varepsilon}$.  Then one has
$$
\sup_{\alpha\in \major}\Bigabs{\sum_{x\leq n\leq x+H}\mu(n)e(\alpha n)} \ll He^{-(\log x)^{1/8}}.$$
\end{proposition}

\begin{proof}
Write $\alpha =\frac aq+\beta$, where $ a\in\mathbb{Z}$ and $q\in\mathbb{N}$ with $(a,q)=1$ such that $0\leq a\leq q\leq (\log x)^{3A}$ and $\lvert \beta\rvert\leq \frac{(\log x)^{5A}}{H}$. Splitting the  sums according to  $q_0=\gcd(q,n)$,  and employing orthogonality of characters it follows that 
\begin{align*}
\sum_{x<n\leq x+H}&\mu(n)e(\alpha n)=\sum_{x<n\leq x+H}\mu(n)e(\beta n)e(an/q)
\\
&=\sum_{q_{0}q_{1}=q}\sum_{\substack{ r\mod{q_{1}}\\ (r,q_{1})=1}}\sum_{\substack{\frac{x}{q_{0}}<n\leq \frac{x}{q_{0}}+\frac{H}{q_{0}}\\ n\equiv r\mod{q_{1}}}}\mu(q_{0}n)e(\beta q_{0}n)e(an/q_{1})
\\
&=\sum_{q_0q_{1}=q} \sum_{\substack{r \mod{q_{1}}\\ (r,q_{1})=1}} e\Big(\frac{ar}{q_{1}}\Big) \phi(q_{1})^{-1} \sum_{\chi\mod{q_{1}}} \overline{\chi(r)}  \sum_{\frac{x}{q_{0}}<n\leq \frac{x}{q_{0}}+\frac{H}{q_{0}}} \mu(q_{0}n)e(\beta q_{0}n) \chi(n).
\end{align*}
Therefore, writing\begin{equation}\label{dd3}
C(a,\chi)=\sum_{\substack{r \mod{q_{1}}\\ (r,q_{1})=1}} e\Big(\frac{ar}{q_{1}}\Big)  \overline{\chi(r)},\end{equation} 
it follows that
$$\sum_{x<n\leq x+H}\mu(n)e(\alpha n)=\sum_{q_0q_{1}=q}\phi(q_{1})^{-1}\sum_{\chi\mod{q_{1}}}C(a,\chi)\sum_{\substack{\frac{x}{q_{0}}<n\leq \frac{x}{q_{0}}+\frac{H}{q_{0}}}} \mu(q_{0}n)e(\beta q_{0}n) \chi(n).$$

We shall next focus on the inner sum. To such an end it is worth considering the $L$-function
$$ L_{q_{0}}(s,\chi)=\sum_{n=1}^{\infty}\frac{\mu(q_{0}n)\chi(n)}{n^{s}}.$$ By employing multiplicativity and the fact that any $n\in\mathbb{N}$ with $(n,q_{0})>1$ satisfies $\mu(q_{0}n)=0$ it transpires that
\begin{align*}L_{q_{0}}(s,\chi)=\mu(q_0)\sum_{\substack{n=1\\ (n,q_{0})=1}}^{\infty}\frac{\mu(n)\chi(n)}{n^{s}}&=\mu(q_{0})\prod_{p}\Big(1-\frac{\chi(p)}{p^{s}}\Big).\prod_{p|q_{0}}\Big(1-\frac{\chi(p)}{p^{s}}\Big)^{-1}
\\
&=\frac{\mu(q_{0})}{L(s,\chi)}\prod_{p|q_{0}}\Big(1-\frac{\chi(p)}{p^{s}}\Big)^{-1}.
\end{align*}
Consequently, an application of the Main Theorem in \cite{Ramachandra} with $F(s)=L_{q_{0}}(s,\chi)$ and $B=\frac{30}{13}$ in view of \cite[Theorem 1.2]{GM} permits one to deduce that whenever $u\gg x^{17/30+\varepsilon/3}$ then 
\begin{equation}\label{ggg}\sum_{\substack{\frac{x}{q_{0}}<n\leq \frac{x}{q_{0}}+u}} \mu(q_{0}n) \chi(n)= \mu(q_{0})\sum_{\substack{\frac{x}{q_{0}}<n\leq \frac{x}{q_{0}}+u\\ (n,q_{0})=1}} \mu(n) \chi(n)\ll ue^{-(\log x)^{1/6}}.\end{equation}  Consequently, summation by parts in conjunction with the above estimate and the bound $\lvert \beta \rvert H\leq (\log x)^{5A}$ yields
\begin{align*}
&\sum_{\substack{\frac{x}{q_{0}}<n\leq \frac{x}{q_{0}}+\frac{H}{q_{0}}}}\mu(q_{0}n)e(\beta q_{0}n) \chi(n)\\
&\ll \Big\lvert\sum_{\substack{\frac{x}{q_{0}}<n\leq \frac{x}{q_{0}}+\frac{H}{q_{0}}\\ (n,q_{0})=1}} \mu(q_0n) \chi(n)\Big\rvert
 +q_{0}\lvert \beta \rvert \int_{0}^{\frac{H}{q_{0}}}\Big\lvert\sum_{\substack{\frac{x}{q_{0}}<n\leq \frac{x}{q_{0}}+u\\ (n,q_{0})=1}} \mu(q_0n) \chi(n)\Big\rvert \rd u\\
&\ll He^{-(\log x)^{1/6}}+x^{17/30+\varepsilon/3}\lvert \beta\rvert H
+q_{0}\lvert \beta \rvert \int_{x^{17/30+\varepsilon/3}}^{\frac{H}{q_{0}}}\Big\lvert\sum_{\substack{\frac{x}{q_{0}}<n\leq \frac{x}{q_{0}}+u\\ (n,q_{0})=1}} \mu(q_0n) \chi(n)\Big\rvert \rd u.
\end{align*}
Therefore, inserting (\ref{ggg}) and the bound $\lvert \beta\rvert H\leq (\log x)^{5A}$ in the above lines delivers
\begin{align*}\sum_{\substack{\frac{x}{q_{0}}<n\leq \frac{x}{q_{0}}+\frac{H}{q_{0}}}}\mu(q_{0}n)e(\beta q_{0}n) \chi(n)\ll He^{-(\log x)^{1/7}}.
\end{align*}
Combining the preceding equations thereby gives
\begin{align*}\sum_{x<n\leq x+H}\mu(n)e(\alpha n)&\ll He^{-(\log x)^{1/7}}\sum_{q_0q_{1}=q}\phi(q_{1})^{-1}\sum_{\chi\mod{q_{1}}}\lvert C(a,\chi)\rvert
\\
&\ll He^{-(\log x)^{1/7}}q^{1/2+\varepsilon}\ll He^{-(\log x)^{1/8}},
\end{align*}
as desired.

\end{proof}


\section{Extracting a small prime factor and reduction}\label{sec4}

Due to the multiplicative property of the M\"obius function, we may apply Ramar\'e's identity to extract a small prime factor from the M\"obius function. To such an end we recall the parameters $P$ and $\tilde P$ in \eqref{PPtilde} and  employ   \cite[Lemma 3.1]{MT} to deduce that
$$
\sum_{x< n\leq x+H}\mu(n)e(\alpha n)=-\sum_{\substack{x<prn\leq x+H\\ P<p\leq \tilde P\\ r\leq x^{\varepsilon/2}}}a_{r}\mu(n)e(\alpha prn)+O\big(H(\log x)^{-\frac13+\eps }\big),
$$
where $a_{r}$ is an explicit sequence satisfying $\lvert a_{r}\rvert\leq d(r)$. 

The application of Heath-Brown's identity to $\mu$ with $k=20$ as in \cite{MT} and a dyadic dissection of the corresponding intervals permits one to express the above sum as a linear combination of $O((\log x)^{2k+1})$ sums of the form
$$\sum_{\substack{x<prn_{1}\cdots n_{2k-1}\leq x+H\\ p\in (P_{1},2P_{1}]\ r\in (R,2R], n_{i}\in (N_{i},2N_{i}]}}a_{r}a_{1}(n_{1})\cdots a_{2k-1}(n_{2k-1})e(\alpha prn_{1}\cdots n_{2k-1}),$$
where
$$P_{1}\in [P,\widetilde P],\ \ \ R\in[1/2,x^{\varepsilon/2}],\ \ \ N_{1},\ldots, N_{k-1}\in [1/2,x],$$
\begin{equation}\label{NNN}N_{k},\ldots,N_{2k-1}\in [1/2,(2x)^{1/k}],\ \ \ P_{1}RN_{1}\ldots N_{2k-1}\asymp x,\end{equation}
and for each $i$ one has $a_{i}(n)=1$ if $1\leq i\leq k-1$ and $a_{i}(n)=\mu(n)1_{(N_{i},2N_{i}]}(n)$ if $k\leq i\leq 2k-1$. We also present for further use the function
$$\mu_{\boldsymbol N}(n)=\sum_{\substack{n=prn_{1}\cdots n_{2k-1}\\ p\in (P_{1},2P_{1}]\ r\in (R,2R], n_{i}\in (N_{i},2N_{i}]}}a_{r}a_{1}(n_{1})\cdots a_{2k-1}(n_{2k-1}),$$ where $\boldsymbol N=(P_{1},R,N_{1},\ldots,N_{2k-1})$, and observe that with this notation
\begin{equation}\label{mobius-n}
\sum_{\substack{x<prn_{1}\cdots n_{2k-1}\leq x+H\\ p\in (P_{1},2P_{1}]\ r\in (R,2R]\\ n_{i}\in (N_{i},2N_{i}]}}a_{r}a_{1}(n_{1})\cdots a_{2k-1}(n_{2k-1})e(\alpha prn_{1}\cdots n_{2k-1})= \sum_{x< n\leq x+H}\mu_{\boldsymbol N}(n)e(\alpha n).\end{equation}
We summarise the above argument in the statement below.
\begin{lemma}\label{lem:reduce-to-mn}
	Let $\alpha\in \T$ and $k=20$. Then, one has
	\[
	\Big| \sum_{x<n \leq x+H} \mu(n) e(\alpha n)\Big| \ll (\log x)^{2k+1}\sup_{\boldsymbol N} \Big| \sum_{x<n \leq x+H} \mu_{\boldsymbol N} (n) e(n\alpha)\Big| + H(\log x)^{-1/3+\eps },
	\]
	where $\mu_{\boldsymbol N}$ is defined in \eqref{mobius-n} and $\boldsymbol N$ runs over the tuples satisfying (\ref{NNN}).
\end{lemma}

As was anticipated in Section~\ref{introduction}, we shall use different Dirichlet approximations to $\alpha$ for different choices of $N_1,\dots,N_{2k-1}$. To this end, let $Q$ be a parameter satisfying $(\log x)^A\leq Q\leq x$ which will depend on $N_1,\dots,N_{2k-1}$. The Dirichlet's approximation theorem implies that for any $\alpha \in \T$ there exist integers $0\leq a\leq q\leq Q$ with $(a,q)=1$ such that 
\begin{align}\label{eq:dirich-app}
	\bigabs{\alpha-a/q} \leq \frac{1}{qQ}.
\end{align}
Setting $\beta=\alpha-a/q$, the procedure in the proof of Proposition \ref{prop:major-arc} enables one to derive
\begin{align}\label{eq:sect4}
\sum_{x<n \leq x+H} \mu_{\boldsymbol N} (n) e(n\alpha) =\sum_{q_0q_1=q} \sum_{\chi \mod{q_1}} \frac{C(a,\chi)}{\phi(q_1)} \sum_{\frac{x}{q_0}<n \leq \frac{x+H}{q_0}} \chi(n) \mu_{\boldsymbol N}(q_0n) e(\beta q_0n).
\end{align} 

We next  restrict the summation over $q_0$ to $q_0\leq\exp((\log x)^{1/2})$, the  threshold being smaller than the parameter $P$ in \eqref{PPtilde}. This will allow us to discard the possibility that primes $p$ appearing in the expression for $\mu_{\boldsymbol N}$ divide $q_{0}$ (see Lemma~\ref{lem: produ-diri-poly}). For a broader applicability, we establish this property for general divisor-bounded functions.

\begin{lemma}\label{lem:restrict-small-q0}

Let $f:\Z\to\R$ be a divisor-bounded  function.  Suppose that $\alpha =a/q+\beta$ and $q\leq H^{1-\eps}$. Then one has
\begin{multline*}
\sum_{x<n \leq x+H} f (n) e(n\alpha) = O\bigbrac{ H \exp(-(\log x)^{1/3})}\\
+\sum_{q_0q_1= q \atop q_0\leq \exp((\log x)^{1/2})}\frac{1}{\phi(q_1)} \sum_{\chi \mod{q_1}} C(a,\chi) \sum_{\frac{x}{q_0}<n \leq \frac{x+H}{q_0}} \chi(n) f(q_0n) e(\beta q_0n).	
\end{multline*}

\end{lemma}

\begin{proof}
Observing that the formula \eqref{eq:sect4} also holds for $f$, we shall therefore focus on proving the equation by showing that the contribution from $q_0>\exp((\log x)^{1/2})$ is negligible. To such an end we shall use orthogonality of Dirichlet characters. Recalling that the function $f$ is divisor-bounded, we can estimate the contribution from $q_0>\exp((\log x)^{1/2})$ as follows:
	\begin{align*}
	&\sum_{ \substack{ q=q_0q_1\\ q_0> \exp((\log x)^{1/2})}}\phi(q_{1})^{-1} \sum_{\chi\mod{q_{1}}}C(a,\chi)\sum_{\substack{\frac{x}{q_{0}}<n\leq \frac{x}{q_{0}}+\frac{H}{q_{0}}}} f(q_{0}n)e(\beta q_{0}n) \chi(n)	\\
	=& \sum_{ \substack{ q=q_0q_1\\ q_0> \exp((\log x)^{1/2})}} \sum_{ \substack{r \mod{q_1} \\ (r,q_1)=1}} e\Big(\frac{ar}{q_1}\Big) \sum_{\substack{\frac{x}{q_{0}}<n\leq \frac{x}{q_{0}}+\frac{H}{q_{0}}\\ n\equiv r \mod {q_1}}} f(q_{0}n)e(\beta q_{0}n)\\
	\ll & H(\log x)^{O(1)}\sum_{ \substack{ q=q_0q_1\\ q_0> \exp((\log x)^{1/2})}} q_0^{-1+\eps},
	\end{align*}
where in the last step we applied Shiu's bound \cite{Shiu} to the inner sum. We now employ Rankin's trick, namely, ${\bf 1}_{q_0> \exp((\log x)^{1/2})}\leq \bigl(\frac{q_0}{\exp((\log x)^{1/2})} \bigr)^{1/3}$,    to see that
\[
\sum_{ \substack{ q=q_0q_1\\ q_0> \exp((\log x)^{1/2})}}  q_0^{-1+\eps } \ll  \exp(-(\log x)^{1/2}/3 )\sum_{d\mid q} d^{-1/2} \ll \exp(-(\log x)^{1/3}),
\]
where we have used Euler products to get the estimate
\[
\sum_{d\mid q} d^{-1/2}  \leq \prod_{p\mid q}(1-p^{-1/2})^{-1} \ll \exp\Bigl(\frac{c(\log q)^{1/2}}{\log \log q}\Bigr)
\]
for some constant $c>0$. Thus the lemma follows.
\end{proof}


\section{van der Corput integral estimate}\label{sec5}

In this section, we shall reduce the estimation of exponential sums to mean value estimates for Dirichlet polynomials. This forms the foundation of our minor-arc analysis. To such an end, for an arithmetic function $f:\Z\to\R$ we introduce the corresponding Dirichlet polynomial and exponential sum
\begin{align}\label{eq:f-dircht-poly-exp-sum}
D_f(s,\chi) = \sum_{n \sim x/{q_{0}}} \frac{f(q_0n)\chi(n)}{n^s}\qquad\text{and} \qquad  M_{f}(\alpha) =\sum_{x<n \leq x+H} f(n) e(n\alpha),
	\end{align}
	where $n \sim x/{q_{0}}$ denotes the condition $P_{1}RN_{1}\ldots N_{2k-1}/q_{0}\leq n\leq 2^{2k+1}P_{1}RN_{1}\ldots N_{2k-1}/q_{0}$.
\subsection{When $|\beta|$ is small}
\begin{proposition} \label{prop:small-beta} 
Let $f:\Z\to\R$ such that $|f|\leq d_{l}$ for some constant $l\in \Z$ and $x^{\varepsilon}\leq H\leq x^{1-2\eps }$.
	Let $\alpha =\frac aq +\beta \in \minor$ with $|\beta| \leq \frac{x(\log x)^{A/3}}{H^2}$. Then the following bounds hold:
	\begin{enumerate}
		\item If $q\leq (\log x)^{A}$ or $ \lvert \beta \rvert>x/H^{2}$, one has
\begin{align*} M_{f}(\alpha)\ll &\frac{H(\log x)^{A/2}}{\sqrt{qx}} \sum_{ \substack{ q=q_0q_1\\ q_0\leq \exp((\log x)^{\frac12})}}\sum_{\chi\mod{q_{1}}}\int_{2\pi \beta x-\frac{x(\log x)^{A}}{H}}^{2\pi \beta x+\frac{x(\log x)^{A}}{H}}\Big\lvert D_{f}\Big(\frac12+it,\chi\Big)\Big\rvert dt+H(\log x)^{-A/2}.
\end{align*}
\item  If instead $q> (\log x)^{A}$ and $\lvert \beta \rvert\leq x/H^{2}$, then one has
\begin{align}\label{rrr} M_{f}(\alpha)\ll &\frac{H(\log x)}{\sqrt{qx}}\sum_{ \substack{ q=q_0q_1\\ q_0\leq \exp((\log x)^{\frac12})}}\sum_{\chi\mod{q_{1}}}\int_{2\pi \beta x-\frac{x(\log x)^{A/50}}{H}}^{2\pi \beta x+\frac{x(\log x)^{A/50}}{H}}\Big\lvert D_{f}\Big(\frac12+it,\chi\Big)\Big\rvert dt+H(\log x)^{-A/52}.
\end{align}

	\end{enumerate}
	
\end{proposition}

\begin{proof}

In light of Lemma~\ref{lem:restrict-small-q0}, it suffices to bound the sum
\[
\sum_{q_0q_1= q \atop q_0\leq \exp((\log x)^{1/2})} \frac{1}{\phi(q_1)} \sum_{\chi \mod{q_1}} C(a,\chi)  \sum_{\frac{x}{q_0}<n \leq \frac{x+H}{q_0}} f(q_0n)\chi(n)  e(\beta q_0n).
\]
For each fixed factorisation $q=q_0q_1$, we  set
\begin{align}\label{eq:function-F}
F_{q_1}(n) = f(q_0n) \frac{1}{\phi(q_1)} \sum_{\chi \mod{q_1}} C(a,\chi)  \chi(n).
\end{align}
To obtain the bounds required above, it therefore suffices to estimate
\[
	\sum_{x/q_0<n \leq (x+H)/q_0} F_{q_1}(n) e(q_0\beta n)
\]
for each factorisation $q=q_0q_1$.

We first rewrite the phase function in the above sum as
\[
e(q_0\beta n) = e(\beta x) \Bigbrac{\frac{q_0n}{x}}^{2\pi i\beta x} W(q_0n),
\]
where $W(y) =e(\beta (y-x-x\log (y/x)))$, the above manoeuvre being motivated by the  Taylor expansion. Therefore, the employment of summation by parts gives
\begin{align*}
&\Bigabs{ \sum_{x/q_0<n \leq (x+H)/q_0} F_{q_1}(n) e(q_0\beta n)} \\
 \leq& \Bigabs{\sum_{x/q_0<n \leq (x+H)/q_0} F_{q_1}(n) n^{2\pi i \beta x} } +
 \biggabs{\int_{\frac{x}{q_0}}^{\frac{x+H}{q_0}} q_0 W'(q_0u)\sum_{\frac{x}{q_0}<n\leq u} F_{q_1}(n) n^{2\pi i \beta x}} \rd u\\
 \ll & \Bigbrac{1+\frac{|\beta|H^2}{x}} \sup_{0<u \leq H/q_0} \Bigabs{ \sum_{\frac{x}{q_0}<n\leq \frac{x}{q_0}+u} F_{q_1}(n) n^{2\pi i \beta x}},
\end{align*}
where in the last line we have used the bound
\[
W'(q_0u) \leq 2\pi |\beta| \Bigabs{1-\frac{x}{q_0u}}\ll |\beta| H/x
\]
uniformly for all $x/q_0\leq u\leq (x+H)/q_0$. 

We now claim that the following estimate holds uniformly for $ u\leq H/q_{0}$, namely,
\begin{align}\label{claim-1}
&\Big\lvert\sum_{\frac{x}{q_0}<n \leq \frac{x}{q_0}+u} F_{q_1} (n) n^{2\pi i \beta x}	\Big\rvert\nonumber\\
 \ll &\frac{q_0^{1/2}u(\log x)}{(q_{1}x)^{1/2}} \sum_{\chi \mod {q_{1}}}  \int_{2\pi \beta x-T_{q}(\beta)}^{2\pi \beta x+T_{q}(\beta)} \Bigabs{D_f\Big(\frac12+it,\chi\Big)} \rd t +Hd_{l}(q_0)q_{0}^{-1}(\log x)^{-B_{q}},
\end{align}
where $T_{q}(\beta)=\frac{x(\log x)^{A}}{H}$ and $B_q=11A/12$ if we are in case (1), while $T_{q}(\beta)=\frac{x(\log x)^{A/50}}{H}$ and $B_q=A/51$ when we are in case (2).

It is therefore easy to deduce the proposition from the claim \eqref{claim-1}. Indeed, when $x/H^2\leq |\beta|\leq x(\log x)^{A/3}/H^2$, summing over the factorisations  $q_0q_1=q$ and applying this claim gives
\begin{align*}
\sum_{q_0q_1= q \atop q_0\leq \exp((\log x)^{1/2})}& \Bigabs{ \sum_{x/q_0<n \leq (x+H)/q_0} F_{q_1}(n) e(q_0\beta n)}	\ll \frac{H}{(\log x)^{B_q-\frac{A}{3}}}\sum_{q_0\mid q}\frac{d_l(q_0)}{q_0} \\
&  +\sum_{q_0q_1= q \atop q_0\leq \exp((\log x)^{1/2})} \frac{H(\log x)^{1+\frac{A}{3}}}{(q_0q_{1}x)^{1/2}} \sum_{\chi \mod {q_{1}}}  \int_{2\pi \beta x-T_{q}(\beta)}^{2\pi \beta x+T_{q}(\beta)}\Bigabs{D_f\Big(\frac12+it,\chi\Big)} \rd t ,
\end{align*}
which is acceptable since $B_q-A/3+O(1)\geq A/2$; while when $|\beta|\leq x/H^2$ we have $1+|\beta|H^2/x\ll 1$, and a similar argument also delivers the estimate \eqref{rrr}.




It remains to prove \eqref{claim-1}.  We now apply Perron's formula to the above sum (cf.~\cite[Theorem 5.2 and Corollary 5.3]{MV}) with truncation height $T_{q}(\beta)$. This entails, for any $0<u \leq \frac{H}{q_0}$, that
	\begin{align*}
	&\sum_{\frac{x}{q_0}<n \leq \frac{x}{q_0}+u} F_{q_1} (n) n^{2\pi i \beta x} = \frac{1}{2\pi i}\int_{1/2-iT_{q}(\beta)}^{1/2+iT_{q}(\beta)} \Bigbrac{\sum_{n\sim \frac{x}{q_0}} \frac{F_{q_{1}}(n)n^{2\pi i x\beta}}{n^s}}\frac{(\frac{x}{q_{0}}+u)^{s}-\big(\frac{x}{q_{0}}\big)^{s}}{ s}\rd s\\
	&+ O\Bigbrac{\twosum{n \sim\frac{x}{q_0}}{} |F_{q_{1}}(n)| \Bigbrac{\min \set{1,\frac{x}{q_{0}T_{q}(\beta)|\frac{x}{q_{0}}-n|}}+\min\set{1,\frac{x}{q_{0}T_{q}(\beta)|n-(\frac{x}{q_0}+u)|}} }}.	
	\end{align*}
By the assumption $|f|\leq d_{l}$ for some constant $l\in \Z$ and the orthogonality of characters, we observe that
\[
\bigabs{ F_{q_1} (n)} \leq |f(q_0n)| \cdot \Bigabs{\sum_{r\mod {q_1} \atop (r,q_1)=1} e\bigbrac{\frac{ar}{q_1}} {\bf1}_{n\equiv r \mod{q_1}}} \ll d_{l}(q_0n).
\]
Splitting the range $[\frac{x}{2q_0}, \frac{2x}{q_0}]$ into two parts, namely $0<|n-(x/q_{0})|\leq x/(q_{0}T_{q}(\beta))$ and $|n-(x/q_{0})|>x/(q_{0}T_{q}(\beta))$, one may deduce from the above estimate and Shiu's bound that the error term is $O(x(\log x)^{O(1)}d_l(q_0)/(q_{0}T_{q}(\beta)))$. Consequently,
\begin{align*}
\sum_{\frac{x}{q_0}<n \leq \frac{x}{q_0}+u} F_{q_1} (n) n^{2\pi i \beta x} =& \frac{1}{2\pi } \int_{-T_{q}(\beta)}^{T_{q}(\beta)}\Bigbrac{ \sum_{n\sim \frac{x}{q_{0}}} \frac{F_{q_{1}}(n) n^{2\pi ix\beta}}{n^{1/2+it}} }\frac{(\frac{x}{q_{0}}+u)^{1/2+it}-\big(\frac{x}{q_{0}}\big)^{1/2+it}}{ 1/2+it}\rd t\\
&+ O(Hd_l(q_0)q_{0}^{-1}(\log x)^{-B_{q}}).	
\end{align*}
Recalling \eqref{eq:f-dircht-poly-exp-sum} and \eqref{eq:function-F}, the main term in the above equation equals
\begin{align*}
   \frac{1}{\phi(q_{1})}\sum_{\chi \mod {q_{1}}} C(a,\chi)\frac{1}{2\pi } \int_{-T_{q}(\beta)}^{T_{q}(\beta)}D_f\Big(\frac12+i(t-2\pi \beta x),\chi\Big) \frac{(\frac{x}{q_{0}}+u)^{1/2+it}-\big(\frac{x}{q_{0}}\big)^{1/2+it}}{ 1/2+it}\rd t.
\end{align*}
We next observe that whenever $\text{Re}(s)=1/2$ then
\begin{align*}
\Bigabs{  \frac{(\frac{x}{q_{0}}+u)^{s}-\big(\frac{x}{q_{0}}\big)^{s} }{s}}\ll \frac{q_{0}^{1/2}u}{x^{1/2}}.
\end{align*}
Therefore, making the change of variables $\Im s- 2\pi x\beta \to \Im s$, and combining the bound  $|C(a,\chi)|\leq q_{1}^{1/2 }$ with the previous estimate yields
\begin{align*}
\biggabs{ \sum_{\frac{x}{q_0}<n \leq \frac{x}{q_0}+u} F_{q_1} (n) n^{2\pi i \beta x}  } \ll &\frac{q_{0}^{1/2}u(\log x)}{(q_{1}x)^{1/2}} \sum_{\chi \mod {q_{1}}} \frac{1}{2\pi i} \int_{2\pi x\beta-T_{q}(\beta)}^{2\pi x\beta+T_{q}(\beta)} \Bigabs{D_f\Big(\frac12+it,\chi\Big)}\rd t
\\
&+ Hd_l(q_0)q_{0}^{-1}(\log x)^{-B_{q}}.	
\end{align*}

This completes the proof of the proposition.
\end{proof}

\subsection{When $|\beta|$ is large}

In order to proceed with the proof for large $\beta$, we first present a lemma concerning the application of the stationary phase method.

\begin{lemma}[van der Corput]\label{lem1}
	Let $g\in \mathcal C^4[a,b]$ be a real-valued function. Suppose that there exists a number $x_0\in[a,b]$ such that $g'(x_0)=0$. Assume that there exist positive quantities $\lambda_2,\lambda_3,\lambda_4>0$ such that
\[
g''(x) \leq -\lambda_2 \qquad |g^{(3)}(x)| \leq \lambda_3 \qquad |g^{(4)}(x)| \leq \lambda_4
\]
hold uniformly for all $x\in [a,b]$. Then
\[
\int_a^b e(g(x)) dx =\frac{e(-\frac{1}{8}+g(x_0))}{|g''(x_0)|^{1/2}} +O(R_1+R_2),
\]	
where
\[
R_1=\min \Big\{\frac{1}{\lambda_2(x_0-a)},\lambda_2^{-1/2}\Big\} + \min \Big\{\frac{1}{\lambda_2(b-x_0)},\lambda_2^{-1/2}\Big\}
\]
and
\[
R_2=\frac{(b-a)\lambda_4}{\lambda_2^2} +\frac{(b-a)\lambda_3^2}{\lambda_2^3}.
\]

\end{lemma}

\begin{proof}
	See \cite[Lemma 3.4]{GK-vander}.
\end{proof}

We shall employ the above lemma to approximate generic exponential sums by oscillatory integrals.
\begin{lemma}\label{lem:large-beta}
 Let $1 \leq H \leq x$ and $x/H^2\leq \beta \leq x^{\varepsilon}$. Suppose that $f:\Z \to \R $ is a function such that $|f(n)| \leq d_l(q) d_l(n)$ for some bounded $l\in \Z$ and $q\leq H$. Then one has
 \begin{align*}
 	\sum_{x<n \leq x+H} f(n) e(n\beta) =&\frac{e^{i\pi /4}}{2\pi\sqrt{\beta x}} \int_{2\pi \beta x}^{2\pi \beta(x+H)} \Bigbrac{\sum_{x<n \leq x+H} f(n) n^{it}} e^{it-it \log\frac{t}{2\pi \beta}} \rd t\\
 	&+O\Bigbrac{(\log x)^{O_l(1)} d_l(q) \Bigbrac{\sqrt{x/\beta} +  H^2/x}}.
 \end{align*}	
\end{lemma}

\begin{proof}

Taking $2\pi g(t) = t- t \log (t/2\pi \beta n)$, it is apparent that
\[
2\pi g'(t) = -\log  (t/2\pi \beta n), \quad g''(t) =-\frac{1}{2\pi t}, \quad g^{(3)}(t) =\frac{1}{2\pi t^2}, \quad g^{(4)}(t) =-\frac{1}{\pi t^3}.
\]
Clearly, the stationary point is $2\pi \beta n$ and it always lies on the interval $[2\pi \beta x, 2\pi \beta (x+H)]$ whenever $x\leq n\leq x+H$. Thus, applying Lemma~\ref{lem1}, and noting $H\leq x$, it transpires that
\begin{align*}
\int_{2\pi \beta x}^{2\pi \beta(x+H)} e^{it-it\log\frac{t}{2\pi \beta n}}\rd t= &\frac{2 \pi e(\beta n)\sqrt{\beta n}}{e^{i\pi/4}} + O\Bigbrac{ \min \set{\frac{x}{|n-x|}, \sqrt{\beta x}}}\\
& + O\Bigbrac{ \min \set{\frac{x+H}{|n-(x+H)|}, \sqrt{\beta x}}+1}.
\end{align*}

We deal with the first term on the right-hand side first. Taylor expansion gives $\sqrt{\beta n} =\sqrt{\beta x}(1+O(H/x))$ whenever $x\leq n\leq x+H$, and hence
\[
\sum_{x<n \leq x+H}  f(n) e(\beta n) \sqrt{\beta n} =\sqrt{\beta x} \sum_{x<n \leq x+H}  f(n) e(\beta n) +O(H^2d_l(q) \sqrt{\beta/x}(\log x)^{O_l(1)}),
\]
where we have used the assumption $|f(n)|\leq d_l(q) d_l(n)$ and Shiu's bound to estimate the error term. Dividing by $\sqrt{\beta x}$ shows that this error is acceptable. Moreover, interchanging the sum and integral gives 
\[
\sum_{x<n \leq x+H} f(n) \int_{2\pi \beta x}^{2\pi \beta(x+H)} e^{it-it\log\frac{t}{2\pi \beta n}}\rd t = \int_{2\pi \beta x}^{2\pi \beta(x+H)} \Bigbrac{\sum_{x<n \leq x+H} f(n) n^{it}}e^{it-it\log\frac{t}{2\pi \beta }} \rd t.
\]
Multiplying this expression by $\frac{e^{i\pi /4}}{2 \pi \sqrt{\beta x}}$ then gives the main term.
Therefore, to conclude the lemma, it suffices to show that
\[
(\beta x)^{-1/2}\sum_{x<n \leq x+H} |f(n)|\min \set{\frac{x}{|n-x|}, \sqrt{\beta x}}\ll d_l(q)  (\log x)^{O_l(1)} \sqrt{\frac{x}{\beta}},
\] 
and similarly for the analogous bound with $|n-(x+H)|$ in place of $|n-x|$. Actually,  when $|n-x|\leq (x/\beta)^{1/2}$,  the desired bound follows from Shiu's bound, while when $|n-x|\geq (x/\beta)^{1/2}$, a dyadic decomposition and Shiu's bound give\[
\sum_{x+\sqrt{\frac{x}{\beta}}<n \leq x+H} \frac{|d_l(q) d_l(n)|}{n-x}\ll d_l(q) \log x  \sup_{ \sqrt{\frac{x}{\beta}}\leq U \leq H}U^{-1}\sum_{x+U<n\leq x+2U} |d_l(n)|\ll d_l(q)(\log x)^{O_l(1)}.
\]
Here the assumption $x/H^2 \leq \beta $ ensures that  $\sqrt{x/\beta}\leq H$. Therefore, the lemma follows.
\end{proof}

Equipped with the above lemma, we are now prepared to present our argument for the case of large $\beta$.

\begin{proposition}\label{pro:large-beta}
 Let $1\leq H\leq x^{1-\eps }$. Let $f:\Z\to\R$ be a divisor-bounded  function and let $M_f(\alpha)$ be defined as in~\eqref{eq:f-dircht-poly-exp-sum}. 
	Let $\alpha =\frac aq +\beta \in \minor$ with $x(\log x)^{A/3}/H^2\leq \beta \leq (\log x)^{-2A}$. Then one has
	\[
	M_f(\alpha) \ll \frac{(\log x)^{O(1)}}{\sqrt {q\beta}}\sum_{q=q_0q_1 \atop q_0\leq \exp((\log x)^{1/2})} \sum_{\chi \mod {q_1}} \int_{2\pi \beta (x-2H)}^{2\pi \beta (x+2H)}|D_f(1/2+it,\chi)| \rd t + H (\log x)^{-A/7}.
	\]

\end{proposition}

\begin{proof}

As in the proof of Proposition~\ref{prop:small-beta}, it suffices to estimate \eqref{eq:function-F}. Applying Lemma~\ref{lem:large-beta} to $F_{q_1}$ with $\beta'=q_0\beta$, as well as $x'=x/q_0$ and $H'=H/q_0$ entails, when $\beta \geq x/H^2$, that 
\begin{align*}
\sum_{\frac{x}{q_0}< n \leq \frac{x+H}{q_0}} F_{q_1}(n)e(nq_0\beta) \ll &\frac{1}{\sqrt{\beta x}} \int_{2\pi \beta x}^{2\pi \beta (x+H)} \Bigabs{\sum_{\frac{x}{q_0}< n \leq \frac{x+H}{q_0}} \frac{F_{q_1}(n)}{n^{-it}}} \rd t\\
& + (\log x)^{O(1)} \frac{d_l(q_0)}{q_0}\Bigbrac{ \sqrt{x/\beta} +\frac{H^2}{x}}.	
\end{align*}
We then sum the second term in the above line over the relevant factorisations. Since $H\leq x^{1-\eps}$ and $\beta\geq x(\log x)^{A/3}/H^2$, it follows that $\sqrt{x/\beta}\ll H(\log x)^{-A/6}$. Therefore, using the bound $\sum_{q_0\mid q}d_l(q_0)/q_0\ll(\log q)^{O(1)}$ we obtain that the contribution to $M_f(\alpha)$ of the above error term is of size $O(H(\log x)^{-A/7})$, which is acceptable. It thereby remains to  bound the first term. 

Similarly to the proof of Proposition~\ref{prop:small-beta}, set $W=\beta H$. It then follows from the truncated Perron's formula that
\[
\Bigabs{\sum_{\frac{x}{q_0}< n \leq \frac{x+H}{q_0}} \frac{F_{q_1}(n)}{n^{-it}}}\ll \int_{\frac12-iW}^{\frac12+iW} \Bigabs{\sum_{n \sim x/{q_{0}}} \frac{F_{q_1}(n)}{n^{s-it}}} \Bigabs{\frac{(\frac{x+H}{q_0})^s-(\frac{x}{q_0})^s}{s}} |\rd s|+ \frac{x(\log x)^{O(1)}d_l(q_0)}{q_0\beta H}.
\]
It transpires that the contribution to $M_f(\alpha)$ from the second term is at most 
\[
\sum_{q=q_0q_1}\frac{\beta H}{\sqrt{\beta x}}  \frac{x(\log x)^{O(1)}d_l(q_0)}{q_0\beta H}\ll \frac{x^{1/2}(\log x)^{O(1)}}{\beta^{1/2}}\ll H(\log x)^{-A/7},
\]
where the last inequality follows from the assumption on $\beta$. Moreover, recalling (\ref{eq:f-dircht-poly-exp-sum}), the integral term is bounded by
\[
\ll q^{-1/2}(\log x)^{O(1)} \sum_{\chi \mod{q_1}} \int_{-W}^W \Big|D_f\Big(\frac{1}{2}+i(w-t),\chi\Big)\Big| \min\set{\frac{H}{\sqrt{x}},\frac{x^{1/2}}{1+|w|}} \rd w.	
\]
Therefore, it is sufficient to prove that
\begin{multline}\label{eq:claim-5.4}
\frac{1}{\sqrt{\beta x}} \int_{2\pi \beta x}^{2\pi \beta (x+H)} \int_{-W}^W \Big|D_f\Big(\frac{1}{2}+i(w-t),\chi\Big)\Big| \min\set{\frac{H}{\sqrt{x}},\frac{x^{1/2}}{1+|w|}} \rd w\rd t\\
\ll \frac{\log x}{\sqrt\beta} \int_{2\pi \beta (x-2H)}^{2\pi \beta (x+2H)}|D_f(1/2+it,\overline{\chi})| \rd t.
	\end{multline} 
Actually, since $\beta H\geq x/H$ the contribution from $|w|\leq x/H$ is bounded by
\[\frac{H}{x\sqrt \beta} \int_{2\pi \beta x}^{2\pi \beta (x+H)} \int_{-x/H}^{x/H}\Big|D_f\Big(\frac{1}{2}+i(w-t),\chi\Big)\Big| \rd w\rd t\ll \frac{1}{\sqrt\beta}\int_{2\pi \beta (x-2H)}^{2\pi \beta (x+2H)}|D_f(1/2+it,\overline{\chi})| \rd t;\]
whilst, since $W=\beta H$, the contribution from $x/H\leq |w|\leq W$ can be bounded by
\begin{multline*}\frac{\log x}{\sqrt\beta}\sup_{x/H \leq U \leq \beta H} \int_{2\pi \beta x}^{2\pi \beta (x+H)} \int_{U}^{2U}\frac{|D_f(\frac{1}{2}+i(w-t),\chi)|}{U} \rd w\rd t\\\ll \frac{\log x}{\sqrt\beta}\int_{2\pi \beta (x-2H)}^{2\pi \beta (x+2H)}|D_f(1/2+it,\overline{\chi})| \rd t.	\end{multline*} Combining the previous estimates yields \eqref{eq:claim-5.4}, as desired.
\end{proof}


When $f=\mu_{\boldsymbol N}$, we shall rewrite $D_f(s,\chi)$ as the product of distinct simpler Dirichlet polynomials, the details of which are presented in the following lemma.

\begin{lemma}[Products of Dirichlet polynomials]\label{lem: produ-diri-poly}
Let $\mu_{\boldsymbol N}$ be as in \eqref{mobius-n}.
Let $\chi$ be a Dirichlet character of conductor $q_1$. 
If $q_0\leq \exp((\log x)^{1/2})$, then
\begin{align}\label{Dss}
D_{\mu_{\boldsymbol N}}(s,\chi) = \sum_{q_Rq_1'q_2\cdots q_{2k-1}=q_0}F_{\mathbf q,\mu_{\boldsymbol N}}(s,\chi);
\end{align}
where 
\begin{equation*}F_{\mathbf{q},\mu_{\boldsymbol N}}(s,\chi)=P(s,\chi)R_{\mathbf{q}}(s,\chi)N_{1,\mathbf{q}}(s,\chi)\cdots N_{2k-1,\mathbf{q}}(s,\chi),\end{equation*} 
and $$P(s,\chi)=\sum_{p\in (P_{1},2P_{1}]}\chi(p)p^{-s},$$
$$  R_{\mathbf{q}}(s,\chi)=\sum_{ r\in \big(\frac{R}{q_{R}},2\frac{R}{q_{R}}\big] \atop (r,q_0/q_R)=1} \frac{a_{rq_{R}}\chi (r)}{r^{s}},\qquad N_{i,\mathbf{q}}(s,\chi)=\sum_{n_{i}\in \big(\frac{N_{i}}{q_{i}},2\frac{N_{i}}{q_{i}}\big] \atop \big(n_i,\frac{q_0}{q_Rq_1'\cdots q_i}\big)=1} \frac{a_{i}(q_{i}n_{i})\chi (n_{i})}{n_{i}^{s}}. $$

\end{lemma}

\begin{proof}
For convenience, set $g(n)=\chi(n)n^{-s}$. Obviously, it is completely multiplicative. Switching the order of summation yields
\begin{align*}
D_{\mu_{\boldsymbol N}}(s,\chi)&= \sum_{n}g(n)\sum_{\substack{q_{0}n=prm_{1}\cdots m_{2k-1}\\ p\in (P_{1},2P_{1}]\ r\in (R,2R], m_{i}\in (N_{i},2N_{i}]}}a_{r}a_{1}(m_{1})\cdots a_{2k-1}(m_{2k-1})\\
&= \sum_{ \substack{p\in (P_{1},2P_{1}]\ r\in (R,2R], m_{i}\in (N_{i},2N_{i}] \\ q_0 \mid prm_1\cdots m_{2k-1}}} a_ra_1(m_1)\cdots a_{2k-1}(m_{2k-1}) g\Big( \frac{prm_1\cdots m_{2k-1}}{q_0} \Big)\\
&=\sum_{p\sim P_1} g(p) \sum_{ \substack{r\in (R,2R], m_{i}\in (N_{i},2N_{i}] \\ q_0 \mid rm_1\cdots m_{2k-1}}} a_ra_1(m_1)\cdots a_{2k-1}(m_{2k-1}) g\Big( \frac{rm_1\cdots m_{2k-1}}{q_0} \Big),
\end{align*}
 the last equation holding since $g$ is completely multiplicative, and since in view of the assumption that $q_0\leq \exp((\log x)^{1/2})$ then for any $p\in[P_1,2P_1]$, we always have $p>\exp((\log x)^{1/2})$, and hence $p\nmid q_{0}$. Therefore, we can replace the condition $q_0 \mid prm_1\cdots m_{2k-1}$ by $q_0 \mid rm_1\cdots m_{2k-1}$. 

Set $q_R=(r,q_0)$, $q_1'=(m_1,q_0/q_R),$ and $q_{i}=(m_i,q_0/q_R q_1'\cdots q_{i})$ whenever $2\leq i\leq 2k-1$. It then transpires that $q_0 \mid rm_1\cdots m_{2k-1}$ if and only if $q_R q_1'q_{2}\cdots q_{2k-1}=q_{0}$, and that there is a unique choice of the tuple $(q_R,q_1'\dots,q_{2k-1})$. After replacing $r/q_R$ by $r$ and $m_i/q_i$ by $n_i$, and using the completely multiplicativity of $g$, the preceding expression may be rewritten as
\begin{align*}
D_{\mu_{\boldsymbol N}}(s,\chi)&= \sum_{q_0=q_Rq_1' \cdots q_{2k-1}}\sum_{p\in (P_{1},2P_{1}]} g(p) \sum_{r\in \big(\frac{R}{q_{R}},2\frac{R}{q_{R}}\big] \atop (r,q_0/q_R)=1} a_{rq_R}g(r)\prod_{i\leq 2k-1} \sum_{ \substack{n_{i}\in \big(\frac{N_{i}}{q_{i}},2\frac{N_{i}}{q_{i}}\big] \\ (n_i, \frac{q_0}{q_R q_1'\cdots q_i})=1} }  a_i(q_in_i) g(n_i).
\end{align*}
This gives \eqref{Dss} and completes the proof of the lemma.
\end{proof}

\section{Type I and type II estimates}\label{sec:typei-ii}\label{sec6}

For convenience, set $N_{2k}=P_1$ and $N_{2k+1}=R$ throughout the rest of the memoir. In this section, we take $Q =\frac{H^2}{x (\log x)^{A/3}}$  in~\eqref{eq:dirich-app}. By Dirichlet approximation, for every $\alpha \in \T$ there exist $a\in\mathbb{Z}$ and $q\in\mathbb{N}$ with $(a,q)=1$ and  $0\leq a\leq q\leq \frac{H^2}{x (\log x)^{A/3}}$	 such that 
\begin{equation}\label{eq6.1}
|\beta|=\bigabs{\alpha-a/q} \leq \frac{x(\log x)^{A/3}}{qH^2}.
\end{equation}

\begin{lemma}[Type II estimates]\label{lem:type-ii} 
Let $x^{1/2+\eps }\leq H \leq x^{1-\eps }$. Let $\mu_{\boldsymbol N}$ be defined as in \eqref{mobius-n}. If there exists a disjoint partition $J_1\sqcup J_2=\set{1,\dots,2k+1}$ such that 
\[
M_1:=\prod_{j\in J_1}N_j\ll Hx^{-\varepsilon^{2}}\ \ \ \ \text{and}\ \ \ \ \  M_2:=\prod_{j\in J_2} N_j \ll Hx^{-\eps^2},
\]
then one has
\[
\sup_{\alpha\in \minor} \bigabs{\sum_{x<n \leq x+H} \mu_{\boldsymbol N} e(n\alpha)} \ll H(\log x)^{-A/100}.
\]	
\end{lemma}

\begin{proof}
 We assume first that $q> (\log x)^{A}$, which in view of (\ref{eq6.1}) entails that $\lvert \beta\rvert\leq x/H^{2}$. By conjugating, and hence replacing $\beta$ with $-\beta$ if necessary, we may assume that $\beta$ is non-negative. The application of Proposition~\ref{prop:small-beta} (2) with $f=\mu_{\boldsymbol N}$ shows that it suffices to prove
\begin{align}\label{eq:typeii-large-q}
\twosum{q=q_0q_1}{ q_0\leq \exp\bigbrac{ (\log x)^{1/2}}} \sum_{\chi \mod{q_1}} \int_{2\pi \beta x-T_{q}(\beta)}^{2\pi \beta x+T_{q}(\beta)} \bigabs{ D_{\mu_{\boldsymbol N}} (1/2+it,\chi)} \rd t \ll (qx)^{1/2} (\log x)^{-A/10}.
\end{align}

For each factorization of $q_0$ as in the sums in~\eqref{Dss} given by the components of $\mathbf{q}=(q_R,q_1',q_2,\cdots q_{2k-1})$ we also write for the sake of concision
$$N_{2k,\mathbf{q}}(s,\chi)=P(s,\chi),\  \ \ \ \ \  \ \ \ \ \  \ \ \ \ \ \ \ \ N_{2k+1,\mathbf{q}}(s,\chi)=R(s,\chi),$$ and introduce
$$M_{1}(s,\chi)=\prod_{j\in J_{1}} N_{j,\mathbf{q}}(1/2+it,\chi),\  \ \  \ \ \ \ \ \ M_{2}(s,\chi)=\prod_{j\in J_{2}} N_{j,\mathbf{q}}(1/2+it,\chi).$$
It then follows from Lemma~\ref{lem:mean-value}, Lemma~\ref{lem: produ-diri-poly} and Cauchy-Schwarz inequality that
\begin{align*}
&\twosum{q=q_0q_1}{ q_0\leq \exp\bigbrac{ (\log x)^{1/2}}} \sum_{\chi \mod{q_1}}\int_{2\pi \beta x-T_{q}(\beta)}^{2\pi \beta x+T_{q}(\beta)} \bigabs{ D_{\mu_{\boldsymbol N}} (1/2+it,\chi)} \rd t \\
&\ll \sum_{q=q_0q_1} \sum_{ q_R q_1'q_2\cdots q_{2k-1}=q_0} \sum_{\chi \mod{q_1}} \int_{2\pi \beta x-T_{q}(\beta)}^{2\pi \beta x+T_{q}(\beta)}\bigabs{M_1(1/2+it,\chi)} \bigabs{M_2(1/2+it,\chi)} \rd t\\
&\ll  (\log x)^{A/49} \twosum{q_0q_1=q}{r_1r_2=q_0} d_{2k+1}(r_1)d_{2k+1}(r_2) \Bigbrac{\frac{q_1x}{H} +\frac{M_1}{r_1}}^{1/2} \Bigbrac{\frac{q_1x}{H} +\frac{M_2}{r_2}}^{1/2}.	
\end{align*}
Using the assumptions $M_1,M_2\ll H x^{-\eps^2}$, $M_1M_2\asymp x$ and $H\geq x^{1/2+\eps }$, it transpires that
\begin{align*}
&\twosum{q=q_0q_1}{ q_0\leq \exp\bigbrac{ (\log x)^{1/2}}} \sum_{\chi \mod{q_1}} \int_{2\pi \beta x-T_{q}(\beta)}^{2\pi \beta x+T_{q}(\beta)} \bigabs{ D_{\mu_{\boldsymbol N}} (1/2+it,\chi)} \rd t \\
\ll &\sqrt{qx} (\log x)^{A/49}\Bigbrac{ 
\frac{\sqrt{qx}}{H}\sum_{q_0\mid q}\frac{d_{4k+2}(q_0)}{q_0} 
+\frac{1}{\sqrt H}  \sum_{q_0\mid q}\frac{d_{4k+2}(q_0)}{q_0^{1/2}} \Bigbrac{\max\set{M_1^{1/2},M_2^{1/2}} } }\\
&+x^{1/2}(\log x)^{A/49}\sum_{q_0\mid q}\frac{d_{4k+2}(q_0)}{q_0}\ll (qx)^{1/2} (\log x)^{-A/10},
\end{align*}
where in the last inequality we employed the conditions $q\leq H^2/x(\log x)^{A/3}$ and the estimate $\sum_{q_{0}|q} d_{4k+2}(q_ {0})q_{0}^{-1/2}\ll \exp((\log x)^{1/2})$. This gives the desired estimate~\eqref{eq:typeii-large-q} whenever $q> (\log x)^{A}$.

We shall assume next that $q\leq (\log x)^A$. For the sake of concission, we shall omit writing for each integer $q_0\mid q$ the condition $q_0\leq \exp((\log x)^{1/2})$ in the corresponding summation. We introduce for further use
\[
G_{\mathbf{q}}(s,\chi)=\prod_{j\in J_{1}\setminus\{2k\}}  N_{j,\mathbf{q}}\Big(\frac{1}{2}+it,\chi\Big)\prod_{j\in J_{2}\setminus\{2k\}} N_{j,\mathbf{q}}\Big(\frac{1}{2}+it,\chi\Big).
\]
We observe first that since $\alpha\in\minor$ and $q\leq (\log x)^{A}$ then $\beta\geq \frac{(\log x)^{5A}}{H}$. Consequently, \begin{equation}\label{jhjh}2\pi  \beta x+T_{q}(\beta)\geq 2\pi  \beta x-T_{q}(\beta)\geq \frac{x(\log x)^{5A}}{H}\geq (\log x)^{4A}.\end{equation} One then deduces from Lemmas \ref{lem:bd-diri-prime} and \ref{lem: produ-diri-poly} that
\begin{align*}
D_{\mu_{\boldsymbol N}}\Big(\frac{1}{2}+it,\chi\Big)&\ll \sum_{q_Rq_1'q_1 \cdots q_{2k-1}=q_0} \Bigabs{P\Big(\frac12+it,\chi\Big)} \Bigabs{G_{\mathbf{q}}\Big(\frac12+it,\chi\Big)}\\
&\ll P_1^{1/2}\Bigbrac{\frac{\log x}{1+|t|} + \exp\bigbrac{-(\log x)^{\eps_{0}/3}}} \sum_{q_Rq_1'q_1 \cdots q_{2k-1}=q_0} \Bigabs{G_{\mathbf{q}}\Big(\frac12+it,\chi\Big)}\\
&\ll  P_1^{1/2} (\log x)^{-3A} \sum_{q_Rq_1'q_1 \cdots q_{2k-1}=q_0} \Bigabs{G_{\mathbf{q}}\Big(\frac12+it,\chi\Big)}
\end{align*}
provided that $t\geq (\log x)^{4A}$, which is indeed the case in view of (\ref{jhjh}). Therefore, due to the above logrithmic-saving, one may deduce from Proposition~\ref{prop:small-beta} (1) that it suffices to prove the estimate
\begin{align}\label{eq:claim-typeii}
\twosum{q=q_0q_1}{q_Rq_1'q_1 \cdots q_{2k-1}=q_0} \sum_{\chi\mod{q_1}} \int_{2\pi \beta x-T_{q}(\beta)}^{2\pi \beta x+T_{q}(\beta)} \Bigabs{G_{\mathbf{q}}\Big(\frac12+it,\chi\Big)} \rd t\ll \Bigbrac{\frac{qx}{P_1}}^{1/2} (\log x)^{2A}.
\end{align}
 We assume without loss of generality that $2k\in J_1$. The application of Cauchy-Schwarz inequality and Lemma~\ref{lem:mean-value} gives
 \begin{align*}
 	&\twosum{q=q_0q_1}{q_Rq_1'q_1 \cdots q_{2k-1}=q_0} \sum_{\chi\mod{q_1}} \int_{2\pi \beta x-T_{q}(\beta)}^{2\pi \beta x+T_{q}(\beta)}  \Bigabs{G_{\mathbf{q}}\Big(\frac12+it,\chi\Big)} \rd t\\
 	&\ll (\log x)^{A} \twosum{q_0q_1=q}{r_1r_2=q_0} d_{2k+1}(r_1) d_{2k+1}(r_2) \Bigbrac{\frac{q_1x}{H} +\frac{M_1}{P_1r_1}}^{1/2} \Bigbrac{ \frac{q_1x}{H} + \frac{M_2}{r_2}}^{1/2}\\
 	&\ll (qx)^{1/2}q^{\varepsilon}(\log x)^{A}\Bigbrac{ \frac{(qx)^{1/2}}{H} + \frac{\max \set{(M_1/P_1)^{1/2}, M_2^{1/2}}}{H^{1/2}} +P_1^{-1/2} } \ll \bigbrac{\frac{qx}{P_1}}^{1/2} (\log x)^{2A},
 \end{align*}
where in the last line we used the fact that $P_1\leq \exp(\frac{\log x}{(\log \log x)^2})$ in conjunction with $M_1M_2\asymp x$, the condition $q\leq (\log x)^A$ and the assumptions $M_1,M_2\ll H x^{-\eps^2}$ and $H\geq x^{1/2+\eps }$. This gives~\eqref{eq:claim-typeii} and concludes the proof.
\end{proof}

\begin{lemma}[Type I estimate]\label{lem:typei}
	Let $x^{1/2+\eps }\leq H\leq x^{19/20}$. Let $\mu_{\boldsymbol N}$ be defined as in \eqref{mobius-n}. If there exists some index $j$ such that $N_j\geq x^{1+\eps}/H$, then
	\[
\sup_{\alpha\in \minor} \bigabs{\sum_{x<n \leq x+H} \mu_{\boldsymbol N} e(n\alpha)} \ll H(\log x)^{-A/100}.
\]
\end{lemma}

\begin{proof}
We may write
$$F_{\mathbf{q},\mu_{\boldsymbol N}}\Big(\frac{1}{2}+it,\chi\Big)=N_{j,\mathbf{q}}\Big(\frac{1}{2}+it,\chi\Big)\prod_{l\in \{1,\ldots,2k+1\}\setminus \{j\}}N_{l,\mathbf{q}}\Big(\frac{1}{2}+it,\chi\Big),$$
where the assumption $N_j\geq x^{1+\eps}/H \geq  x^{1/20+\eps }$ entails that the corresponding coefficient  is $a_j={\bf 1}_{(N_j,2N_j]}$,  and therefore
\[
N_{j,\mathbf{q}}\Big(\frac{1}{2}+it,\chi\Big) = \sum_{ \substack{ \frac{N_j}{q_j}< n\leq \frac{2N_j}{q_j}  \\ (n,\frac{q_0}{q_R \cdots q_j})=1}} \frac{\chi(n)}{n^{1/2+it}}.
\]
Using the identity ${\bf 1}_{(a,b)=1}= \sum_{l\mid (a,b)}\mu(l)$, it transpires that
\[
N_{j,\mathbf{q}}\Big(\frac{1}{2}+it,\chi\Big) =\sum_{l\mid \frac{q_0}{q_R \cdots q_j}} \mu(l) \sum_{ \substack{ \frac{N_j}{q_j}< n\leq \frac{2N_j}{q_j} \\ l\mid n}} \frac{\chi(n)}{n^{1/2+it}} = \sum_{l\mid \frac{q_0}{q_R \cdots q_j}} \frac{\mu(l) \chi(l)}{l^{1/2+it}} \sum_{\frac{N_j}{lq_j}< n\leq \frac{2N_j}{lq_j} }\frac{\chi(n)}{n^{1/2+it}}.
\]
In light of Proposition~\ref{prop:small-beta}, it then suffices to bound the expression
\begin{align}\label{eq:typei}
\twosum{q=q_0q_1}{q_0=q_R \cdots q_{2k-1}} \sum_{l\mid \frac{q_0}{q_R \cdots q_j}} \frac{1}{\sqrt l} \sum_{\chi \mod {q_1}} \int_{2\pi \beta x-T_{q}(\beta)}^{2\pi \beta x+T_{q}(\beta)}  \Bigabs{ \sum_{\frac{N_j}{lq_j}< n\leq \frac{2N_j}{lq_j} }\frac{\chi(n)}{n^{1/2+it}} \prod_{j'\in [2k+1]\backslash \set{j}} N_{j',\mathbf{q}} \Big(\frac12+it,\chi\Big)} \rd t,
\end{align}
where as above we may assume that $\beta$ is non-negative. Since $(2\pi \beta x)^{1/3} \leq (x/H)^{2/3+\eps } \leq x/H\leq T_{q}(\beta)$, Lemma~\ref{lem:second-moment} is applicable. 

We first suppose that $2\pi\beta x\geq 4T_{q}(\beta)$. Indeed, by writing $C_{q}=(\log x)^{A/49}$ if $q> (\log x)^{A}$ (and hence $\lvert \beta \rvert \leq x/H^{2}$) and $C_{q}=(\log x)^{A+3}$ when $q\leq(\log x)^{A}$ it follows that
\[
\sum_{\substack{\chi \mod {q_1}}} \int_{2\pi \beta x-T_{q}(\beta)}^{2\pi \beta x+T_{q}(\beta)}  \Bigabs{ \sum_{\frac{N_j}{lq_j}< n\leq \frac{2N_j}{lq_j} }\frac{\chi(n)}{n^{1/2+it}}}^2\rd t\ll (\log x)^{C_{q}} \Bigbrac{\frac{q_{1}x}{H} + \frac{N_{j}H}{lq_j x}},
\]
where we used the fact that $T_{q}(\beta)\gg x/H$. Thus, by the Cauchy--Schwarz inequality, the preceding estimate and Lemma~\ref{lem:mean-value}, the inner sum in~\eqref{eq:typei} is bounded above by
\begin{align*}
&  \Bigbrac{\sum_{\substack{\chi }}\int_{2\pi \beta x-T_{q}(\beta)}^{2\pi \beta x+T_{q}(\beta)}   \Bigabs{ \sum_{\frac{N_j}{lq_j}< n\leq \frac{2N_j}{lq_j} }\frac{\chi(n)}{n^{1/2+it}}}^2\rd t}^{1/2}  \Bigbrac{\sum_{\substack{\chi }}\int_{2\pi \beta x-T_{q}(\beta)}^{2\pi \beta x+T_{q}(\beta)}  \Bigabs{ \prod_{j'\in [2k+1]\backslash \set{j}} N_{j',\mathbf{q}} \Big(\frac12+it,\chi\Big)}^2 \rd t }^{1/2}\\
&\ll (\log x)^{C_{q}} \Bigbrac{\frac{q_1x}{H} + \frac{N_{j}H}{lq_j x}}^{1/2}  \Bigbrac{\frac{q_1x}{H} +\frac{x}{N_j}}^{1/2}.	
\end{align*}
Consequently, the expression in \eqref{eq:typei} is bounded by
\begin{align*}
&(\log x)^{C_{q}} \sum_{q=q_0q_1} d(q_{0})d_{4k+2}(q_{0}) \Bigbrac{\frac{q_1x}{H} + \frac{N_{j}H}{lx}}^{1/2}   \Bigbrac{\frac{q_1x}{H} +\frac{x}{N_j}}^{1/2}\\
\ll & \frac{qx(\log x)^{C_{q}}}{H} \sum_{q_0q_1=q} \frac{d(q_{0})d_{4k+2}(q_0)}{q_0} + (qx)^{1/2} \Big(\frac{x}{HN_j}\Big)^{1/2} (\log x)^{C_{q}} \sum_{q_0q_1=q} \frac{d(q_{0})d_{4k+2}(q_0)}{q_0^{1/2}}\\
&+ (qx)^{1/2}(\log x)^{C_{q}}\Big(\frac{N_{j}}{x}\Big)^{1/2} \sum_{q_0q_1=q} \frac{d(q_{0})d_{4k+2}(q_0)}{q_0^{1/2}} + H^{1/2} (\log x)^{C_{q}} \sum_{q_0q_1=q}d(q_{0})d_{4k+2}(q_0)\\
\ll& (qx)^{1/2}(\log x)^{C_{q}+O(1)}\Big((qx)^{1/2}/H +x^{-\eps/4}+\exp\big((\log x)^{1/2}\big)\Big(\frac{N_{j}}{x}\Big)^{1/2}+\frac{H^{1/2}}{x^{1/2}}\Big)
\\
\ll& (qx)^{1/2}(\log x)^{C_{q}+O(1)}\Big((qx)^{1/2}/H +\exp\Big(-\frac{1}{4}(\log x)^{2/3}\Big)\Big).
\end{align*}
It is worth noting that here we have used the assumptions $N_j\geq x^{1+\eps }/H$ and the restriction $H\leq x^{19/20}$. We have also employed the fact that by equation (\ref{NNN}), one has $N_{j}\leq x/P_{1}\ll x\exp(-(\log x)^{2/3})$, and that $q_{0}\leq \exp\big((\log x)^{1/2}\big).$ If $q>(\log x)^{A}$ then $C_{q}=A/49$, whence the bound $q\leq H^{2}/x(\log x)^{A/3}$ entails that the above is $O((qx)^{1/2}(\log x)^{-A/10})$, which yields the desired estimate. If instead $q\leq (\log x)^{A}$, the last line is $O((qx)^{1/2}\exp(-\frac{1}{4}(\log x)^{2/3}))$, as required.

If instead $2\pi\beta x< 4T_{q}(\beta)$ we note first that $q> (\log x)^{A}$, since if $q\leq (\log x)^{A}$ and $\alpha\in\minor$ one would have $\beta\geq (\log x)^{5A}/H$, and hence the above would not occur. Then, by making a dyadic dissection of $[2\pi \beta x-T_{q}(\beta),2\pi \beta x+T_{q}(\beta)]$, it suffices to consider the integral over intervals $[T_{1},2T_{1}]$ with $0\leq T_{1}\ll T_{q}(\beta)$.  Therefore, Lemma~\ref{lem:second-moment} delivers
\[
 \int_{T_{1}}^{2T_{1}}  \Bigabs{ \sum_{\frac{N_j}{lq_j}< n\leq \frac{2N_j}{lq_j} }\frac{\chi(n)}{n^{1/2+it}}}^2\rd t\ll (\log x)^{O(1)} \Bigbrac{q_{1}T_{1} + \frac{N_{j}}{lq_j (1+T_{1})}}.
\]
Consequently, combining Lemma \ref{lem:mean-value} with Cauchy-Schwarz and the above estimate yields
\begin{align*}
& \sum_{\chi \mod {q_1}} \int_{T_{1}}^{2T_{1}} \Bigabs{ \sum_{\frac{N_j}{lq_j}< n\leq \frac{2N_j}{lq_j} }\frac{\chi(n)}{n^{1/2+it}} \prod_{j'\in [2k+1]\backslash \set{j}} N_{j',\mathbf{q}} \Big(\frac12+it,\chi\Big)} \rd t\ll 
\\
&\Bigbrac{\sum_{\chi \mod {q_1}}\int_{T_{1}}^{2T_{1}}   \Bigabs{ \sum_{\frac{N_j}{lq_j}< n\leq \frac{2N_j}{lq_j} }\frac{\chi(n)}{n^{1/2+it}}}^2\rd t}^{1/2}  \Bigbrac{\sum_{\chi \mod {q_1}}\int_{T_{1}}^{2T_{1}}  \Bigabs{ \prod_{j'\in [2k+1]\backslash \set{j}} N_{j',\mathbf{q}} \Big(\frac12+it,\chi\Big)}^2 \rd t }^{1/2}\\
&\ll (\log x)^{O(1)} \Bigbrac{q_{1}T_{1}+ \frac{N_{j}}{lq_j (1+T_{1})}}^{1/2}  \Bigbrac{q_{1}T_{1} +\frac{x}{N_j}}^{1/2}.
\end{align*}

Therefore, an analogous argument as before in conjunction with the fact that $T_{1}\ll T_{q}(\beta)$ permits one to conclude that \eqref{eq:typei} is bounded by a constant times
\begin{align*} \ll& 	(\log x)^{O(1)} (qx)^{1/2}\Big((qx)^{1/2}/H +x^{-\eps/4}+\exp\big(-\frac{1}{4}(\log x)^{2/3}\big)\Big)
\\
&+(\log x)^{O(1)} \frac{q^{\varepsilon}x^{1/2}}{(1+T_{1})^{1/2}}.
\end{align*}

By the same argument as above, the first summand is $O((qx)^{1/2} (\log x)^{-A/10})$. Moreover, since $q>(\log x)^{A}$ one has $(\log x)^{O(1)}q^{\varepsilon}x^{1/2}(1+T_{1})^{-1/2}\ll q^{1/2}x^{1/2}(\log x)^{-A/3}$, as desired.

\end{proof}

\section{Type I$_2$ estimates}\label{sec7}

\begin{lemma}[Type I$_2$ estimates]\label{lem:typei2}

Let $x^{121/210+\eps }\leq H\leq x^{17/20-4\eps }$. Let $\mu_{\boldsymbol N}$ be defined as in \eqref{mobius-n}. If there exist two distinct indexes $j_1, j_2\in \set{1,\dots,2k+1}$  such that $N_{j_1},N_{j_2} \leq x^{1+\eps }/H$ and $N_{j_1}N_{j_2} \geq (x/H)^{4/3+\eps }$, then one has
\[
\sup_{\alpha\in \minor} \bigabs{\sum_{x<n \leq x+H} \mu_{\boldsymbol N} e(n\alpha)} \ll H(\log x)^{-A/4}.
\]	
\end{lemma}

\begin{proof}
We first remark that the $x^\eps$ loss in Lemma~\ref{lem:fourth-moment} prevents one from using the same Dirichlet approximation as in Section~\ref{sec:typei-ii}. We therefore take $Q=H^2/x^{1+\eps}$ in~\eqref{eq:dirich-app}, so that $1\leq q\leq H^2/x^{1+\eps}$ and
\begin{equation}\label{khkh}
q|\beta|=\norm{q\alpha}\leq x^{1+\eps}/H^2.
\end{equation}

As above, we may suppose that $\beta$ is non-negative. We start by observing that the assumptions imply $N_{j_1},N_{j_2}\geq x^{1/20+\eps}$, since otherwise, $(x/H)^{4/3+\eps } \leq N_{j_1}N_{j_2} \leq x^{1/20+\eps }(x/H)^{1+\eps }$, which contradicts the proviso $H\leq x^{17/20-4\eps}$. Thus the corresponding Dirichlet polynomial coefficients satisfy $a_{j_1}=a_{j_2}=1$. As in the proof of Lemma~\ref{lem:typei}, we may therefore assume that for $l=1,2$ then
\[
\Big| N_{j_{l},\mathbf{q}}\Big(\frac{1}{2}+it,\chi\Big)\Big|\ll \sum_{d_l \mid \frac{q_0}{q_R\cdots q_{j_l}}} d_l^{-1/2} \bigg| \sum_{n_{j_l}\sim\frac{N_{j_l}}{q_{j_l} d_l}} \frac{\chi(n_{j_l})}{n_{j_l}^{1/2+it}}\bigg|.
\]
Since the integration ranges in Propositions~\ref{prop:small-beta} and~\ref{pro:large-beta} depend on the size of $\beta$, we set for convenience
\[
F(\beta)=\begin{cases}T_{q}(\beta) \qquad &\text{ if } \beta\leq x(\log x)^{A/3}/H^2;\\
 	 2\beta H, & \text{ if } x(\log x)^{A/3}/H^2 < \beta\leq x^{1+\eps }/(qH^2).
 	 \end{cases}\]
By these propositions, it now suffices to prove 
\begin{align*}
& \sum_{q_1,\chi} \int_{2\pi \beta x-F(\beta)}^{2\pi \beta x+F(\beta)} \Bigabs{ N_{j_{1},\mathbf{q}}\Big(\frac{1}{2}+it,\chi\Big) N_{j_{2},\mathbf{q}}\Big(\frac{1}{2}+it,\chi\Big)\prod_{j'\in[2k+1]\backslash\set{j_1,j_2}}  N_{j',\mathbf{q}}\Big(\frac{1}{2}+it,\chi\Big)}\rd t\\
\ll& \sum_{q_1,\chi}  \sum_{\substack{d_1 \mid \frac{q_0}{q_R\cdots q_{j_1}}\\ d_2 \mid \frac{q_0}{q_R\cdots q_{j_2}}}} (d_1d_2)^{-1/2} \int_{2\pi \beta x-F(\beta)}^{2\pi \beta x+F(\beta)}\prod_{l=1,2} \bigg| \sum_{n_{j_l}\sim\frac{N_{j_l}}{q_{j_l} d_l}} \frac{\chi(n_{j_l})}{n_{j_l}^{1/2+it}}\bigg| \biggabs{ \prod_{j'\in[2k+1]\backslash\set{j_1,j_2}}  N_{j',\mathbf{q}}\Big(\frac{1}{2}+it,\chi\Big)}\rd t\\
\ll& \begin{cases}
 	\sqrt{qx}(\log x)^{-A}, \qquad&\text{ if $\beta\leq x(\log x)^{A/3}/H^2 $ and either $q\leq (\log x)^{A}$ or $ \beta >x/H^{2}$},\\
\sqrt{qx} (\log x)^{-A/10}, & \text{ if $q> (\log x)^{A}$ and $ \beta \leq x/H^{2}$},\\
 	 \sqrt{q \beta}H (\log x)^{-A}, & \text{ if } x(\log x)^{A/3}/H^2 < \beta\leq x^{1+\eps }/(qH^2),
 \end{cases}
 \end{align*}
where we used the symbol $\sum_{q_{1},\chi}$ to denote the sum $\sum_{q_0q_{1}=q}\sum_{q_{R}q_{1}'q_{2}\cdots q_{2k-1}=q_{0}}\sum_{\chi\mod{q_{1}}}$.

We first assume that $2\pi\beta x\geq 4F(\beta)$. Since $N_{j_l}\leq x^{1+\eps }/H$, it therefore transpires that $N_{j_l}^{2}(\beta x)^{-3}\ll 1\leq q_1 F(\beta)$. Consequently, applying Lemma~\ref{lem:fourth-moment} we obtain for $l=1,2$ the estimate
\begin{align}\label{eqz}
\sum_{\chi \mod {q_1}} \int_{2\pi \beta x-F(\beta)}^{2\pi \beta x+F(\beta)} \bigg| \sum_{n_{j_l}\sim\frac{N_{j_l}}{q_{j_l} d_l}} \frac{\chi(n_{j_l})}{n_{j_l}^{1/2+it}}\bigg|^4 \rd t& \ll x^{\eps/100 } \bigbrac{q_{1}F(\beta) +q_1^{\frac{71}{96}}(\beta x)^{2/3}} +\frac{N_{j_l}^2}{(\beta x)^{3}}\\
& \ll x^{\eps/100 } \bigbrac{q_1F(\beta) +q_1^{\frac{71}{96}}( \beta x)^{2/3}}\nonumber.
\end{align}
We then deduce from H\"older's inequality, the above estimate and Lemma~\ref{lem:mean-value} that
\begin{align}\label{chi1}
& \sum_{q_1,\chi}  \sum_{\substack{d_1 \mid \frac{q_0}{q_R\cdots q_{j_1}}\\ d_2 \mid \frac{q_0}{q_R\cdots q_{j_2}}}} (d_1d_2)^{-1/2} \int_{2\pi \beta x-F(\beta)}^{2\pi \beta x+F(\beta)}\prod_{l=1,2} \bigg| \sum_{n_{j_l}\sim\frac{N_{j_l}}{q_{j_l} d_l}} \frac{\chi(n_{j_l})}{n_{j_l}^{1/2+it}}\bigg| \biggabs{ \prod_{j'\in[2k+1]\backslash\set{j_1,j_2}}  N_{j',\mathbf{q}}\Big(\frac{1}{2}+it,\chi\Big)}\rd t\\
\ll& x^{\eps/100 } \sum_{q=q_0q_1} d(q_{0})^{2}d_{6k+3}(q_0) \bigbrac{ q_1F(\beta) +q_1^{\frac{71}{96}}( \beta x)^{2/3}}^{1/2} \Bigbrac{q_1F(\beta) + \frac{x}{N_{j_1} N_{j_2}}}^{1/2}.\nonumber
\end{align}
Suppose first that $\beta \leq x(\log x)^{A/3}/H^2$, and hence $F(\beta)=T_{q}(\beta)$. Since $N_{j_1} N_{j_2}\gg (x/H)^{4/3+\eps }$ and since (\ref{khkh}) holds, the preceding expression is bounded by
\begin{align}\label{jgjg}
\ll&x^{\eps/90} (qx)^{1/2} \Bigbrac{ \frac{(qx)^{1/2}}{H} + \frac{x^{1/2}}{(HN_{j_1} N_{j_2})^{1/2}} + \frac{q^{\frac{7}{192}} (q\beta x)^{1/3}}{H^{1/2}} + \frac{q^{-\frac{89}{192}} (q\beta x)^{1/3} }{ (N_{j_1} N_{j_2})^{1/2}} }\nonumber\\
\ll & x^{\eps/90} (qx)^{1/2} \Bigbrac{ \frac{(qx)^{1/2}}{H} + \frac{(x/H)^{1/2}}{(x/H)^{2/3}} + \frac{q^{7/192}x^{2/3+\eps}}{H^{7/6}} +\frac{x^{2/3}}
{H^{2/3}(x/H)^{2/3+\eps/2}}}\nonumber\\
\ll & x^{-\eps /90}	 (qx)^{1/2},
\end{align}
the last inequality following from the assumptions $q\leq H^2/x^{1+\eps}$ and $H\geq x^{121/210+\eps }$. This gives the desired bound.

Suppose instead that $x(\log x)^{A/3}/H^2< \beta \leq x^{1+\eps}/qH^2$, and hence $F(\beta)=2\beta H$. Using this upper bound on $\beta$, it follows that (\ref{chi1}) is bounded above by
\begin{align*}
&\ll x^{\eps/90} (q\beta)^{1/2} H \Bigbrac{ (q\beta)^{1/2} +\sqrt{\frac{x}{HN_{j_1}N_{j_2}}} +\frac{q^{7/192}(q\beta x)^{1/3}}{H^{1/2}} +\frac{q^{-\frac{89}{192}}(q\beta x)^{1/3}}{(N_{j_1}N_{j_2})^{1/2}}}	\\
&\ll x^{-\eps/90 }(q\beta)^{1/2} H,
\end{align*}
where we used the estimate (\ref{jgjg}) to bound the last three summands and the condition $H\geq x^{1/2+\varepsilon}$.

We next assume that $2\pi\beta x<4F(\beta)$. Note that then $q> (\log x)^{A}$, since $\alpha\in\minor$ and hence if $q\leq (\log x)^{A}$ one would have $\beta\geq (\log x)^{5A}/H$ and the above bound would not occur. Moreover, it is apparent that $\beta\leq x/H^{2}$, since otherwise $2\pi \beta x>F(\beta)$. By making a dyadic dissection of $[2\pi \beta x-F(\beta),2\pi \beta x+F(\beta)]$, in order to estimate (\ref{eqz}) it suffices to consider the integral over intervals $[T_{1},2T_{1}]$ with $0\leq T_{1}\ll F(\beta)$. Indeed, an application of Lemma \ref{lem:fourth-moment} enables one to obtain for $l=1,2$ the estimate

\begin{align*}
\sum_{\chi \mod {q_1}} \int_{T_{1}}^{2T_{1}}  \bigg| \sum_{n_{j_l}\sim\frac{N_{j_l}}{q_{j_l} d_l}} \frac{\chi(n_{j_l})}{n_{j_l}^{1/2+it}}\bigg|^4 \rd t&\ll (\log x)^{O(1)}q_{1}T_{1} +\frac{N_{j_l}^2}{(1+T_{1}^{3})}.
\end{align*}
We then deduce from H\"older's inequality, the above estimate and Lemma~\ref{lem:mean-value} that
\begin{align*}
& \sum_{q_1,\chi}\sum_{\substack{d_1 \mid \frac{q_0}{q_R\cdots q_{j_1}}\\ d_2 \mid \frac{q_0}{q_R\cdots q_{j_2}}}} (d_1d_2)^{-1/2}  \int_{T_{1}}^{2T_{1}}\prod_{l=1,2}  \bigg| \sum_{n_{j_l}\sim\frac{N_{j_l}}{q_{j_l} d_l}} \frac{\chi(n_{j_l})}{n_{j_l}^{1/2+it}}\bigg| \biggabs{ \prod_{j'\in[2k+1]\backslash\set{j_1,j_2}}  N_{j',\mathbf{q}}\Big(\frac{1}{2}+it,\chi\Big)}\rd t\\
\ll& q^{\eps/10}(\log x)^{O(1)}  \Bigbrac{ q_{1}T_{1}+\frac{N_{j_1}^2}{(1+T_{1}^{3})}}^{1/4}\Bigbrac{ q_{1}T_{1}+\frac{N_{j_2}^2}{(1+T_{1}^{3})}}^{1/4} \Bigbrac{q_1T_{1}+ \frac{x}{N_{j_1} N_{j_2}}}^{1/2}
\\
\ll &q^{\varepsilon/10}(\log x)^{O(1)}\Big(qT_{1}+ \frac{(qN_{j_1} N_{j_2})^{1/2}}{(1+T_{1})}+q^{3/4}\max(N_{j_1} ^{1/2},N_{j_2}^{1/2})+(qT_{1})^{1/2}\Big(\frac{x}{N_{j_1} N_{j_2}}\Big)^{1/2}
\\
&+\frac{x^{1/2}}{1+T_{1}^{3/2}}+q^{1/4}\max(N_{j_1} ^{1/2},N_{j_2}^{1/2})\Big(\frac{x}{(1+T_{1})N_{j_1} N_{j_2}}\Big)^{1/2}\Big)
\\\ll &q^{\varepsilon/10}(\log x)^{O(1)}\Bigg(qT_{q}(\beta)+ (qN_{j_1} N_{j_2})^{1/2}+q^{3/4}\Big(\frac{x}{H}\Big)^{1/2+\varepsilon}
\\
&+\Big((qT_{1})^{1/2}+q^{1/4}\Big(\frac{x}{H}\Big)^{1/2+\varepsilon}\Big)\Big(\frac{x}{N_{j_1} N_{j_2}}\Big)^{1/2}+\frac{x^{1/2}}{1+T_{1}^{3/2}}\Bigg),
\end{align*}
where we used that $N_{j_1}, N_{j_2}\ll (x/H)^{1+\varepsilon}$. We then note that by assumption, $N_{j_1} N_{j_2}\ll (x/H)^{2}x^{2\eps}\ll  x^{1-4\varepsilon}.$ Moreover, $N_{j_1} N_{j_2}\gg (x/H)^{4/3}\gg T_{q}(\beta)x^{4\varepsilon}\gg T_{1}x^{4\varepsilon}$ and $q^{1/4}\ll H^{1/2}x^{-1/4}$. Therefore, the preceding term is bounded above by
\begin{align*}(\log x)^{O(1)} \Big(q^{1+\varepsilon/10}T_{q}(\beta)+(qx)^{1/2}x^{-\varepsilon}+(qx)^{1/2+\varepsilon}\frac{H^{1/6}}{x^{1/6}}+\frac{q^{\varepsilon}x^{1/2}}{1+T_{1}^{3/2}}\Big).
\end{align*}
Since $H\leq x^{17/20}$ then the third term is $O((qx)^{1/2}x^{-\varepsilon})$. Moreover, recalling that $q> (\log x)^{A}$ one has $(\log x)^{O(1)}q^{\varepsilon}x^{1/2}(1+T_{1}^{3/2})^{-1}\ll q^{1/2}x^{1/2}(\log x)^{-A/3}$, which would suffice since $\beta\leq x/H^{2}$. It is also worth noting that $qT_{q}(\beta)\ll (qx)^{1/2}x^{-\varepsilon/4}$. Consequently, 
$$q^{1+\varepsilon/10}T_{q}(\beta)+ (qx)^{1/2}x^{-\varepsilon}+\frac{q^{\varepsilon}x^{1/2}}{1+T_{1}^{3/2}}\ll (qx)^{1/2}(\log x)^{-A/3}.$$ 
This completes the proof.

\end{proof}


\section{Type III estimates}\label{sec8}

In this section, we study mean values of products of three Dirichlet polynomials, leading to what we call the \emph{Type III} estimates. We use the large-value estimates in Theorem~\ref{thm:large-values}. The proof is somewhat lengthy because it requires splitting the analysis into several cases depending on the sizes of the variables and the corresponding estimate which applies in each case. The guiding principle is to use \eqref{pp9} or~\eqref{pp91} when the parameter $\sigma$ (in the proof) lies in $[7/10,8/10]$, and the classical bound~\eqref{rr4} otherwise.

\begin{lemma}[Type III estimate]\label{lem:type-iii}

Let $\mu_{\boldsymbol N}$ be as in \eqref{mobius-n}.
	Suppose that there exists a disjoint partition $J_{1}\sqcup J_{2}\sqcup J_{3}=\{1,\ldots,2k+1\}$  and parameters 
\begin{equation*}M_{1}=\prod_{j\in J_{1}}N_{j},\ \ \ \ \ \ \ \ \ \ \ \ M_{2}=\prod_{j\in J_{2}}N_{j},\ \ \ \ \ \ \ \ \ \ \ \ M_{3}=\prod_{j\in J_{3}}N_{j}, \end{equation*}
such that  $M_{1}M_{2}M_{3}\asymp x$, with $M_{2}\leq M_{1}$ and  \begin{equation}\label{uu} 
 M_{2}\gg (x/H)^{4/5+2\varepsilon},\ \  \ \ \ x^{4/25}\ll M_{3}\ll x^{1/5}, \ \ \ \ \ \ M_1\ll H^{10} x^{-16/3-\eps }\end{equation}
and 
\begin{equation}\label{bb13}M_{2}^{1/2}M_{3}\ll x^{2/5}.\end{equation}
Then, whenever $x^{29/50+\eps }\leq H\leq x^{32/55}$ one has
$$\sup_{\alpha \in \mathfrak m} \Big|\sum_{x<n\leq x+H}\mu_{\boldsymbol N}(n)e(\alpha n)\Big|\ll H(\log x)^{-A}.$$  
	
\end{lemma}

\begin{proof}
Set $Q=x^{2/25+\eps}$ in~\eqref{eq:dirich-app}. As above, we may assume that $\beta$ is non-negative. Thus by Dirichlet approximation there exist $a\in\mathbb{Z}$ and $q\in\mathbb{N}$ with $(a,q)=1$ and such that 
\begin{equation}\label{sss}q\leq x^{2/25+\eps},\ \ \ \ \ \ \ \ \ \ \ \ 0\leq \beta=\alpha-a/q\leq \frac{1}{qx^{2/25+\eps}}.\end{equation} 
As in the proof of Lemma~\ref{lem:typei2}, we introduce for further convenience  the function $F(\beta)=T_{q}(\beta)$ if $\beta \leq x(\log x)^{A/3}/H^{2}$ and $F(\beta)=2\beta H$ if $\beta > x(\log x)^{A/3}/H^{2}$. Then, by a similar argument as in previous sections,
\begin{align*}&\sum_{ q_0q_1=q \atop q_0\leq \exp((\log x)^{1/2})}\sum_{\chi\mod{q_{1}}} \int_{2\pi \beta x-F(\beta)}^{2\pi \beta x+F(\beta)} \lvert D_{\mu_{\boldsymbol N}}(1/2+it,\chi)\rvert dt\nonumber
\\
&\ll\sum_{q_{1},\chi}  \int_{2\pi \beta x-F(\beta)}^{2\pi \beta x+F(\beta)} \prod_{j\in J_{1}}\lvert  N_{j,\mathbf{q}}(1/2+it,\chi)\rvert \prod_{j\in J_{2}}\lvert  N_{j,\mathbf{q}}(1/2+it,\chi)\rvert \prod_{j\in J_{3}}\lvert  N_{j,\mathbf{q}}(1/2+it,\chi)\rvert dt,
\end{align*}
where $\sum_{q_{1},\chi}$ denotes the sum $\sum_{q_0q_{1}=q}\sum_{q_{R}q_{1}'q_{2}\cdots q_{2k-1}=q_{0}}\sum_{\chi\mod{q_{1}}}$ with the restriction $q_0\leq \exp((\log x)^{1/2})$. We then write
$$\prod_{j\in J_{1}} N_{j,\mathbf{q}}(s,\chi)=\sum_{m\sim M}a_{m}\chi(m)m^{-s},\ \ \ \ \ \ \ \ \ \ \ \ \ \ \ \ \prod_{j\in J_{2}} N_{j,\mathbf{q}}(s,\chi)=\sum_{n\sim N}b_{n}\chi(n)n^{-s}$$ and $$ \prod_{j\in J_{3}} N_{j,\mathbf{q}}(s,\chi)=\sum_{r\sim R}c_{r}\chi(r)r^{-s},$$ where $$M=M_{1}/r_{1},\ \ \ \ \ \ \ \ \ \ N=M_{2}/r_{2},\ \ \ \ \ \ \ \ \ \ R=M_{3}/r_{3}$$ and $r_{1}r_{2}r_{3}=q_{0}$, with $a_{m},b_{n},c_{r}$ being divisor-bounded functions. Meanwhile, since $q_0=r_1r_2r_3 \leq \exp((\log x)^{1/2})$ and $M_2\leq M_1$, we may assume, without loss of generality, that $N\leq M$.

In order to prove the lemma, Propositions~\ref{prop:small-beta} and \ref{pro:large-beta} reduce matters to showing that
\[
\sum_{q_0q_1=q \atop q_0\leq \exp((\log x)^{1/2})}
\sum_{q_{R}q_{1}'q_{2}\cdots q_{2k-1}=q_{0}}
\sum_{\chi \mod{q_1}}
 \int_{2\pi \beta x-F(\beta)}^{2\pi \beta x+F(\beta)} 
\prod_{1\leq l\leq 3}\prod_{j\in J_l}
\lvert N_{j,\mathbf{q}}(1/2+it,\chi)\rvert \,dt
\]
is $O((qx)^{1/2}(\log x)^{-A})$ when
$\beta\leq x(\log x)^{A/3}/H^2$, and 
$O((q\beta)^{1/2}H(\log x)^{-A})$ when $\beta>x (\log x)^{A/3}/H^2$.
We now pass to a discretized formulation. We  say that a set $\mathcal S_{q_{1}}$ of pairs
$(t,\chi)$, where $t\in [2\pi \beta x-F(\beta),2\pi \beta x+F(\beta)]$ and $\chi$ is a character $\mod{q_1}$, is \emph{well-spaced} if, whenever
$(t_1,\chi),(t_2,\chi)\in\mathcal S_{q_{1}}$ with $t_1\neq t_2$, one has
\[
2\pi \beta x-F(\beta)<t_1,t_2<2\pi \beta x+F(\beta),
\qquad
|t_1-t_2|\geq 1.
\]
Let $s\in\mathbb C$ with $\Re s=1/2$ and $\Im s=t$. It therefore suffices to prove that, for any fixed $1/2\leq\sigma_1,\sigma_2,\sigma_3\leq1$ and any well-spaced set $\mathcal S_{q_{1}}$ such that every $(t,\chi)\in\mathcal S_{q_{1}}$ satisfies
\[
\Big\lvert \sum_{m\sim M}a_{m}\chi(m)m^{-s}\Big\rvert\asymp M^{\sigma_{1}-\frac12},\ \ \  \Big\lvert \sum_{n\sim N}a_{n}\chi(n)n^{-s}\Big\rvert\asymp N^{\sigma_{2}-\frac12},\ \ \\ \Big\lvert \sum_{r\sim R}a_{r}\chi(r)r^{-s}\Big\rvert\asymp R^{\sigma_{3}-\frac12},
\]
  one then has
 \begin{align}\label{small-beta}
  \sum_{q_0q_1=q }d_{6k+3}(q_{0})\lvert \mathcal S_{q_{1}}\rvert M^{\sigma_{1}}N^{\sigma_{2}}R^{\sigma_{3}}\ll q^{1/2 }x(\log x)^{-A}\quad \text{ if } 0\leq \beta\leq x (\log x)^{A/3}/H^2	
\end{align}
and
 \begin{align}\label{large-beta}
 	 \sum_{q_0q_1=q }d_{6k+3}(q_{0})\lvert \mathcal S_{q_{1}}\rvert M^{\sigma_{1}}N^{\sigma_{2}}R^{\sigma_{3}}\ll \frac{(q\beta )^{1/2}Hx^{1/2}}{ (\log x)^{A}} \quad \text{ if } \beta> x(\log x)^{A/3}/H^2.
 \end{align}
We shall divide the analysis into several cases according to the sizes of \(\beta\) and \(q\). In most of the occasions it will suffice to bound the sums over the divisor factors crudely by $d(q)d_{6k+3}(q)\ll q^\eps$. We shall also write $\mathcal S$ in lieu of $\mathcal{S}_{q_{1}}$ for simplicity.

 \subsection*{Case 1. $\beta \leq x(\log x)^{A/3}/H^2$} 
In this context, we recall that \eqref{small-beta} is the desired estimate.  Applying \eqref{rr4} to the Dirichlet polynomial running over $m$,
we obtain
\[
|\mathcal{S}| \ll (\log x)^{O(1)}M^{2-2\sigma_1}
+\frac{qx^{1+\frac{\eps}{100}}}{H}M^{\min\{1-2\sigma_1,\,4-6\sigma_1\}}.
\]
Moreover, when $M\leq (qx/H)^{5/6}$, the assumptions
\[
q\leq x^{2/25+\eps},\qquad
M\geq (x/H)^{4/5+2\eps},\qquad
H\leq x^{3/5}
\]
imply that
\[
(qx/H)^{2/3}
\leq \frac{x^{18/25+2\eps/3}}{H^{2/3}}
\leq (x/H)^{4/5+2\eps}
\leq M,
\] and similarly for $N$. Thus, in this regime, \eqref{pp9} gives
\[
|\mathcal{S}| \ll (\log x)^{O(1)}M^{2-2\sigma_1} 
+\frac{qx^{1+\frac{\eps}{100} }}{H}M^{12/5-4\sigma_1}.
\]
Combining the preceding two estimates, we obtain
\begin{equation}\label{eq1}
|\mathcal{S}| \ll (\log x)^{O(1)}M^{2-2\sigma_1} 
+\frac{qx^{1+\frac{\eps}{100} }}{H}M^{\min\{1-2\sigma_1,\,12/5-4\sigma_1,\,4-6\sigma_1\}}
\end{equation}
provided that $M\leq (qx/H)^{5/6}$. On the other hand, if $M> (qx/H)^{5/6}$ it transpires from \eqref{pp91} that
\begin{align}\label{m-large-gm}
|\mathcal{S}| \ll (\log x)^{O(1)}M^{2-2\sigma_1}  + x^{\frac{\eps}{100}} \Big(\frac{qx}{H}\Big)^{1/2} M^{3-4\sigma_1}.
\end{align}
 We now define some auxiliary functions $f_M$ and $h_M$ as follows. If
$M\leq (qx/H)^{5/6}$, set 
\begin{align}\label{f-q-large}
f_M(\sigma)=
\begin{cases}
1-2\sigma, & 1/2\leq \sigma\leq 7/10,\\
12/5-4\sigma, & 7/10\leq \sigma\leq 8/10,\\
4-6\sigma, & 8/10\leq \sigma\leq 1,
\end{cases} \qquad \text{and} \qquad  h_M(\sigma)=1.
\end{align}
If instead $M\geq (qx/H)^{5/6}$, we set
\[
f_M(\sigma)=
\begin{cases}
1-2\sigma, & \text{if $qx/H\leq M^{4-4\sigma}$},\\
3-4\sigma, & \text{otherwise},
\end{cases}
\qquad
h_M(\sigma)=
\begin{cases}
1, &  \text{if $qx/H\leq M^{4-4\sigma}$},\\
1/2, & \text{otherwise}.
\end{cases}
\]
For simplicity, we write $f$ and $h$ in place of $f_M$ and $h_M$ when no confusion can arise. Therefore, applying (\ref{eq1}) if $M\leq (qx/H)^{5/6}$, and \eqref{m-large-gm} and \eqref{rr4} when $M> (qx/H)^{5/6}$ one always has
\begin{align}\label{bound-m-small-beta}
|\mathcal{S}|
\ll (\log x)^{O(1)}M^{2-2\sigma_1} 
+x^{\frac{\eps}{100}} \left(\frac{qx}{H}\right)^{h_{M}(\sigma_1)}M^{f_{M}(\sigma_1)}.
\end{align}
Similarly, it transpires that
 \begin{equation}\label{bound-n}
\lvert \mathcal{S}\rvert\ll (\log x)^{O(1)} N^{2-2\sigma_{2}} + x^{\frac{\eps}{100}} \Big(\frac{qx}{H}\Big)^{h_{N}(\sigma_{2})}N^{f_N(\sigma_{2})}.\end{equation}
For the $R$-polynomial, we notice first that for any positive integer $g$ one has from Montgomery--Hal\'asz--Huxley's large value estimates (equation (\ref{rr4})) that
\begin{align}\label{r-large}
|\mathcal S| \ll x^{\frac{\eps}{100}} \Big(R^{2g(1-\sigma_3)} +\frac{qx}{H} R^{g(4-6\sigma_3)}\Big).
\end{align}
Moreover, if one further has
$$\Big(\frac{qx}{H}\Big)^{2/3}\leq R^{g}\leq \Big(\frac{qx}{H}\Big)^{5/6}$$ then one may apply the bound (\ref{pp9}) to deduce that
\begin{equation}\label{bound-r-2}
 \lvert \mathcal{S}\rvert\ll x^{\frac{\eps}{100}} \Big(R^{2g(1-\sigma_{3})} +\frac{qx}{H}R^{gf(\sigma_{3})}\Big),\end{equation}
where $f(\sigma_3)$ is as in \eqref{f-q-large}. 
 
In what follows, we write $V=\sum_{q_0q_1=q }d_{6k+3}(q_{0})\lvert \mathcal{S}\rvert M^{\sigma_{1}}N^{\sigma_{2}}R^{\sigma_{3}}$. By (\ref{small-beta}) it suffices to prove that, for all $1/2\leq \sigma_1,\sigma_2,\sigma_3\leq 1$, then
\[
  V=\sum_{q_0q_1=q }d_{6k+3}(q_{0})\lvert \mathcal{S}\rvert M^{\sigma_{1}}N^{\sigma_{2}}R^{\sigma_{3}}\ll q^{1/2 } x (\log x)^{-A}.
\] 
We shall divide the analysis into four cases according to whether the diagonal terms dominate in \eqref{bound-m-small-beta} and in  \eqref{bound-n}.

\subsection*{Case 1.1 } If both the condition
$(\log x)^{O(1)}M^{2-2\sigma_1} \geq x^{\frac{\eps}{100}} \big(\frac{qx}{H}\big)^{h(\sigma_1)} M^{f(\sigma_1)}$ and its counterpart $(\log x)^{O(1)} N^{2-2\sigma_2}\geq x^{\frac{\eps}{100}} \big(\frac{qx}{H}\big)^{h(\sigma_2)} N^{f(\sigma_2)}$ hold. We first assume that $q\geq (\log x)^{3A}$ and note from Cauchy-Schwarz inequality, \eqref{bound-m-small-beta} and \eqref{bound-n} that
\begin{align*}
 V& \ll (\log x)^{O(1)}  \sum_{q_0q_1=q }d_{6k+3}(q_{0}) \big(M^{2-2\sigma_1} N^{2-2\sigma_2}\big)^{1/2} M^{\sigma_1} N^{\sigma_2} R
\\
&\ll (\log x)^{O(1)}x \sum_{q_0q_1=q } \frac{ d_{6k+3}(q_{0})}{q_0} \ll q^{1/2}x(\log x)^{-A}.
\end{align*}
If instead $q\leq (\log x)^{3A}$, observing that $a_r$ may be expressed as the Dirichlet convolution of $\mu, \log \cdot$ and $1_{[N_i,2N_i]}$ and deducing that $|t|>(\log x)^{4A}$ by employing a similar argument to that in (\ref{jhjh}), we have from Lemma \ref{lem:upper-bd-diri} that 
\begin{align}\label{pointwise-r}
\Big| \sum_{r\sim R} \frac{a_r \chi(r)}{r^{1/2+it}}\Big| \ll R^{1/2} (\log x)^{-3A/2},
\end{align}
and therefore,
\begin{align*}
 V&\ll (\log x)^{O(1)} \sum_{q_0q_1=q }d_{6k+3}(q_{0})(M^{2-2\sigma_{1}}N^{2-2\sigma_{2}})^{1/2}M^{\sigma_{1}}N^{\sigma_{2}}R(\log x)^{-3A/2}
\\
&\ll (\log x)^{O(1)-3A/2}x \sum_{q_0q_1=q } \frac{ d_{6k+3}(q_{0})}{q_0} \ll q^{1/2}x(\log x)^{-A}.	
\end{align*}

\subsection*{Case 1.2 } If both the condition
$(\log x)^{O(1)}M^{2-2\sigma_1} \leq x^{\frac{\eps}{100}} \big(\frac{qx}{H}\big)^{h(\sigma_1)} M^{f(\sigma_1)}$ and its counterpart $(\log x)^{O(1)}N^{2-2\sigma_2} \leq x^{\frac{\eps}{100}} \big(\frac{qx}{H}\big)^{h(\sigma_2)} N^{f(\sigma_2)}$ hold. 
As suggested by the definitions of $f(\sigma)$ and $h(\sigma)$, we must further organise the analysis according to the size of $q$.

\subsection*{Case 1.2 (i)}
Suppose that $M,N\leq (qx/H)^{5/6}$. Then \eqref{bound-m-small-beta} yields
\[
|\mathcal S|\ll
\frac{qx^{1+\frac{\eps}{100} }}{H}M^{f_M(\sigma_1)},
\]
where $f_M(\sigma_1)$ is as in
\eqref{f-q-large}; the analogous estimate holding with $N$ and
$\sigma_2$ in lieu of $M$ and $\sigma_1$. Therefore, given any real number $1/4\leq a\leq 1/2$, one has
\begin{align}\label{a-sigma}
|\mathcal S|^a M^{\sigma_1} \ll x^{\frac{a\eps}{100}}(qx/H)^a M^{af(\sigma_1)+\sigma_1}\ll x^{\frac{a\eps}{100}} (qx/H)^a M^{\frac{7-4a}{10}}.
\end{align}
Here, the second bound follows by  noticing  from the definition of $f(\sigma)$ in \eqref{f-q-large} that for every $1/4\leq a< 1/2$, the inequality 
\begin{equation}\label{eq*}
af(\sigma) +\sigma \leq af(7/10) +7/10 =\frac{7-4a}{10}
\end{equation}
holds, where we essentially used the fact that the maximum of the left side of the above line is attained when $\sigma=7/10$.  

We first consider the range  $\Big(\frac{qx}{H}\Big)^{2/3}\leq R^{2}\leq \Big(\frac{qx}{H}\Big)^{5/6}$ so that \eqref{bound-r-2} is applicable. We may take $c_{1}=\frac{1}{4}$, $a_{1}=b_{1}=3/8$, and $c_{2}=1/8$, $a_{2}=b_{2}=7/16$. Then, using the  bound \eqref{a-sigma} for both the $M$-polynomial and $N$-polynomial in conjunction with \eqref{bound-r-2} for the choice $g=2$ for the $R$-polynomial gives
\begin{align*}
x^{-\varepsilon/10}V\ll&
\left(\frac{qx}{H}M^{f(\sigma_1)}\right)^{a_1}
\left(\frac{qx}{H}N^{f(\sigma_2)}\right)^{b_1}
\left(R^{4(1-\sigma_3)}\right)^{c_1}
M^{\sigma_1}N^{\sigma_2}R^{\sigma_3}\\
&+
\left(\frac{qx}{H}M^{f(\sigma_1)}\right)^{a_2}
\left(\frac{qx}{H}N^{f(\sigma_2)}\right)^{b_2}
\left(\frac{qx}{H}R^{24/5-8\sigma_3}\right)^{c_2}
M^{\sigma_1}N^{\sigma_2}R^{\sigma_3}\nonumber\\
=&
\left(\frac{qx}{H}\right)^{a_1+b_1}
M^{a_1f(\sigma_1)+\sigma_1}N^{b_1f(\sigma_2)+\sigma_2}R + \frac{qx}{H} M^{a_2f(\sigma_1)+\sigma_1}N^{b_2f(\sigma_2)+\sigma_2}R^{3/5}.\nonumber
\end{align*}
Consequently, it follows from the second estimate in \eqref{a-sigma} that
\begin{align*}
V \ll& x^{\eps/10} \Big(\Big(\frac{qx}{H}\Big)^{a_{1}+b_{1}}(MN)^{\frac{7-4a_{1}}{10}}R+\frac{qx}{H}(MN)^{\frac{7-4a_{2}}{10}}R^{3/5}\Big)\\
 \ll& x^{\eps/10} \Big(\Big(\frac{qx}{H}\Big)^{3/4}(MN)^{11/20}R
+\frac{qx}{H}(MN)^{21/40}R^{3/5} \Big)\\
\ll& q^{1/2} x^{1+\eps/10 } \Big( \frac{q^{1/4}x^{3/10}R^{9/20}}{H^{3/4}}+q^{1/2}x^{21/40}R^{3/40}H^{-1}\Big), 
\end{align*}
where in the last inequality we have employed the fact that $MNR\ll x$. Using the assumptions $q\ll x^{2/25+\eps }$, $R\ll x^{1/5}$ (from \eqref{uu}), and $H\geq x^{29/50+\eps }$, and rescaling $\eps$ if necessary, we obtain
\[
V \ll q^{1/2} x^{1+\eps/10} \big( x^{\frac1{50}+\frac3{10}+\frac9{100}-\frac{87}{200}+\eps } +x^{\frac1{25}+\frac{21}{40}+\frac3{200}-\frac{29}{50}-\eps/2 }\big)\ll q^{1/2}x^{1-\eps/10 },
\]
as desired.

We next assume that  either $R< \Big(\frac{qx}{H}\Big)^{1/3}$ or  $R> (qx/H)^{5/12}$. Let $g\geq 1$ be an integer to be specified later. We now take $c_{1}=(2g)^{-1}$, $a_{1}=b_{1}=\frac{1}{2}(1-c_{1})$, and $c_{2}=(6g)^{-1}$, $a_{2}=b_{2}=\frac{1}{2}(1-c_{2})$ and apply the estimates in \eqref{bound-m-small-beta} and \eqref{r-large},  to obtain 
\begin{align}\label{bb89}
x^{-\varepsilon}V \ll& \Big(\Big(\frac{qx}{H}\Big)M^{f(\sigma_{1})}\Big)^{a_{1}} \Big(\Big(\frac{qx}{H}\Big)N^{f(\sigma_{2})}\Big)^{b_{1}}(R^{g(2-2\sigma_{3})})^{c_{1}}M^{\sigma_{1}}N^{\sigma_{2}}R^{\sigma_{3}}\nonumber
\\
&+\Big(\Big(\frac{qx}{H}\Big)M^{f(\sigma_{1})}\Big)^{a_{2}} \Big(\Big(\frac{qx}{H}\Big)N^{f(\sigma_{2})}\Big)^{b_{2}}\Big(\frac{qx}{H}R^{g(4-6\sigma_{3})}\Big)^{c_{2}}M^{\sigma_{1}}N^{\sigma_{2}}R^{\sigma_{3}}\nonumber
\\
=&\Big(\frac{qx}{H}\Big)^{a_{1}+b_{1}} M^{a_1f(\sigma_{1})+\sigma_1}N^{b_1f(\sigma_2)+\sigma_{2}}R+\frac{qx}{H}M^{a_{2}f(\sigma_{1})+\sigma_1} N^{b_{2}f(\sigma_{2})+\sigma_2}R^{2/3}.
\end{align}
Consequently, it follows from \eqref{a-sigma} and the fact that $MNR\ll x$ that
\begin{align*}
V &\ll \Big(\frac{qx}{H}\Big)^{2a_{1}}(MN)^{\frac{7-4a_{1}}{10}}R+\Big(\frac{qx}{H}\Big)(MN)^{\frac{7-4a_{2}}{10}}R^{2/3}
\\
&\ll \Big(\frac{qx}{H}\Big)^{1-\frac{1}{2g}}(MN)^{\frac{1}{2}+\frac{1}{10g}}R+\Big(\frac{qx}{H}\Big)(MN)^{\frac{1}{2}+\frac{1}{30g}}R^{2/3}
\\
&\ll q^{1/2}x^{1+\eps } \Big(\frac{q^{1/2-1/2g} x^{1/2-2/5g} R^{1/2-1/10g}}{ H^{1-1/2g}} +\frac{q^{1/2} x^{1/2+1/30g } R^{1/6-1/30g}}{H}\Big),
\end{align*}
whence taking $g=3$ yields
\begin{align}\label{ttj}
V&\ll q^{1/2}x^{1+\eps }  \Big(\frac{q^{1/3}x^{11/30}R^{7/15}}{H^{5/6}}+q^{1/2}x^{23/45}R^{7/45}H^{-1}\Big).
\end{align}

We shall then assume first that $R^3<qx/H$ and insert the above bound into the previous estimate to obtain 
$$V\ll  q^{1/2}x^{1+\varepsilon}\Big(\frac{q^{22/45}x^{47/90}}{H^{89/90}}+q^{149/270}x^{76/135}H^{-142/135}\Big).$$ Consequently, by the assumption $q\leq x^{2/25+\eps}$, it transpires that 
$$V\ll  q^{1/2}x^{1+\varepsilon}\Big(\frac{x^{421/750+\eps}}{H^{89/90}}+x^{683/1125+\eps}H^{-142/135}\Big).$$
The preceding bound is acceptable whenever $H\geq \max( x^{1263/2225+3\eps},x^{2049/3550+3\eps}),$ as desired. Note that both of the above exponents are $<29/50$.

We then assume instead that $q\leq R^{12/5}Hx^{-1}$. It follows from (\ref{ttj}) and the restriction $R\ll x^{1/5}$ at the statement of the proposition that
\[V\ll q^{1/2}x^{1+\varepsilon}\Big(\frac{x^{1/30}R^{19/15}}{H^{1/2}}+\frac{x^{1/90}R^{61/45}}{H^{1/2}}\Big)
\ll q^{1/2}x^{1+\varepsilon}(x^{43/150}H^{-1/2}+x^{127/450}H^{-1/2}),
\]
which delivers the desired result once $H\geq x^{43/75+3\varepsilon}$ and $H\geq x^{127/225+3\varepsilon}$.

\subsection*{Case 1.2 (ii)}
We suppose next that $N \leq (qx/H)^{5/6}\leq M$. Under the latter assumption, equation \eqref{bound-m-small-beta} shows that for any real number $1/4\leq a\leq 1/2$ one has
\[
|\mathcal S|^a M^{\sigma_1} \ll 
\begin{cases}
\big(qx^{1+\frac{\eps}{100} }/H\big)^a M^{a(1-2\sigma_1)+\sigma_1} \qquad &\text{ if } \frac{qx}{H} \leq M^{4-4\sigma_1};	\\
\big(qx^{1+\frac{\eps}{100}}/H\big)^{a/2} M^{a(3-4\sigma_1) +\sigma_1} &\text{ otherwise}.
\end{cases}
\]
We thus conclude that
\begin{align}\label{a-sigma-m-large}
|\mathcal S|^a M^{\sigma_1} \ll x^{\frac{a\eps}{100}}  M^{1-a} (qx/H)^{3a/2-1/4}
\end{align}
uniformly for $1/2\leq \sigma_1 \leq 1$ and $1/4\leq a\leq 1/2$, where we essentially used the fact that the maximum is attained when $qx/H \asymp M^{4-4\sigma_1}$. Since $N\leq(qx/H)^{5/6}$, the inequality~\eqref{a-sigma} also holds with $M$ and $\sigma_1$ replaced by $N$ and $\sigma_2$, respectively.

We still assume first that $(qx/H)^{2/3} \leq R^2\leq (qx/H)^{5/6}$. A similar procedure as the one in Case 1.2 (i) would ultimately lead to a worse exponent for $H$, and hence one is forced to deploy a finer argument. Indeed, we note that whenever $f(\sigma)$ is defined as in \eqref{f-q-large}, one has for any real $c\in [1/8,1/4]$ the inequality
\[
2cf(\sigma) +\sigma \leq 2cf(7/10) +7/10 =\frac{7-8c}{10}.
\]
Therefore, we conclude from the bound \eqref{bound-r-2} corresponding to the $R$-polynomial, the estimate \eqref{a-sigma-m-large}
for the $M$-polynomial, the estimate \eqref{a-sigma} for the $N$-polynomial, and the choice of parameters $(a_1,b_1,c_1)=(1/2,1/4,1/4)$ whenever the first term in \eqref{bound-r-2} dominates, and $(a_2,b_2,c_2)=(1/2,3/8,1/8)$ whenever the second one dominates, that
\begin{align*}
x^{-\eps/10 }V \ll & \Big( \big(\frac{qx}{H}\big)^{h(\sigma_1)} M^{f(\sigma_1)}\Big)^{a_1}	\Big(\frac{qx}{H} N^{f(\sigma_2)}\Big)^{b_1}	 \Big( R^{4-4\sigma_3}\Big)^{c_1}	M^{\sigma_1}N^{\sigma_2} R^{\sigma_3}\\
&+ \Big( \big(\frac{qx}{H}\big)^{h(\sigma_1)} M^{f(\sigma_1)} \Big)^{a_2}	\Big( \frac{qx}{H} N^{f(\sigma_2)}\Big)^{b_2}	 \Big( \frac{qx}{H} R^{2f(\sigma_3)}\Big)^{c_2}	M^{\sigma_1}N^{\sigma_2} R^{\sigma_3}\\
\ll& \Big(\frac{qx}{H}\Big)^{\frac{3a_1}{2}-\frac{1}{4}+b_1} M^{1-a_1} N^{\frac{7-4b_1}{10}} R+\Big(\frac{qx}{H}\Big)^{\frac{3a_2}{2} -\frac{1}{4}+b_2+c_2}  M^{1-a_2} N^{\frac{7-4b_2}{10}} R^{\frac{7-8c_2}{10}}
\\
\ll& \Big(\frac{qx}{H}\Big)^{3/4} M^{1/2} N^{3/5} R+\Big(\frac{qx}{H}\Big)M^{1/2} N^{11/20} R^{3/5}.
\end{align*}

Since $MNR \ll x$, we now bound the first term employing the assumptions $q\leq x^{2/25+\eps }, R\ll x^{1/5}$ and $N \leq (qx/H)^{5/6}$. This then yields
\begin{align*}
V \ll & q^{1/2}x^{1+\eps/10 } \Big( \big(\frac{qx}{H^3}\big)^{1/4} N^{1/10} R^{1/2} +q^{1/2}H^{-1}x^{1/2}N^{1/20}R^{1/10}\Big)\nonumber\\
\ll& 	 q^{1/2}x^{1+\eps/10 } \Big( x^{23/50} H^{-5/6} +x^{\frac{27}{50}}H^{-1}N^{1/20}R^{1/10} \Big).
\end{align*}
The first term is acceptable provided that $H\gg x^{0.58+\eps }$. For the second one, one gets the desired bound whenever $N^{1/2}R\ll x^{2/5}$ and $H\gg x^{0.58+\eps}$. This completes the analysis of the present case, since the latter condition follows from \eqref{bb13}.

We now assume that either $R < \big(\frac{qx}{H}\big)^{1/3}$ or $R> \big(\frac{qx}{H}\big)^{5/12}$. In this case, we apply the estimate \eqref{r-large} to the $R$-polynomial. Arguing as in \eqref{bb89} but using \eqref{a-sigma-m-large} to estimate the $M$-polynomial and taking $c_{1}=(2g)^{-1}$, $a_{1}=b_{1}=\frac{1}{2}(1-c_{1})$, and $c_{2}=(6g)^{-1}$, $a_{2}=b_{2}=\frac{1}{2}(1-c_{2})$ gives
\[
V \ll x^{\eps/50} \Big( \Big(\frac{qx}{H}\Big)^{\frac{3a_1}{2}-\frac{1}{4}+b_1} M^{1-a_1} N^{\frac{7-4b_1}{10}} R+ \Big(\frac{qx}{H}\Big)^{\frac{3a_2}{2} -\frac{1}{4}+b_2+c_2}  M^{1-a_2} N^{\frac{7-4b_2}{10}} R^{2/3} \Big).
\]
Consequently,
\[
V\ll x^{\eps/50}\left(\Big(\frac{qx}{H}\Big)^{1-\frac{5}{8g}}
M^{\frac12+\frac{1}{4g}}
N^{\frac12+\frac{1}{10g}}R+
\Big(\frac{qx}{H}\Big)^{1-\frac{1}{24g}}
M^{\frac12+\frac{1}{12g}}
N^{\frac12+\frac{1}{30g}}R^{2/3}\right).
\]
Taking $g=3$ and recalling $MNR \ll x$, we then have
\begin{align}\label{f22}
V &\ll x^{\eps/50}\left(\Big(\frac{qx}{H}\Big)^{19/24}M^{7/12}N^{8/15}R + \Big(\frac{qx}{H}\Big)^{71/72}M^{19/36}N^{23/45}R^{2/3}\right)\nonumber\\
&\ll q^{\frac{1}{2}}x^{1+\eps/50}\left(
q^{\frac{7}{24}}x^{\frac{13}{40}}
M^{\frac{1}{20}}R^{\frac{7}{15}}H^{-\frac{19}{24}}+
q^{\frac{35}{72}}x^{\frac{179}{360}}
M^{\frac{1}{60}}R^{\frac{7}{45}}H^{-\frac{71}{72}}
\right).	
\end{align}

If $q\leq R^{12/5}H x^{-1}$, we have from the condition $M\ll H^{10} x^{-16/3-\eps }$ in \eqref{uu} and $R\ll x^{1/5}$ that
\begin{align*}
V&\ll q^{\frac{1}{2}}x^{1+\eps/50}\bigbrac{
R^{\frac{7}{6}}x^{\frac{1}{30}}M^{\frac{1}{20}}H^{-\frac{1}{2}}
+R^{\frac{119}{90}}x^{\frac{1}{90}}M^{\frac{1}{60}}H^{-\frac{1}{2}}}
\\
&\ll q^{\frac{1}{2}}x^{1+\eps/50}(x^{4/15}M^{1/20}H^{-1/2}+x^{62/225}M^{1/60}H^{-1/2})\ll q^{\frac{1}{2}}x^{1+\eps/50}\bigbrac{
x^{-\eps/20}+x^{\frac{14}{75}}H^{-\frac{1}{3}}},
\end{align*}
which is acceptable provided that $H\gg x^{14/25+\eps }$.

It only remains to handle the case when $R^3<qx/H$. Inserting the preceding condition into (\ref{f22}) and employing the restriction $q\ll x^{2/25+\eps},$ we note that
\begin{align*}
V\ll& q^{1/2}x^{1+\eps }\Big(  q^{161/360}x^{173/360}M^{1/20}H^{-341/360}
+q^{581/1080}x^{593/1080}M^{1/60}H^{-1121/1080}\Big)\\
\ll & q^{1/2}x^{1+2\eps }\Big( x^{1549/3000} M^{1/20} H^{-341/360} +x^{5329/9000} M^{1/60} H^{-1121/1080}\Big)
\\
\ll& q^{1/2}x^{1+2\eps }\Big( x^{749/3000}  H^{-161/360} +x^{4529/9000} H^{-941/1080}\Big)\ll q^{1/2}x^{1-\eps },
\end{align*}
where in the last inequalities we used the provisos $M\ll H^{10} x^{-16/3-\eps }$ and $H\gg x^{0.58 }$.

As we assumed that $N\leq M$, the case $M < (qx/H)^{5/6} \leq N$ would not occur, whence it only remains to examine the instance in which $ (qx/H)^{5/6}\leq \min(M, N)$.

\subsection*{Case 1.2 (iii)}
Suppose  that $ (qx/H)^{5/6}\leq M, N$. In such a situation, we merely apply Montgomery--Hal\'asz--Huxley's bound to the $R$-polynomial (equation \eqref{r-large}). We shall instead employ the estimate \eqref{a-sigma-m-large} for both the $N$ and the $M$-polynomial, and deduce by making the choices $c_{1}=(2g)^{-1}$, $a_{1}=b_{1}=\frac{1}{2}(1-c_{1})$, and $c_{2}=(6g)^{-1}$, $a_{2}=b_{2}=\frac{1}{2}(1-c_{2})$ that
\begin{align*}
x^{-\eps }V \ll & \Big( \Big(\frac{qx}{H}\Big)^{h(\sigma_1)} M^{f(\sigma_1)}\Big)^{a_1}	\Big( \Big(\frac{qx}{H}\Big)^{h(\sigma_2)}  N^{f(\sigma_2)}\Big)^{a_1}	 \Big( R^{2g(1-\sigma_3)}\Big)^{c_1}	M^{\sigma_1}N^{\sigma_2} R^{\sigma_3}\\
&+ \Big( \Big(\frac{qx}{H}\Big)^{h(\sigma_1)} M^{f(\sigma_1)} \Big)^{a_2}	\Big(  \Big(\frac{qx}{H}\Big)^{h(\sigma_2)}  N^{f(\sigma_2)}\Big)^{a_2}	 \Big( \frac{qx}{H} R^{g(4-6\sigma_3)}\Big)^{c_2}	M^{\sigma_1}N^{\sigma_2} R^{\sigma_3}\\
\ll & \Big( \frac{qx}{H}\Big)^{3a_1-1/2}(MN)^{1-a_1} R + \Big( \frac{qx}{H}\Big)^{3a_2-1/2+c_2}(MN)^{1-a_2}R^{2/3}.
\end{align*}
We next note that $q\leq N^{6/5}H/x\ll (MN)^{3/5}H/x$ since $N\leq M$ by assumption. Therefore, the preceding bound combined with the condition $MN\ll x/R$, the proviso $x^{4/25 } \ll R\ll x^{1/5}$ and the choice $g=3$ yields
\begin{align*}
V&\ll q^{1/2}x^\eps  \Big( q^{1/4} (x/H)^{3/4} (MN)^{7/12} R+ q^{17/36}(x/H)^{35/36}(MN)^{19/36} R^{2/3}\Big)\\
&\ll  q^{1/2}x^{1+\eps}\Big(x^{-1/2}(NM)^{11/15}RH^{-1/2}+ x^{-1/2}H^{-1/2}(NM)^{73/90}R^{2/3}\Big)
\\
&\ll  q^{1/2} x^{1+\eps } \Big(x^{7/30}R^{4/15} H^{-1/2} + x^{14/45}R^{-13/90} H^{-1/2}\Big)
\\
&\ll q^{1/2} x^{1+\eps } \Big(x^{43/150} H^{-1/2} + x^{36/125} H^{-1/2}\Big)\ll  q^{1/2} x^{1-\eps }
\end{align*}
provided that $H\geq x^{0.58 }$. This gives the desired bound.

\subsection*{Case 1.3 } If both the condition
$(\log x)^{O(1)} M^{2-2\sigma_1}\geq x^{\frac{\eps}{100}} \big(\frac{qx}{H}\big)^{h(\sigma_1)} M^{f(\sigma_1)}$ and its counterpart $(\log x)^{O(1)}N^{2-2\sigma_2} \leq x^{\frac{\eps}{100}} \big(\frac{qx}{H}\big)^{h(\sigma_2)} N^{f(\sigma_2)}$ hold.
Under these circumstances, we notice that the first term in  \eqref{bound-m-small-beta} dominates. Therefore, we have
\[
|\mathcal S| \ll (\log x)^{O(1)} M^{2-2\sigma_1};
\]
while for the $N$-polynomial, the classical bound \eqref{rr4} suffices, namely
\[
|\mathcal S|\ll\frac{qx^{1+\frac{\eps}{100}}}{H}N^{g(\sigma_2)},
\]
where $g(\sigma_2) = \min\set{1-2\sigma_2,4-6\sigma_2}$. It then follows from \cite[(7.2.17)]{Har} that for any real $1/6\leq b\leq 1/2$ we have
\[
bg(\sigma) +\sigma \leq bg(3/4) +3/4=\frac{3-2b}{4}.
\]
In an analogous manner as above, we employ \eqref{r-large} for the $R$-polynomial. Therefore, as in the proof of case 3 (ii) of \cite[Lemma 7.3]{Har}, we apply H\"older's inequality for the choices $a_{1}=\frac{1}{2}$, $b_{1}=\frac{1}{2}-\frac{1}{2g}$, $c_{1}=\frac{1}{2g}$ and $a_{2}=\frac{1}{2}$, $b_{2}=\frac{1}{2}-\frac{1}{6g}$, $c_{2}=\frac{1}{6g}$. This gives
\begin{align*}
V \ll & x^{\eps/10} \Big( MR (qx/H)^{\frac{1}{2}-\frac{1}{2g}} N^{\frac{1}{2}+\frac{1}{4g}} + MR^{2/3} (qx/H)^{\frac{1}{2}} N^{\frac{1}{2}+\frac{1}{12g}}\Big)\\
\ll & q^{1/2} x^{1+\eps/10 } \Big( (x/H)^{\frac{1}{2}-\frac{1}{2g}} N^{-\frac{1}{2}+\frac{1}{4g}} + x^{1/6}H^{-\frac{1}{2}} M^{1/3} N^{-\frac{1}{6}+\frac{1}{12g}} \Big),	
\end{align*}
where we employed the proviso $MNR\ll x$ in the second line and used the trivial bound $q^{-1/(2g)}\leq1$ in the first one. Now taking $g=3$ delivers
\[
V\ll q^{1/2} x^{1+\eps/10 } \Big( (x/H)^{1/3}N^{-5/12} + x^{1/6}M^{1/3} N^{-5/36}H^{-1/2} \Big).
\]
The first term is acceptable since, by \eqref{uu}, we have $N\gg (x/H)^{4/5+2\eps }$. For the second term, since $M\ll x/(NR)$, it transpires that
\[
M^{1/3}N^{-5/36}\ll x^{1/3} N^{-17/36} R^{-1/3}\ll x^{-22/225 }H^{17/45}\ll H^{1/2}x^{-1/6 -2\eps },
\]
where in the penultimate inequality we  used that $N\gg (x/H)^{4/5+2\eps }$ and $R\gg x^{4/25 }$ by \eqref{uu}, the last inequality in the above line holding provided that $H\gg x^{31/55+100\varepsilon}$. Therefore, we are done in this case.

\subsection*{Case 1.4 }If both the condition $(\log x)^{O(1)}M^{2-2\sigma_1} \leq x^{\frac{\eps}{100}} \big(\frac{qx}{H}\big)^{h(\sigma_1)} M^{f(\sigma_1)}$ and its counterpart $(\log x)^{O(1)}N^{2-2\sigma_2} \geq x^{\frac{\eps}{100}} \big(\frac{qx}{H}\big)^{h(\sigma_2)} N^{f(\sigma_2)}$ hold.
 This case is symmetrical to Case 1.3, whence no further details are thereby required.

This completes the estimation of \eqref{small-beta} for all $1/2\leq \sigma_1,\sigma_2,\sigma_3\leq 1$. We now turn to the complementary range $\beta \geq x(\log x)^{A/3}/H^2$.

\subsection*{Case 2. $x(\log x)^{A/3}/H^2\leq \beta$} In this context one has $F(\beta)=2\beta H$, and the goal becomes to prove \eqref{large-beta} in this range. We first note that the Diophantine approximation conditions in \eqref{sss} entail that $q\beta\leq x^{-2/25-\eps}$. Consequently, whenever $H\ll x^{32/55}$, one has
\[
(q\beta H)^{2/3} \leq H^{2/3} x^{-4/75 } \ll (x/H)^{4/5} \ll N\leq M.
\]
Thus, the lower-bound hypotheses on $M$ and $N$ required to apply \eqref{pp9} are satisfied. Arguing as in the preceding range, we apply Theorem \ref{thm:large-values} to the Dirichlet polynomials running over $m$, $n$ and $r$, the only change being that the assumption underlying \eqref{pp91} regarding the sizes of $N$ and $M$ is now $N,M>(q\beta H)^{5/6}$. More precisely, one gets analogous estimates to those in \eqref{bound-m-small-beta}-\eqref{bound-r-2}, namely
\begin{align} \label{bound-m-large-beta}
|\mathcal S| \ll (\log x)^{O(1)}M^{2-2\sigma_1}  + x^{\frac{\eps}{100}} (q\beta H)^{h(\sigma_1)}M^{f(\sigma_1)},	
\end{align}
where if $M\leq (q\beta H)^{5/6}$ then   $f(\sigma_1)$ and $h(\sigma_1)$ are defined as in \eqref{f-q-large}; and otherwise 
\[
f(\sigma_1)=
\begin{cases}
1-2\sigma_1, & q\beta H\leq M^{4-4\sigma_1},\\
3-4\sigma_1, & \text{else},
\end{cases}
\qquad
h(\sigma_1)=
\begin{cases}
1, & q\beta H\leq M^{4-4\sigma_1},\\
1/2, & \text{else}.
\end{cases}
\]
The same argument applied to the $N$-polynomial gives
\begin{align} \label{bound-n-large-beta}
|\mathcal S| \ll  (\log x)^{O(1)}N^{2-2\sigma_2} + x^{\frac{\eps}{100}} (q\beta H)^{h(\sigma_2)}N^{f(\sigma_2)}.	
\end{align}
For the $R$-polynomial, the Montgomery--Hal\'asz--Huxley's large value estimate (\ref{rr4}) delivers,  for every positive integer $g$, the bound
\begin{align}\label{r-large-beta-large}
|\mathcal S| \ll x^{\frac{\eps}{100}} (R^{2g(1-\sigma_3)} +q\beta H R^{g(4-6\sigma_3)}).
\end{align}
Moreover, if 
$$(q\beta H)^{2/3}\leq R^{g}\leq (q\beta H)^{5/6}$$ then one may apply Theorem \ref{thm:large-values} to deduce that
\begin{equation}\label{bound-r-2-beta-large}
 \lvert \mathcal{S}\rvert\ll x^{\frac{\eps}{100}}\Big(R^{2g(1-\sigma_{3})}+q\beta HR^{gf(\sigma_{3})}\Big),\end{equation}
where $f(\sigma_3)$ is the piecewise-linear function defined in \eqref{f-q-large}. 
 
As before, set $V=\sum_{q_0q_1=q } d_{6k+3}(q_{0})\lvert \mathcal{S}\rvert M^{\sigma_{1}}N^{\sigma_{2}}R^{\sigma_{3}}$. To prove \eqref{large-beta}, it is enough to show that, uniformly for $1/2\leq \sigma_1,\sigma_2,\sigma_3\leq 1$, one has
\[
  V=\sum_{q_0q_1=q } d_{6k+3}(q_{0})\lvert \mathcal{S}\rvert M^{\sigma_{1}}N^{\sigma_{2}}R^{\sigma_{3}}\ll (q\beta x)^{1/2} H (\log x)^{-A}.
\] 
In an analogous manner, we divide the proof into four cases, according to whether the diagonal term dominates in each of \eqref{bound-m-large-beta} and \eqref{bound-n-large-beta}.

\subsection*{Case 2.1} If both the bound $(\log x)^{O(1)}M^{2-2\sigma_1}  \geq x^{\frac{\eps}{100}} (q\beta H)^{h(\sigma_1)}M^{f(\sigma_1)}$ and its counterpart $(\log x)^{O(1)} N^{2-2\sigma_2}  \geq  x^{\frac{\eps}{100}} (q\beta H)^{h(\sigma_2)}N^{f(\sigma_2)}$ hold.
Suppose first that $q\geq(\log x)^{3A}$. Then, combining \eqref{bound-m-large-beta} and \eqref{bound-n-large-beta} with the Cauchy-Schwarz inequality, we obtain

\begin{align*}
 V& \ll (\log x)^{O(1)}  \sum_{q_0q_1=q }d_{6k+3}(q_{0}) \big(M^{2-2\sigma_1} N^{2-2\sigma_2}\big)^{1/2} M^{\sigma_1} N^{\sigma_2} R
\\
&\ll (\log x)^{O(1)}x \sum_{q_0q_1=q } \frac{ d_{6k+3}(q_{0})}{q_0} \ll q^{1/2}x(\log x)^{-A}\ll \frac{(q\beta x)^{\frac12}H }{ (\log x)^A},
\end{align*}
where in the last step we used the assumption $\beta \geq x(\log x)^{A/3}/H^2$. This proves the desired estimate in this case.

It remains to consider  $1\leq q\leq (\log x)^{3A}$.  In this range, the pointwise estimate \eqref{pointwise-r} provides an arbitrarily large logarithmic saving for the $R$-polynomial. Therefore, arguing as in Case 1.1 we obtain
\[
V\ll q^{1/2} x(\log x)^{-3A/2+O(1)} \ll (q\beta x)^{1/2}H (\log x)^{-A},
\]
where we used the condition $\beta \geq x(\log x)^{A/3}/H^2$.

\subsection*{Case 2.2} If both the bound $(\log x)^{O(1)}M^{2-2\sigma_1}  \leq x^{\frac{\eps}{100}} (q\beta H)^{h(\sigma_1)}M^{f(\sigma_1)}$ and its counterpart $(\log x)^{O(1)}N^{2-2\sigma_2}  \leq  x^{\frac{\eps}{100}} (q\beta H)^{h(\sigma_2)}N^{f(\sigma_2)}$ hold.  
\subsection*{Case 2.2 (i)} Suppose that $M,N \leq (q\beta H)^{5/6}$. In this case the functions $f(\sigma)$ and $h(\sigma)$ in estimates \eqref{bound-m-large-beta} and \eqref{bound-n-large-beta} are as in \eqref{f-q-large}. Therefore, in a similar manner as in \eqref{a-sigma}, one has for any real number  $1/4\leq a\leq 1/2$ the bound
\begin{align}\label{a-sigma-large-beta}
|\mathcal S|^a M^{\sigma_{1}} \ll x^{\frac{a\eps}{100}} (q\beta H)^a M^{af(\sigma_1)+\sigma_1}\ll x^{\frac{a\eps}{100}} (q\beta H)^{a}M^{\frac{7-4a}{10}}
\end{align}
for arbitrary $1/2\leq \sigma_1,\sigma_2\leq 1$. An analogous estimate holds as well for $N$.

As in the proof of Case 1.2 (i) we first assume that $ (q\beta H)^{2/3}\leq R^{2}\leq (q\beta H)^{5/6}$. It then follows from \eqref{bound-r-2-beta-large} that
\[
|\mathcal S| \ll x^{\frac{\eps}{100}} \big( R^{4-4\sigma_3} + q\beta H R^{24/5-8\sigma_3}\big).
\]
Similarly as therein, we take $a_1=b_1=3/8,c_1=1/4$, and $a_2=b_2=7/16$, $c_2=1/8$ and deduce from the above estimates and  $MNR\ll x$ that
\begin{align*}
V \ll & x^{\eps/10 } \Big( (q\beta H)^{2a_1} M^{a_1f(\sigma_1)+\sigma_1}N^{a_1 f(\sigma_2)+\sigma_2}R + q\beta H M^{a_2f(\sigma_1)+\sigma_1}N^{a_2 f(\sigma_2)+\sigma_2} R^{3/5}\Big)\\
\ll & x^{\eps/10 } \big( (q\beta H)^{3/4} (MN)^{11/20} R+ q\beta H (MN)^{21/40} R^{3/5}\big)\\
\ll &  x^{\eps/10} (q\beta x)^{1/2} H\big( (q\beta)^{1/4} x^{1/20} R^{9/20} H^{-1/4} +(q\beta)^{1/2} x^{1/40} R^{3/40}\big).	
\end{align*}
We now use the assumptions $q\beta \ll x^{-2/25-\eps }$ and $R\ll x^{1/5}$ in conjunction with the condition $H\gg x^{0.58+\eps}$ to obtain
\[
V \ll x^{\eps/10} (q\beta x)^{1/2} H \big(x^{3/25-\eps/4 } H^{-1/4} +x^{-\eps/2}\big) \ll  x^{-\eps/10} (q\beta x)^{1/2} H,
\]
as desired.

We shall next consider the case in which either $R< (q\beta H)^{1/3}$ or $R> (q\beta H)^{5/12}$. Taking $c_1=1/6$, $a_1=5/12$ and $c_2=1/18$, $a_2=17/36$, we have by \eqref{bound-r-2-beta-large} for the choice $g=3$ and $MNR\ll x$  that
\begin{align}\label{ec135}
V &\ll x^\eps \Big( (q\beta H)^{2a_1} M^{a_1f(\sigma_1)+\sigma_1}N^{a_1 f(\sigma_2)+\sigma_2} R +q\beta H M^{a_2f(\sigma_1)+\sigma_1}N^{a_2 f(\sigma_2)+\sigma_2} R^{2/3}\Big)\nonumber\\
&\ll x^\eps \Big((q\beta H)^{2a_{1}}(MN)^{\frac{7-4a_{1}}{10}}R+(q\beta H)(MN)^{\frac{7-4a_{2}}{10}}R^{2/3}\Big)\nonumber
\\
&\ll  x^\eps  \big((q\beta H)^{5/6} (MN)^{8/15} R+ q\beta H (MN)^{23/45} R^{2/3}\big)\nonumber\\
&\ll x^\eps (q\beta x)^{1/2} H \big( (q\beta)^{1/3} x^{1/30} R^{7/15} H^{-1/6} +(q\beta)^{1/2} x^{1/90} R^{7/45}\big).
\end{align}
Suppose first that $R^3\leq q\beta H$. One deduces by this assumption and $q\beta \ll x^{-2/25-\eps }$ that
\begin{align*}
V \ll &x^{\eps} (q\beta x)^{1/2} H \big((q\beta)^{22/45}x^{1/30}H^{-1/90}+(q\beta)^{149/270}x^{1/90}H^{7/135}\big)\\
&\ll x^{\eps} (q\beta x)^{1/2} H \big( x^{-2\eps  }+ x^{-223/6750  }H^{7/135}\big) \ll x^{-\eps/2 } (q\beta x)^{1/2} H 
\end{align*}
provided that $H\ll x^{223/350-100\eps}$, which holds in view of the restriction $H\ll x^{32/55}$.

On the other hand, suppose that $R > (q\beta H)^{5/12}$. Then $q\beta < R^{12/5}/H$, whence inserting this inequality and the assumption $R\ll x^{1/5}$ in (\ref{ec135}) we find that
\begin{align*}
V&\ll x^\eps (q\beta x)^{1/2} H (x^{1/30}R^{19/15}H^{-1/2}+x^{1/90}R^{61/45}H^{-1/2})
\\
&\ll x^\eps (q\beta x)^{1/2} H (x^{43/150} H^{-1/2} + x^{127/450}H^{-1/2} ).
\end{align*}
The preceding bound is acceptable whenever $H\gg x^{43/75+3\eps }$, as desired.

\subsection*{Case 2.2 (ii)} Suppose that $N \leq (q\beta H)^{5/6} \leq M$. Since $M\geq (q\beta H)^{5/6}$, one may apply \eqref{bound-m-large-beta} to the $M$-polynomial. Arguing as in \eqref{a-sigma-m-large}, we thus obtain
\begin{align}\label{sigma-m-ll}
	|\mathcal S|^a M^{\sigma_1} \ll x^{\frac{a\eps}{100}} M^{1-a} (q\beta H)^{\frac{3a}{2}-\frac{1}{4}}
\end{align}
 uniformly for  $1/2\leq \sigma_1\leq 1$ and $1/4\leq a\leq 1/2$.
 
 When $(q\beta H)^{2/3} \leq R^2 \leq (q\beta H)^{5/6}$, we make the choices $b_1=c_1=1/4, a_1=1/2$, and $a_2=1/2, b_2=3/8$ and $c_2=1/8$. Employing \eqref{sigma-m-ll} for the $M$-polynomial, the analogue of \eqref{a-sigma-large-beta} for the $N$-polynomial, and \eqref{bound-r-2-beta-large} with $g=2$ for the $R$-polynomial, and using the inequality (\ref{eq*}) and $MNR\ll x$, we obtain 
 \begin{align*}
V\ll& x^{\eps/4} \big((q\beta H)^{\frac{3a_1}{2}-\frac{1}{4}+b_1} M^{1-a_1} N^{\frac{7-4b_1}{10}} R+(q\beta H)^{\frac{3a_2}{2}-\frac{1}{4}+b_2+c_2} M^{1-a_2} N^{\frac{7-4b_2}{10}} R^{\frac{7-8c_2}{10}}\big)\\
\ll& x^{\eps/4} \big((q\beta H)^{3/4}M^{1/2} N^{3/5} R + q\beta H M^{1/2}N^{11/20} R^{3/5}\big)\\
\ll & x^{\eps/4} (q\beta x)^{1/2} H \big((q\beta)^{1/4} H^{-1/4}N^{1/10} R^{1/2} +(q\beta)^{1/2} N^{1/20} R^{1/10}\big). 
\end{align*}
We then bound the first term using the assumptions $N \ll (q\beta H)^{5/6}$, $q\beta \ll x^{-2/25-\eps}$ and  $R\ll x^{1/5}$; and the second one employing $q\beta \ll x^{-2/25-\eps }$. This leads to 
\begin{align*}
V&\ll x^{\eps/4} (q\beta x)^{1/2} H \big((q\beta)^{1/3} H^{-1/6}x^{1/10} +(q\beta)^{1/2} N^{1/20} R^{1/10}\big)
\\
&\ll x^{\eps/4}   (q\beta x)^{1/2} H \big(x^{11/150}H^{-1/6}+x^{-1/25-\varepsilon/2}N^{1/20}R^{1/10}\big),
\end{align*}
the first term being acceptable if $H\gg x^{1/2}$, and the second one being of the required size whenever $N^{1/2}R\ll x^{2/5}$, which was assumed in \eqref{bb13}.

If instead either $R< (q\beta H)^{1/3}$ or $R>(q\beta H)^{5/12}$, arguing as in Case 2.2 (i), we shall apply \eqref{bound-r-2-beta-large} with $g=3$ to the $R$-polynomial. Therefore, the choices $(a_1,b_1,c_1)=(5/12,5/12,1/6)$ and $(a_2,b_2,c_2)=(17/36,17/36,1/18)$ combined with the condition $MNR\ll x$ deliver
\begin{align}\label{ec72}
V \ll & x^{\eps/50} \big( (q\beta H)^{\frac{3a_1}{2}-\frac{1}{4}+b_1} M^{1-a_1} N^{\frac{7-4b_1}{10}} R+(q\beta H)^{ \frac{3a_2}{2}-\frac{1}{4}+b_2+c_2} M^{1-a_2} N^{\frac{7-4b_2}{10}} R^{2/3} \big)\nonumber\\
\ll& x^{\eps/50} \big((q\beta H)^{19/24} M^{7/12}N^{8/15}R+(q\beta H)^{71/72}M^{19/36}N^{23/45}R^{2/3}\big)\nonumber\\
\ll& x^{\eps/50} (q\beta x)^{\frac12} H \big( (q\beta)^{\frac{7}{24}} H^{-\frac{5}{24}} x^{\frac{1}{30}} M^{\frac{1}{20}} R^{\frac{7}{15}}
+(q\beta)^{\frac{35}{72}} H^{-\frac{1}{72}} x^{\frac{1}{90}} M^{\frac{1}{60}} R^{\frac{7}{45}} \big).	
\end{align}

We suppose that $q\beta H< R^{12/5}$ first. It then transpires by the above equation and the assumption $R\ll x^{1/5}$ that
\begin{align*}
V \ll& x^{\eps/50} (q\beta x)^{1/2} H \big( x^{1/30} M^{1/20} R^{7/6} H^{-1/2} + x^{1/90} M^{1/60} R^{119/90} H^{-1/2}\big)\\
\ll&  x^{\eps/50} (q\beta x)^{1/2} H \big( x^{4/15} M^{1/20} H^{-1/2} + x^{62/225} M^{1/60} H^{-1/2}\big)\\
\ll & x^{\eps/50} (q\beta x)^{1/2} H \big( x^{-\eps/20} + x^{14/75}  H^{-1/3}\big)\ll x^{-\eps/50} (q\beta x)^{1/2} H,
\end{align*}
the last inequalities following provided that $M\ll H^{10} x^{-16/3-\eps}$, which is assumed in \eqref{uu}, and $H\gg x^{0.58+\eps }$.

If instead $R< (q\beta H)^{1/3}$, we insert the condition $q\beta \ll x^{-2/25-\eps }$ into (\ref{ec72}) to derive 
\begin{align*}
V&\ll  x^\eps (q\beta x)^{1/2} H \big(  (q\beta )^{161/360} H^{-19/360} x^{1/30} M^{1/20}  +(q\beta)^{581/1080} H^{41/1080} x^{1/90} M^{1/60} \big)
\\
&\ll  x^\eps (q\beta x)^{1/2} H \big( x^{-11/4500}M^{\frac{1}{20}}H^{-19/360} +x^{-\frac{431}{13500}}M^{1/60}H^{41/1080} \big).
\end{align*}
This is acceptable when $M\ll H^{10}x^{-16/3-\eps }$ and $x^{0.58+\eps }\leq H \leq x^{0.59}$, since then 
\begin{align*}
V&\ll  x^\eps (q\beta x)^{1/2} H \big( x^{-1211/4500}H^{161/360} +x^{-\frac{1631}{13500}}H^{221/1080} \big)\ll x^{-\eps} (q\beta x)^{1/2} H,
\end{align*}
as desired.
\subsection*{Case 2.2 (iii)} Suppose that $M,N \geq (q\beta H)^{5/6}$. We shall apply the inequality \eqref{sigma-m-ll} to both the M-polynomial and the $N$-polynomial. Arguing as in Case 1.2 (iii), we also apply \eqref{r-large-beta-large} with $g=3$ to the $R$-polynomial. It then follows by making the choices $(a_1,b_1,c_1)=(5/12,5/12,1/6)$ and $(a_2,b_2,c_2)=(17/36,17/36,1/18)$ that
\begin{align*}
V\ll &x^{\eps } \big( (q\beta H)^{3a_1-1/2} (MN)^{1-a_1} R +(q\beta H)^{3a_2-1/2+c_2} (MN)^{1-a_2} R^{2/3}\big)	\\
\ll & x^\eps (q\beta H)^{1/2}  \big( (q\beta H)^{1/4} (MN)^{7/12} R +(q\beta H)^{17/36}  (MN)^{19/36} R^{2/3}\big).
\end{align*}
We now remind the reader that $N\leq M$ and deduce from the bound $q\beta H\ll N^{6/5} \ll (MN)^{3/5}$ and $MNR\ll x$ that
\begin{align*}
V&\ll x^{\eps}(q\beta x)^{1/2}H\Big(x^{-1/2}H^{-1/2}(NM)^{11/15}R+ x^{-1/2}H^{-1/2}(NM)^{73/90}R^{2/3}\Big)
\\
&\ll x^\eps (q\beta x)^{1/2}H\big( x^{7/30} R^{4/15}H^{-1/2} +x^{14/45} R^{-13/90}H^{-1/2}\big).
\end{align*}
Recalling that $x^{4/25} \leq R\ll x^{1/5}$, one gets
\[
x^{7/30} R^{4/15} +x^{14/45} R^{-13/90} \ll x^{43/150}+x^{36/125}\ll H^{1/2}x^{-2\varepsilon}
\]
provided that $H\gg x^{72/125+8\eps}$. This gives the desired bound as well.

\subsection*{Case 2.3} If both the bound 
$(\log x)^{O(1)}M^{2-2\sigma_1}  \geq x^{\frac{\eps}{100}} (q\beta H)^{h(\sigma_1)}M^{f(\sigma_1)}$ and its counterpart $(\log x)^{O(1)}N^{2-2\sigma_2}  \leq  x^{\frac{\eps}{100}} (q\beta H)^{h(\sigma_2)}N^{f(\sigma_2)}$ hold.
 For the $M$-polynomial, we have from \eqref{bound-m-large-beta} that
\[
|\mathcal S| \ll x^{\frac{\eps}{100}} M^{2-2\sigma_1},
\]
and for the $N$-polynomial we shall merely utilise \eqref{rr4}, which entails that
\[
|\mathcal S| \ll x^{\frac{\eps}{100}} q\beta H N^{\min\set{1-2\sigma_2,4-6\sigma_2}}.
\] 
The reader may note from Case 1.3 that for any real $1/6\leq b\leq 1/2$  and $1/2\leq \sigma_2\leq 1$ the following inequality 
\[
|\mathcal S|^b N^{\sigma_2}\ll x^{\frac{b\eps}{100}} (q\beta H)^b N^{b\min\set{1-2\sigma_2,4-6\sigma_2}+\sigma_2} \ll x^{\frac{b\eps}{100}} (q\beta H)^b N^{\frac{3-2b}{4}}
\] 
holds. We shall also apply \eqref{r-large-beta-large} with $g=3$ to the $R$-polynomial. The choices $(a_1,b_1,c_1)=(1/2,1/3,1/6)$ and $(a_2,b_2,c_2)=(1/2,4/9,1/18)$ in conjunction with the fact that $MNR\ll x$ then yield
\begin{align*}
V &\ll x^{\eps/10} \big( (q\beta H)^{b_1} M N^{\frac{3-2b_1}{4}} R +(q\beta H)^{b_1+c_1} M N^{\frac{3-2b_2}{4}} R^{2/3} \big)\\
&\ll x^{\eps/10}(q\beta x)^{1/2}H\big(x^{1/2}(q\beta)^{-1/6} H^{-2/3} N^{-5/12} +x^{1/2}H^{-1/2}N^{-17/36} R^{-1/3} \big).
\end{align*}
Using the proviso $\beta \geq x/H^2$ (as we are in Case 2) and $N\gg (x/H)^{4/5+2\eps }$ combined with $H\leq x^{3/5}$, one deduces that
\[
x^{1/2}(q\beta)^{-1/6} H^{-2/3} N^{-5/12} \ll x^{1/2} x^{-1/6} H^{1/3} H^{-2/3} (x/H)^{-1/3-5\eps/6 }\ll x^{-\eps/3}
\]
which is acceptable. For the second term, using the aditional restraint $R\gg x^{4/25}$, one has
\[
x^{1/2}H^{-1/2}N^{-17/36} R^{-1/3} \ll (x/H)^{1/2-17/45} x^{-4/75}\ll x^{31/450} H^{-11/90}\ll x^{-2\eps }
\]
provided that $H\gg x^{31/55+90\eps }$. Therefore, the desired bound also follows in this case.

\subsection*{Case 2.4} If both the condition
$(\log x)^{O(1)}M^{2-2\sigma_1}  \leq x^{\frac{\eps}{100}} (q\beta H)^{h(\sigma_1)}M^{f(\sigma_1)}$ and its counterpart $(\log x)^{O(1)}N^{2-2\sigma_2}  \geq  x^{\frac{\eps}{100}} (q\beta H)^{h(\sigma_2)}N^{f(\sigma_2)}$ hold, the situation is symmetrical to the preceding Case 2.3, whence no further details are required.

\end{proof}


\section{Proof of Theorem~\ref{thm:mobius}}\label{sec9}

With these preparations in place, we now complete the proof of Theorem~\ref{thm:mobius}.

\begin{proof}[Proof of Theorem~\ref{thm:mobius}]

We first claim that it suffices to prove the estimate for the range $x^{0.58+\eps }\leq H\leq x^{32/55 }$. Indeed, when $H\geq x^{32/55}$, we may decompose $(x,x+H]$ into intervals of length $\Theta(x^{0.58+\eps })$ and apply the preceding claim to each of the former intervals. Moreover, Proposition \ref{prop:major-arc} provides quasipolynomial cancellation for the exponential sum on the major arcs. We may therefore restrict our attention to the minor arcs, where it suffices to prove that
\[
\sup_{\alpha \in \mathfrak m} \Big|\sum_{x<n \leq x+H} \mu(n) e(n\alpha)\Big| \ll H(\log x)^{-1/3+\eps }.
\]
In light of Lemma \ref{lem:reduce-to-mn}, it further suffices to show that
\begin{align}\label{mn-bound-a}
\sup_{\alpha \in \mathfrak m} \Big|\sum_{x<n \leq x+H} \mu_{\boldsymbol N}(n) e(n\alpha)\Big| \ll H(\log x)^{-A/100 },
\end{align} where $\mu_{\boldsymbol N}$ is defined in \eqref{mobius-n}. 


Let $k=20$ in \eqref{mobius-n} and write $N_i= x^{\gamma_i}$ for $1\leq i\leq 39$, $R=x^{\gamma_{40}}$ and $P_1=x^{\gamma_{41}}$.  Whenever $(N_i,2N_i]$ is the supporting range of a M\"obius factor, the construction in \eqref{mobius-n} gives $\gamma_i\leq1/20$. Thus,
\begin{align}\label{cond-gamma}
\gamma_i\geq 0,\qquad	\sum_{1\leq i\leq 41} \gamma_i=1+O\Big(\frac{1}{\log x}\Big), \qquad \text{and } \gamma_i\leq 1/20 \quad\text{ when } 20\leq i\leq 41.
\end{align}
 Set $H=x^\theta$. Lemmas \ref{lem:type-ii}--\ref{lem:type-iii} show that \eqref{mn-bound-a}  holds uniformly for $x^{0.58+\eps }\leq H \leq x^{32/55}$ whenever at least one of the following four conditions is satisfied:
\begin{enumerate}
	\item[Case 1] (Type I sum). There exists some $\gamma_i$ such that $\gamma_i>1-\theta$.
	\item[Case 2] (Type I$_2$ sum). There exist $\gamma_i\neq \gamma_j$ such that $\gamma_i+\gamma_j \geq 4(1-\theta)/3$.
	\item[Case 3] (Type II sum). There exists a  partition $\set{1,\dots,41}=I\sqcup J$ such that 
    \[
	1-\theta <\sum_{i\in I} \gamma_i,\sum_{j\in J} \gamma_j<\theta.
	\]
	\item[Case 4] (Type III sum). There exists a  partition $\set{1,\dots,41}=I\sqcup J \sqcup K$ such that
	\begin{align}\label{con1}
	 4(1-\theta)/5< \sum_{j\in J} \gamma_j \leq \sum_{i\in I} \gamma_i &< 10\theta -16/3;\\ \label{con2}
	 4/25 \leq \sum_{k\in K}\gamma_k &\leq 1/5;\\
	\text{and}\quad \frac{1}{2}\sum_{j\in J}\gamma_j+\sum_{k\in K} \gamma_k&\leq 2/5. \label{con3}
			\end{align}
	
\end{enumerate}
It therefore remains to prove that every sequence  $\set{\gamma_1,\dots,\gamma_{41}}$ satisfying \eqref{cond-gamma} falls into one of these four cases. We may assume that Cases 1, 2, and 3 do not occur. We shall prove that Case 4 must then hold. For the sake of simpliciy we also assume that the sum in (\ref{cond-gamma}) is $1$, since the extra error term $O(1/\log x)$ would only add bounded constants when raising $x$ to such terms. 

For any partition of the index set $\set{\gamma_1,\dots,\gamma_{41}}$ into two disjoint subsets $\mathcal I$ and $\mathcal J$, write  $\gamma_I$ and $\gamma_J$ as the summation over the elements in $\mathcal I$ and $\mathcal J$ respectively, and label the two subsets so that $\gamma_I\leq\gamma_J$. Since Case 3 does not occur, we have
 \[
 \gamma_I = \min\set{\gamma_I,\gamma_J} <1-\theta.
 \]
Choose a partition $(\mathcal I,\mathcal J)$ for which the discrepancy $\gamma_J-\gamma_I$ is minimal, and set  $d=\gamma_J-\gamma_I$. Clearly,
\begin{align}\label{d-lower-bd}
d= 1-2\gamma_I >2\theta -1.
\end{align}	

Since neither Case 1 nor Case 2 occur and by our assumption $\gamma_J>\theta >4(1-\theta)/3$ provided that $\theta>4/7$, we thereby have $|\mathcal J|\geq3$. We claim that, in fact, $|\mathcal J|=3$. Let $\gamma'$ be the smallest element in $\mathcal J$. Indeed, if we transfer $\gamma'$ from $\mathcal J$ to $\mathcal I$, the minimality of the discrepancy gives
\[
|(\gamma_J-\gamma') -(\gamma_I +\gamma')| =|d-2\gamma'|\geq d.
\]
Since $\gamma'>0$, this forces
\begin{align}\label{gamma'-lower}
\gamma'\geq d.	
\end{align} 
Meanwhile, as $(\mathcal I,\mathcal J)$ is the partition so that the discrepancy $\gamma_J-\gamma_I$ is minimal and $\gamma_I<1-\theta$, we must have 
\begin{equation}\label{dfd}
\gamma_J-\gamma'\leq \gamma_I <1-\theta.\end{equation}
Now if  $|\mathcal J|\geq 4$ it is then apparent from the above inequality, the assumption on $\gamma'$, and \eqref{d-lower-bd}-\eqref{gamma'-lower} that 
\[
6\theta -3 <3\gamma' <\gamma_J-\gamma'<1-\theta.
\]
This leads to $\theta\leq 4/7$, which contradicts the proviso $\theta>0.58$. Consequently, $|\mathcal J|=3$.

 Moreover, we also have from \eqref{d-lower-bd} and \eqref{gamma'-lower} that
\[
1-\gamma_I =\gamma_J \geq 3\gamma' \geq 3(1-2\gamma_I).
\]
This gives 
\begin{align}\label{gamma-i}
	\gamma_I \geq 2/5 \qquad \text{and } \qquad \gamma'\leq 1/5.
\end{align}

Let $K=\{\gamma'\}$, and write  $I$ for the set of elements in $\mathcal I$ and $J$ for the set of elements in $\mathcal J/\set{\gamma'}$. It remains to verify that $(I,J,K)$ satisfies the conditions of Case 4. Firstly,
\[
\sum_{i\in I}\gamma_i =\gamma_I <1-\theta < 10\theta -16/3,
\]
where the last inequality holds provided that $\theta >19/33$. Next,  it follows from (\ref{dfd}) that
\begin{align*}
\sum_{i\in J}\gamma_j=\gamma_J-\gamma' \leq \gamma_I=\sum_{i\in I}\gamma_i.	
\end{align*}
Furthermore, using both (\ref{gamma-i}) and the bound $\gamma_{J}\geq \theta$ we also have 
\[
\sum_{j\in J}\gamma_j = \gamma_J-\gamma'\geq \theta -1/5 >\frac{4(1-\theta)}{5}
\]
provided that $\theta>5/9$. The above three inequalities give (\ref{con1}). Moreover, it transpires from \eqref{d-lower-bd} and \eqref{gamma'-lower} that
\[
\gamma' \geq 1-2\gamma_I \geq 1-2(1-\theta) =2\theta -1> 4/25
\]
whenever $\theta>0.58$. Combining this with the upper bound given in \eqref{gamma-i} delivers the condition (\ref{con2}), it only remaining to verify ~\eqref{con3}. Indeed, employing \eqref{gamma-i} enables one to deduce that
\[
\frac{\sum_{j\in J}\gamma_j}{2}+\gamma' = \frac{1-\gamma_I+\gamma' }{2} \leq \frac{1-2/5+1/5}{2} =2/5.
\]
Consequently, all the conditions of Case 4 are satisfied, thereby completing the proof.
\end{proof}

The reader may observe that in the above argument, whenever we wrote $<$ in equation (\ref{con1}) and each of the other cases, we could have added an extra $\varepsilon>0$ summand on the left side, this having been deliberately avoided for the sake of concission. Therefore, the assumptions in either Lemma \ref{lem:type-ii}, \ref{lem:typei}, \ref{lem:typei2} or \ref{lem:type-iii} are met.


\section{Proof of Theorem~\ref{thm:mangoldt}}\label{sec10}

The proof of Theorem~\ref{thm:mangoldt} follows essentially the same strategy as that of Theorem~\ref{thm:mobius}. We therefore give only a brief outline, highlighting the differences. Applying Heath-Brown's identity~\cite[(13.37)]{IK} with $K=20$ and following the argument leading to Lemma~\ref{lem:reduce-to-mn}, we obtain
\[
\bigabs{\sum_{x<n \leq x+H} \Lambda(n) e(n\alpha)} \ll (\log x)^{40} \sup_{\boldsymbol N} \bigabs{\sum_{x<n \leq x+H} \Lambda_{\boldsymbol N} (n) e(n\alpha)},
\]
where $$\Lambda_{\boldsymbol N} (n) =\twosum{n=n_1\cdots n_{40}}{n_i\in (N_i,2N_i]} a_1(n_1) \cdots a_{40}(n_{40}),$$ with $N_1\cdots N_{40} \asymp x$, $a_i\in \set{1,\mu, \log \cdot}$, and $N_i\ll x^{1/20}$ for each $a_i=\mu$. Therefore, an analogous argument to that in Lemma~\ref{lem: produ-diri-poly} would assert that
\[
D_{\Lambda_{\boldsymbol N}}(s,\chi) = \sum_{q_1'q_2\cdots q_{2k}=q_0}F_{\mathbf q,\Lambda_{\boldsymbol N}}(s,\chi),
\]
where $D_{\Lambda_{\boldsymbol N}}(s,\chi)$ is defined as in~\eqref{eq:f-dircht-poly-exp-sum} and
\[
F_{\mathbf{q},\Lambda_{\boldsymbol N}}(s,\chi)=N_{1,\mathbf{q}}(s,\chi)\cdots N_{2k,\mathbf{q}}(s,\chi).
\]
The only substantial difference from the corresponding decomposition for the M\"obius function is that the coefficients $a_i$ for which $N_i\gg x^{1/20}$ may now be both $\log$ and $1$.

Since Lemmas~\ref{lem:second-moment} and~\ref{lem:fourth-moment} apply to both types of coefficient sequences, the Type I, Type I$_2$, and Type III arguments hold for $\Lambda_{\boldsymbol N}$ as well. The Type II argument requires more care: when $q\leq(\log x)^A$, the proof of Lemma~\ref{lem:type-ii} obtains a saving from the Dirichlet polynomial over small primes, denoted by $P(1/2+it,\chi)$. This part of the argument does not apply to $\Lambda_{\boldsymbol N}$. However, by the statement of Theorem~\ref{thm:mangoldt}, we may assume that
 \[
 (\log x)^A \leq q \leq \frac{H^2}{x(\log x)^{A/3}} \quad \text{and}\quad \norm{q\alpha} \leq \frac{x(\log x)^{A/3}}{H^2},
 \]
 which excludes the previous case. In contrast, the large $q$ case of Lemma~\ref{lem:type-ii} still applies. The proof is then completed by the same argument as in Section~\ref{sec9}. Likewise, the proof of (\ref{ec1.1}) for $\mu$ follows by an identical manner save that $a_{i}(n)$ can no longer be $\log n$.

\section{Proof of Theorem \ref{thm1.3}}\label{sec11}
We shall now prove Theorem \ref{thm1.3} by combining Theorem \ref{thm:mangoldt} with the application of zero density estimates in short intervals. To such an end we shall derive these by incorporating Guth and Maynard's new large value estimates \cite{GM}.
\begin{theorem}\label{thm1}
Suppose that $(b_{n})$ is a sequence of complex numbers with $\lvert b_{n}\rvert\leq 1$ and $(t_{r})_{r\leq R}$ is a sequence of $1$-separated points in $[T, T+U]$ such that
$$\Big\lvert \sum_{n=N}^{2N}b_{n}n^{it_{r}}\Big\rvert\geq V$$ for all $r\leq R$. Then one has
\begin{equation}\label{rr44}
R\ll U^{\varepsilon} \big(N^{2}V^{-2}+U\min(NV^{-2},N^{4}V^{-6})\big).\end{equation} Moreover,
\begin{equation} \label{RR}R\ll U^{\varepsilon}(N^{2}V^{-2}+N^{18/5}V^{-4}+UN^{12/5}V^{-4}).\end{equation} 

\end{theorem}
\begin{proof}
The theorem holds by the same argument as in \ref{thm:large-values} but employing \cite[Theorem 1.1]{GM} instead of \cite[Theorem 1.1]{CGY}.
\end{proof}
In what follows we shall fix some $q\leq (\log x)^{A}$. We introduce for $U>0$ and $0\leq \sigma\leq 1$ the number $N(\sigma, T, U,q)$ of zeros of $$\prod_{\chi\mod{q}}L(s,\chi)$$ in the rectangle $\text{Re}(s)\geq \sigma$ and $T-U\leq \text{Im}(s)\leq T+U$. In order to provide estimates for the above object we essentially follow the classical zero-detecting methodology (see \cite{May-Pratt}). We first present some notation for further use. Let $\delta>0$ be a small constant, let
$$a(n)=\sum_{\substack{m|n\\ m\leq U^{\delta}}}\mu (m)$$
 and
\begin{equation*}M(s)=\sum_{m\leq U^{\delta}}\frac{\chi(m)\mu(m)}{m^{s}},\ \ \ \ \ \ \ \ \ \ Y=U^{\frac{25}{42}}.\end{equation*}

\begin{lemma}
Let $\chi$ be a character $\mod{q}$ and $\rho=\beta+i\gamma$ be a non-trivial zero of $L(s,\chi)$. Then either
\begin{equation}\label{hhg}\Big\lvert \sum_{N\leq n\leq 2N}\frac{a(n)\chi(n)}{n^{\rho}}\exp\Big(-\frac{n}{Y}\Big)\Big\rvert\geq \frac{1}{3\log T}\ \ \ \ \ \ \ \ \ \text{for some $N=2^{j}\in [U^{\delta}, Y(\log T)^{2}]$}\end{equation} or
\begin{equation}\label{po}\Big\lvert \frac{1}{2\pi i}\int_{\text{Re$(s)$}=-\beta+1/2}Y^{s}\Gamma(s)M(\rho+s)L(\rho+s,\chi)ds\Big\rvert\geq \frac{1}{3}.\end{equation}
\end{lemma}
\begin{proof}
The result follows by the argument in the proof of \cite[Lemma 23]{May-Pratt}. 
\end{proof}

\begin{theorem}\label{pr}

Let $q\leq (\log x)^{A}$. Suppose that $T\leq U^{2+\varepsilon}$ and that (\ref{eq-moment}) holds. Then, if $1/2\leq \sigma \leq 1$ one has $$N(\sigma,T,U,q)\ll U^{50(1-\sigma)/21+\varepsilon}.$$

\end{theorem}
\begin{proof}
By the above lemma, we may write $N(\sigma,T,U,q)=R_{I}(\sigma,T,U,q)+R_{II}(\sigma,T,U,q)$, where $R_{I}(\sigma,T,U,q)$ denotes the number of Type I zeros, that is, zeros satisfying (\ref{hhg}), and $R_{II}(\sigma,T,U,q)$, that denotes the number of Type II zeros, that is, zeros satisfying (\ref{po}).

Following the proof of Lemma 24 in \cite{May-Pratt}, we deduce that
\begin{align*}R_{II}(\sigma,T,U,q)&\ll \frac{(\log T)^{O(1)}}{Y^{4\sigma-2}}\sum_{\chi\mod{q}}\int_{T-2U}^{T+2U}\Big\lvert M\Big(\frac{1}{2}+it\Big)L\Big(\frac{1}{2}+it,\chi\Big)\Big\rvert^{4} dt
\\
&\ll U^{\frac{25}{21}(1-2\sigma)+9\delta/2}(U+T^{\frac{25}{42}+\varepsilon})\ll U^{\frac{50}{21}(1-\sigma)+5\delta},
\end{align*} where in the second step we employed the assumption (\ref{eq-moment}) and in the last one we used the proviso $T\leq U^{2+\varepsilon}$.

In order to bound the type I zeros, we introduce $$\widetilde{b_{n}}=\Big(\frac{N}{n}\Big)^{\sigma} a(n)\chi(n)\exp\Big(-\frac{n}{Y}\Big),\ \ \ \ \widetilde{D}_{N}(t)=\sum_{n\in [N,2N]} \widetilde{b_{n}}n^{it}.$$
By the classical zero-detecting methodology, it suffices to show (see the proof of Theorem 1.2 in \cite{GM}) that if $U^{\delta}\leq N <Y(\log T)^{2}$ and $W_{\chi}$ is a $1$-separated set in $[T-U, T+U]$ such that $\lvert \widetilde{D}_{N} (t)\rvert\gg N^{\sigma}$ then one has $\lvert W_{\chi}\rvert\leq U^{50(1-\sigma)/21+\varepsilon}$. We begin by showing first that whenever $U^{5/7}\ll N\ll U^{25/21+\varepsilon}$ then \begin{equation}\label{11.8}\lvert W_{\chi}\rvert\leq U^{50(1-\sigma)/21+\varepsilon}.\end{equation}

Indeed, if $\sigma\leq 7/10$ it follows by (\ref{rr44}) that 
$$\lvert W_{\chi}\rvert\ll U^{\varepsilon}\big(N^{2-2\sigma}+UN^{1-2\sigma}\big).$$ We then claim that the right side of the above equation is $O(U^{50(1-\sigma)/21+\varepsilon}).$ This may be seen by observing that when $U^{5/7}\ll N\ll U^{25/21+\varepsilon}$ and $\sigma=7/10$ it transpires that
$$N^{2-2\sigma}+UN^{1-2\sigma}=N^{3/5}+UN^{-2/5}\ll U^{5/7+\varepsilon}=U^{50(1-\sigma)/21+\varepsilon}.$$ If instead $\sigma<7/10$ then $N^{2-2\sigma}\ll U^{50(1-\sigma)/21+\varepsilon}$, whence it suffices to show that 
\begin{equation}\label{khl}N\Big(\frac{U^{50/21}}{N^{2}}\Big)^{\sigma}\ll U^{29/21+\varepsilon}.\end{equation} We may deduce that the above holds since it is satisfied when $\sigma=7/10$ and since $N^{2}\ll U^{50/21+2\varepsilon}$.

We shall now focus on the instance $\sigma>7/10$, and suppose first that $U^{5/7}\ll N\ll U^{5/6}$. Note that under such circumstances, the third term in the right side of (\ref{RR}) dominates over the second one. Therefore, an application of such a bound yields
$$\lvert W_{\chi}\rvert\ll U^{\varepsilon}\big(N^{2-2\sigma}+UN^{12/5-4\sigma}\big).$$ As above, the term $N^{2-2\sigma}$ is $O(U^{50(1-\sigma)/21+3\varepsilon}).$ Therefore, it is enough showing that 
$$N^{12/5}\Big(\frac{U^{50/21}}{N^{4}}\Big)^{\sigma}\ll U^{29/21}.$$ We infer that the preceding estimate is satisfied when $\sigma>7/10$ since it holds for $\sigma=7/10$ and since $U^{25/42}\ll U^{5/7}\ll N$, and hence $U^{50/21}\ll N^{4}$.

It therefore remains to discuss the case in which $\sigma>7/10$ and $U^{5/6}\ll N\ll U^{25/21+\varepsilon}.$ We thus apply (\ref{rr44}) and obtain
$$\lvert W_{\chi}\rvert\ll U^{\varepsilon}\big(N^{2-2\sigma}+\min(UN^{1-2\sigma},UN^{4-6\sigma})\big).$$
As above, the term $N^{2-2\sigma}$ is negligible. Moreover, if $\sigma=3/4$ it transpires that
$$\min(UN^{1-2\sigma},UN^{4-6\sigma})\ll U^{50(1-\sigma)/21}$$ whenever $UN^{-1/2}\ll U^{25/42}$, which holds since by assumption, $U^{17/21}\ll U^{5/6}\ll N$. Consequently, it suffices to show that whenever $7/10<\sigma\leq 3/4$ then (\ref{khl}) holds, and whenever $3/4<\sigma\leq 1$ then 
\begin{equation}\label{11.7}N^{4}\Big(\frac{U^{50/21}}{N^{6}}\Big)^{\sigma}\ll U^{29/21}.\end{equation}
Indeed, the bound in (\ref{khl}) in this context is satisfied as well by a similar argument, the estimate (\ref{11.7}) following since it is valid for $\sigma=3/4$ and since $U^{50/21}\ll N^{6}$ in view of the fact that $U^{25/63}\ll U^{5/6}\ll N$. This concludes the proof of (\ref{11.8}).

We next claim that if $U^{\delta}\leq N <Y(\log T)^{2}$ then there exists some $k\in\mathbb{N}$ with $k\ll \delta^{-1}$ such that $U^{5/7}\ll N^{k}\ll U^{25/21+\varepsilon}$. Assuming the preceding statement, one would then apply (\ref{11.8}) to $\widetilde{D}_{N}(t)^{k}$ to deliver the desired result. It thus remains to verify the previous claim. Indeed, whenever $U^{5/14}\ll N\ll Y(\log T)^{2}=U^{25/42}(\log T)^{2}$ then $U^{5/7}\ll N^{2}\ll U^{25/21+\varepsilon}$. Moreover, $U^{5/(7k)}<U^{25/21(k+1)}$ whenever $k\geq 2$. This in turn entails that $$[U^{\delta},Y(\log T)^{2}]\subset\bigcup_{k=1}^{K_{\delta}}[U^{5/(7k)},U^{25/(21k)+\varepsilon/k}]$$ for some $K_{\delta}\in\mathbb{N}$ satisfying $K_{\delta}\ll\delta^{-1}$, as desired.
\end{proof}
\emph{Proof of Theorem \ref{thm1.3}}. We are now prepared to deliver the proof of the aforementioned theorem. Let $\alpha\in \minor$. By Dirichlet approximation, there exists $a\in\mathbb{Z}$ and $q\in \mathbb{N}$ with $0\leq a\leq q$ such that $(a,q)=1$ and satisfying \[
  q \leq \frac{H^2}{x(\log x)^{A/3}} \quad \text{and}\quad \lvert \alpha-a/q\rvert \leq \frac{x(\log x)^{A/3}}{H^2}.
 \]
The proof of Theorem~\ref{thm:mangoldt} enables one to deduce that if $q> (\log x)^{A}$ then 	\[
\Bigabs{\sum_{x<n \leq x+H} \Lambda (n) e(n\alpha)} \leq H(\log x)^{-A/100}.
\] If instead $q\leq (\log x)^{A}$ then by the definition of $\minor$ it follows that $\lvert \alpha-a/q\rvert\geq \frac{(\log x)^{5A}}{H}$. We combine the preceding assumptions with Lemma 6 and equation (10) of \cite{PanPan} to obtain
\begin{align*}\sum_{x<n \leq x+H} \Lambda (n) e(n\alpha)\ll& H(\log x)^{-A}+q^{-1/2}(\log x)\lvert \beta\rvert^{-1}
\\
&+q^{-1/2}(\log x)^{O(1)}H\max_{\substack{1/2\leq \sigma\leq 1\\ 1\leq T_{1}\leq \lvert \beta\rvert x(\log x)^{B}}}x^{\sigma-1}N(\sigma,T_{1},F_{1}(\beta),q),\end{align*} where the function $F_{1}(\beta)$ is defined as $F_{1}(\beta)=x/H$ if $\lvert\beta \rvert \leq x/10H^{2}$ and $F_{1}(\beta)=10\lvert \beta\rvert H$ if $\lvert\beta \rvert > x/10H^{2}$, and $B>0$ is some large but fixed constant. Consequently, inserting the above lower bound on $\lvert\beta\rvert$ in the previous equation, observing that $\lvert \beta\rvert x(\log x)^{B}\leq (x/H)^{2}(\log x)^{A+B}\leq F_{1}(\beta)^{2+\varepsilon}$ and applying Theorem \ref{pr} yields
\begin{align*}\sum_{x<n \leq x+H} \Lambda (n) e(n\alpha)&\ll H(\log x)^{-A}+H(\log x)^{2}\max_{\substack{1/2\leq \sigma\leq 1-c(\log x)^{-4/5}}}\Big((x/H)^{50/21+\varepsilon}x^{-1}\Big)^{1-\sigma}
\\
&\ll H(\log x)^{-A}.\end{align*} The reader may observe that in the above lines we implicitly employed the classical known zero-free region for $L(s,\chi)$ and the assumption $H\geq x^{0.58+\varepsilon}.$ This concludes the proof.


\bibliographystyle{plain}  

\bibliography{md_references.bib}

@article{BM,
  title={A new upper bound for sets with no square differences},
  author={Bloom, Thomas F and Maynard, James},
  journal={Compos. Math.},
  volume={158},
  number={8},
  pages={1777--1798},
  year={2022},
  publisher={London Mathematical Society}
}

@article {May-Pratt,
    AUTHOR = {Maynard, J. and Pratt, K.},
     TITLE = {Half-isolated zeros and zero-density estimates},
   JOURNAL = {Int. Math. Res. Not. IMRN},
  FJOURNAL = {International Mathematics Research Notices. IMRN},
      YEAR = {2024},
    NUMBER = {19},
     PAGES = {12978--13014},
      ISSN = {1073-7928,1687-0247},
   MRCLASS = {11M26},
  MRNUMBER = {4805841},
MRREVIEWER = {Timothy\ S.\ Trudgian},
       DOI = {10.1093/imrn/rnae191},
       URL = {https://doi.org/10.1093/imrn/rnae191},
}

@Article{MPPTW,
  author  = {Matom\"aki, K. and Pandey, M. and Pliego, J. and Ter\"av\"ainen, J. and Wang, M.},
  journal = {arXiv:2610.09567},
  title   = {Higher order uniformity of the primes and cancelation of the {M}\"obius function in shorter intervals},
  year    = {2026},
}

@article {PanPan,
    AUTHOR = {Pan, C. and Pan, C.},
     TITLE = {On estimations of trigonometric sums over primes in short intervals. {III}},
   JOURNAL = {Chinese Ann. Math. Ser. B},
      VOLUME = {11},
      YEAR = {1990},
    NUMBER = {2},
     PAGES = {138--147},
      ISSN = {0252-9599},
   MRCLASS = {11L20 (11P32)},
  MRNUMBER = {1062085},
MRREVIEWER = {Matti\ Jutila},
}

@article{MPR,
  title={Exponential sums over primes},
  author={Maynard, J. and Pandey, M. and Radziwi{\l}{\l}, M.},
  journal={ arXiv:2608.14777},
  year={2026}
}

@Book{montgomery,
  author    = {Montgomery, H.},
  publisher = {Springer},
  title     = {Topics in multiplicative number theory},
  year      = {2006},
}

@book{GK-vander,
  title={Van der Corput's method of exponential sums},
  author={Graham, S. W and Kolesnik, G.},
  volume={126},
  year={1991},
  publisher={Cambridge University Press}
}

@article{Ramachandra-conve,
AUTHOR = {Ramachandra, K.},
     TITLE = {On the Phragm\'en-Lindel\"of theorem and some applications},
     journal = {Math. Z.},
   FJOURNAL = {Mathematische Zeitschrift},
    VOLUME = {72},
      YEAR = {1959},
    NUMBER = {1},
     PAGES = {192--204},
       DOI = {10.1007/BF01162949},
       URL = {https://doi.org/10.1007/BF01162949},
}

@article {Ramachandra,
    AUTHOR = {Ramachandra, K.},
     TITLE = {Some problems of analytic number theory},
   JOURNAL = {Acta Arith.},
  FJOURNAL = {Polska Akademia Nauk. Instytut Matematyczny. Acta Arithmetica},
    VOLUME = {31},
      YEAR = {1976},
    NUMBER = {4},
     PAGES = {313--324},
      ISSN = {0065-1036},
   MRCLASS = {10H15 (10H10 10H25)},
  MRNUMBER = {424723},
MRREVIEWER = {W.\ Schwarz},
       DOI = {10.4064/aa-31-4-313-324},
       URL = {https://doi.org/10.4064/aa-31-4-313-324},
}

@article{PSW,
  title={Local divisor correlations in almost all short intervals},
  author={Pliego, J. and Sun, Y. and Wang, M.},
  journal={arXiv:2503.02962},
  year={2025}
}

@article {JM,
    AUTHOR = {Jutila, M. and Motohashi, Y.},
     TITLE = {Mean value estimates for exponential sums and {$L$}-functions:
              a spectral-theoretic approach},
   JOURNAL = {J. Reine Angew. Math.},
  FJOURNAL = {Journal f\"ur die Reine und Angewandte Mathematik. [Crelle's
              Journal]},
    VOLUME = {459},
      YEAR = {1995},
     PAGES = {61--87},
      ISSN = {0075-4102,1435-5345},
   MRCLASS = {11L07 (11F72 11M06)},
  MRNUMBER = {1319517},
MRREVIEWER = {D.\ R.\ Heath-Brown},
       DOI = {10.1515/crll.1995.459.61},
       URL = {https://doi.org/10.1515/crll.1995.459.61},
}

@article {Jutila,
    AUTHOR = {Jutila, M.},
     TITLE = {Mean value estimates for exponential sums with applications to
              {$L$}-functions},
   JOURNAL = {Acta Arith.},
  FJOURNAL = {Polska Akademia Nauk. Instytut Matematyczny. Acta Arithmetica},
    VOLUME = {57},
      YEAR = {1991},
    NUMBER = {2},
     PAGES = {93--114},
      ISSN = {0065-1036},
   MRCLASS = {11L26 (11M06)},
  MRNUMBER = {1092763},
MRREVIEWER = {D.\ R.\ Heath-Brown},
       DOI = {10.4064/aa-57-2-93-114},
       URL = {https://doi.org/10.4064/aa-57-2-93-114},
}

@article {Kim,
    AUTHOR = {Kim, H.},
     TITLE = {Functoriality for the exterior square of {${\rm GL}_4$} and
              the symmetric fourth of {${\rm GL}_2$}},
     
   JOURNAL = {J. Amer. Math. Soc.},
  FJOURNAL = {Journal of the American Mathematical Society},
    VOLUME = {16},
      YEAR = {2003},
    NUMBER = {1},
     PAGES = {139--183},
      ISSN = {0894-0347,1088-6834},
   MRCLASS = {11F70 (11R39 22E46)},
  MRNUMBER = {1937203},
MRREVIEWER = {Mahdi\ Asgari},
       DOI = {10.1090/S0894-0347-02-00410-1},
       URL = {https://doi.org/10.1090/S0894-0347-02-00410-1},
}

@article{CGY,
      title={Large Value Estimates for Dirichlet Polynomials with Characters and Zero Density of Dirichlet {L}-Functions}, 
      author={B. Chen and V. Gupta and Y.C. Li},
      year={2026},
      journal={arXiv:2507.08296} 
}

@article {GM,
    AUTHOR = {Guth, L. and Maynard, J.},
     TITLE = {New large value estimates for {D}irichlet polynomials},
   JOURNAL = {Ann. of Math. (2)},
  FJOURNAL = {Annals of Mathematics. Second Series},
    VOLUME = {203},
      YEAR = {2026},
    NUMBER = {2},
     PAGES = {623--675},
      ISSN = {0003-486X,1939-8980},
   MRCLASS = {11M26 (11N05)},
  MRNUMBER = {5037841},
MRREVIEWER = {Timothy\ S.\ Trudgian},
       DOI = {10.4007/annals.2026.203.2.6},
       URL = {https://doi.org/10.4007/annals.2026.203.2.6},
}

@article {Zhan,
    AUTHOR = {Zhan, T.},
     TITLE = {On the representation of large odd integer as a sum of three almost equal primes},
      
   JOURNAL = {Acta Math. Sinica (N.S.)},
  FJOURNAL = {Acta Mathematica Sinica. New Series},
    VOLUME = {7},
      YEAR = {1991},
    NUMBER = {3},
     PAGES = {259--272},
      ISSN = {1000-9574},
   MRCLASS = {11P32 (11M06)},
  MRNUMBER = {1141240},
MRREVIEWER = {D.\ R.\ Heath-Brown},
       DOI = {10.1007/BF02583003},
       URL = {https://doi.org/10.1007/BF02583003},
}

@article {HB,
    AUTHOR = {Heath-Brown, D. R.},
     TITLE = {Prime numbers in short intervals and a generalized {V}aughan
              identity},
   JOURNAL = {Canadian J. Math.},
  FJOURNAL = {Canadian Journal of Mathematics. Journal Canadien de
              Math\'ematiques},
    VOLUME = {34},
      YEAR = {1982},
    NUMBER = {6},
     PAGES = {1365--1377},
      ISSN = {0008-414x,1496-4279},
   MRCLASS = {10H15 (10H25)},
  MRNUMBER = {678676},
MRREVIEWER = {Kenneth\ A.\ Jukes},
       DOI = {10.4153/CJM-1982-095-9},
       URL = {https://doi.org/10.4153/CJM-1982-095-9},
}

@book{MV,
  title={Multiplicative number theory {I}: {C}lassical theory},
  author={Montgomery, H. and Vaughan, R.},
  number={97},
  year={2007},
  publisher={Cambridge university press}
}

@article {MT,
    AUTHOR = {Matom\"aki, K. and Ter\"av\"ainen, J.},
     TITLE = {On the {M}\"obius function in all short intervals},
   JOURNAL = {J. Eur. Math. Soc. },
    VOLUME = {25},
      YEAR = {2023},
    NUMBER = {4},
     PAGES = {1207--1225},
      ISSN = {1435-9855,1435-9863},
       MRCLASS = {11N37},
  MRNUMBER = {4577962},
MRREVIEWER = {Peter\ Shiu},
}

@book{IK,
  title={Analytic number theory},
  author={Iwaniec, H. and Kowalski, E.},
  volume={53},
  year={2021},
  publisher={American Mathematical Soc.}
}

@article {MRTI,
    AUTHOR = {Matom\"aki, K. and Radziwi\l\l, M. and Tao, T.},
     TITLE = {Correlations of the von {M}angoldt and higher divisor
              functions {I}. {L}ong shift ranges},
   JOURNAL = {Proc. Lond. Math. Soc. (3)},
  FJOURNAL = {Proceedings of the London Mathematical Society. Third Series},
    VOLUME = {118},
      YEAR = {2019},
    NUMBER = {2},
     PAGES = {284--350},
      ISSN = {0024-6115,1460-244x},
   MRCLASS = {11N37 (11A25)},
  MRNUMBER = {3909235},
MRREVIEWER = {Olivier\ Bordell\`es},
       DOI = {10.1112/plms.12181},
       URL = {https://doi.org/10.1112/plms.12181},
}

@article {Shiu,
    AUTHOR = {Shiu, P.},
     TITLE = {A {B}run-{T}itchmarsh theorem for multiplicative functions},
   JOURNAL = {J. Reine Angew. Math.},
  FJOURNAL = {Journal f\"ur die Reine und Angewandte Mathematik. [Crelle's
              Journal]},
    VOLUME = {313},
      YEAR = {1980},
     PAGES = {161--170},
      ISSN = {0075-4102,1435-5345},
   MRCLASS = {10H25},
  MRNUMBER = {552470},
MRREVIEWER = {A.\ I.\ Vinogradov},
       DOI = {10.1515/crll.1980.313.161},
       URL = {https://doi.org/10.1515/crll.1980.313.161},
}

@book{Tao,
  title={Higher order Fourier analysis},
  author={Tao, Terence},
  volume={142},
  year={2012},
  publisher={American Mathematical Soc.}
}

@book{Har,
  title={Prime-Detecting Sieves.(LMS-33)},
  author={Harman, G.},
  year={2012},
  publisher={Princeton university press}
}

@article {MSTT,
    AUTHOR = {Matom\"aki, K. and Shao, X. and Tao, T. and
              Ter\"av\"ainen, J.},
     TITLE = {Higher uniformity of arithmetic functions in short intervals
              {I}. {A}ll intervals},
   JOURNAL = {Forum Math. Pi},
  FJOURNAL = {Forum of Mathematics. Pi},
    VOLUME = {11},
      YEAR = {2023},
     PAGES = {Paper No. e29, 97},
      ISSN = {2050-5086},
   MRCLASS = {11N37 (11B30)},
  MRNUMBER = {4658200},
MRREVIEWER = {Peter\ Shiu},
       DOI = {10.1017/fmp.2023.28},
       URL = {https://doi.org/10.1017/fmp.2023.28},
}

\end{document}